\documentclass[11pt,a4paper,reqno]{amsart}
\usepackage{a4wide}
\usepackage{amsaddr}
\usepackage{amsmath}
\usepackage{amsfonts, amssymb, amsthm}
\usepackage{paralist}
\usepackage[colorlinks=true]{hyperref}
\hypersetup{urlcolor=blue, citecolor=red}

\makeatletter
\renewcommand{\email}[2][]{%
	\ifx\emails\@empty\relax\else{\g@addto@macro\emails{,\space}}\fi%
	\@ifnotempty{#1}{\g@addto@macro\emails{\textrm{(#1)}\space}}%
	\g@addto@macro\emails{#2}%
}
\makeatother

\newtheorem{theorem}{Theorem}[section]

\newtheorem{lemma}[theorem]{Lemma}
\newtheorem{proposition}[theorem]{Proposition}
\theoremstyle{definition}

\numberwithin{equation}{section}

\usepackage[scrtime]{prelim2e}

\usepackage[normalem]{ulem}
\usepackage{color}
\newcommand{\R}{{\mathbb R}}

\begin{document}
\title[Nonlinear Choquard  equations with critical combined nonlinearities]
{Non-existence and multiplicity of positive solutions  for Choquard equations with  critical combined nonlinearities}
\author{Shiwang Ma}\email{shiwangm@nankai.edu.cn}
\address{School of Mathematical Sciences and LPMC, Nankai University\\ 
	Tianjin 300071, China}

%\thanks{*Corresponding author\\
%Email Addresses: shiwangm@nankai.edu.cn (SM)\\
%This work is supported by the National Natural Science Foundation of
%China ( No. 11571187.)}
%}



\date{}

\begin{abstract}
We study the non-existence and multiplicity of positive  solutions of the nonlinear  Choquard type equation
$$
-\Delta u+ \varepsilon u=(I_\alpha \ast |u|^{p})|u|^{p-2}u+ |u|^{q-2}u, 
\quad {\rm in} \ \mathbb R^N,
 \eqno(P_\varepsilon)
 $$
where $N\ge 3$ is an integer, $p\in (\frac{N+\alpha}{N}, \frac{N+\alpha}{N-2}]$, $q\in (2,\frac{2N}{N-2}]$, $I_\alpha$ is the Riesz potential of order $\alpha\in (0,N)$ and  $\varepsilon>0$ is  a  parameter.  We fix one of  $p,q$ as a critical exponent (in the sense of Hardy-Littlewood-Sobolev and Sobolev inequalities ) and  view the others in $p,q,\varepsilon, \alpha$ as parameters, we find regions in the $(p,q,\alpha, \varepsilon)$-parameter space, such that the corresponding equation has no positive ground state or admits multiple positive solutions. This is a counterpart of the Brezis-Nirenberg Conjecture (Brezis-Nirenberg, CPAM, 1983) for nonlocal elliptic equation in the whole space. Particularly, some  threshold results for the existence of ground states and some conditions which insure  two positive solutions are obtained.  These results are quite different in nature from the corresponding local equation with combined powers nonlinearity and reveal the  special influence of the nonlocal term. To the best of our knowledge, the only two papers concerning the multiplicity of positive solutions of elliptic equations with critical growth nonlinearity  are  given by Atkinson, Peletier (Nonlinear Anal, 1986) for elliptic equation on a ball and Juncheng Wei, Yuanze Wu (Proc. Royal Soc.  Edinburgh, 2022) for elliptic equation with a combined powers nonlinearity in the whole space. The ODE technique is main ingredient in the proofs of the above mentioned papers, however, ODE technique does not work any more in our model equation due to 
the presence of the nonlocal term.

\end{abstract}

\maketitle
%\tableofcontents

\newpage

\section{Introduction and main results}\label{s1}

We study standing--wave solutions of the nonlinear Schr\"odinger equation with attractive  combined  nonlinearity
\begin{equation}\label{e11}
%\begin{equation}\label{e11}
i\psi_t=\Delta \psi+(I_\alpha\ast |\psi|^{p})|\psi|^{p-2}\psi+ |\psi|^{q-2}\psi\quad\text{in $\R^N\times\R$,}
%\eqno(1.1)
\end{equation}
%\end{equation}
where $N\ge 3$ is an integer, $\psi: \mathbb R^N\times \mathbb R\to \mathbb C$,  $p\in (\frac{N+\alpha}{N}, \frac{N+\alpha}{N-2}]$, $q\in (2,2^*]$ with $2^*=\frac{2N}{N-2}$, and $I_\alpha$ is the Riesz potential of order $\alpha\in (0,N)$ and is defined for every $x\in \mathbb R^N\setminus \{0\}$ by 
$$
I_\alpha (x)=\frac{A_\alpha(N)}{|x|^{N-\alpha}}, \quad A_\alpha(N)=\frac{\Gamma(\frac{N-\alpha}{2})}{\Gamma(\frac{\alpha}{2})\pi^{N/2}2^\alpha},
$$
where $\Gamma$ denotes the Gamma function.

A theory of NLS with combined power nonlinearities was developed by Tao, Visan and Zhang \cite{Tao} and attracted a lot attention during the past decade (cf. \cite{Akahori-2,Akahori-3, Cazenave-1, Coles,   Li-2, Li-1, Li-3, Li-4, Moroz-1, Sun-1} and further references therein).

A standing--wave solutions of  \eqref{e11} with a frequency $\varepsilon >0$ is a finite energy solution in the form
$$
\psi(t,x)=e^{-i\varepsilon t}u(x).
$$
This ansatz yields   the equation for $u$ in the form
$$
-\Delta u+\varepsilon u=(I_\alpha \ast |u|^{p})|u|^{p-2}u+|u|^{q-2}u,
\quad {\rm in} \ \mathbb R^N.
 \eqno(P_\varepsilon)
$$

Let 
$$
v(x)=\varepsilon^{s}u(\varepsilon^{t}x).
$$
Then it is easy to see that if $s=-\frac{2+\alpha}{4(p-1)}, t=-\frac{1}{2}$, $(P_\varepsilon)$ is reduced to 
$$
-\Delta v+v=(I_\alpha\ast |v|^p)|v|^{p-2}v+\varepsilon^{\Lambda_1}|v|^{q-2}v,
$$
where 
$$
\Lambda_1=\frac{q(2+\alpha)-2(2p+\alpha)}{4(p-1)}=\left\{\begin{array}{rcl}
>0,   \quad {\rm  if}  \   \    q>\frac{2(2p+\alpha)}{2+\alpha},\\
=0,  \quad {\rm if}   \    \   q=\frac{2(2p+\alpha)}{2+\alpha},\\
<0,   \quad  {\rm if}  \    \  q<\frac{2(2p+\alpha)}{2+\alpha}.
\end{array}\right.
$$

If  $s=-\frac{1}{q-2}, t=-\frac{1}{2}$, $(P_\varepsilon)$ is reduced to 
$$
-\Delta v+v=\varepsilon^{\Lambda_2}(I_\alpha\ast |v|^p)|v|^{p-2}v+|v|^{q-2}v,
$$
where 
$$
\Lambda_2=\frac{2(2p+\alpha)-q(2+\alpha)}{2(q-2)}=\left\{\begin{array}{rcl}
<0,   \quad {\rm if}  \   \    q>\frac{2(2p+\alpha)}{2+\alpha},\\
=0,  \quad {\rm if}   \    \   q=\frac{2(2p+\alpha)}{2+\alpha},\\
>0,   \quad {\rm if}  \    \  q<\frac{2(2p+\alpha)}{2+\alpha}.
\end{array}\right.
$$

Motivated by this, in the present paper, we consider  the following  equations
$$
-\Delta v+v=(I_\alpha \ast |v|^{p})|v|^{p-2}v+\lambda |v|^{q-2}v, 
\quad {\rm in} \  \ \mathbb R^N,
 \eqno(Q_\lambda)
$$
and 
$$
-\Delta v+v=\mu(I_\alpha \ast |v|^{p})|v|^{p-2}v+|v|^{q-2}v, 
\quad {\rm in} \ \  \mathbb R^N,
 \eqno(Q_\mu)
$$
where $p\in (\frac{N+\alpha}{N}, \frac{N+\alpha}{N-2}]$, $q\in (2,2^*]$ and $\lambda, \mu>0$ are  parameters.

The corresponding  functionals are defined by 
\begin{equation}\label{e12}
I_\lambda(v):=\frac{1}{2}\int_{\mathbb R^N}|\nabla v|^2+|v|^2-\frac{1}{2p}\int_{\mathbb R^N}(I_\alpha\ast |v|^p)|v|^p-\frac{\lambda}{q}\int_{\mathbb R^N}|v|^q
%\eqno(1.2)
\end{equation}
and 
\begin{equation}\label{e13}
I_\mu(v):=\frac{1}{2}\int_{\mathbb R^N}|\nabla v|^2+|v|^2-\frac{\mu}{2p}\int_{\mathbb R^N}(I_\alpha\ast |v|^p)|v|^p-\frac{1}{q}\int_{\mathbb R^N}|v|^q,
%\eqno(1.3)
\end{equation}
respectively. The energy of the ground states given by
\begin{equation}\label{e14}
m_\lambda:=\inf_{v\in \mathcal M_\lambda}I_\lambda(v)
\quad {\rm and } \quad 
m_\mu:=\inf_{v\in \mathcal M_\mu}I_\mu(v)
%\eqno(1.4)
\end{equation}
are well-defined and positive, where $\mathcal M_\lambda$ and $\mathcal M_\mu$ denote the correspoding Nehari manifolds 
$$
\mathcal M_\lambda:=\left\{ v\in H^1(\mathbb R^N)\setminus\{0\}  \ \left | \ \int_{\mathbb R^N}|\nabla v|^2+|v|^2=\int_{\mathbb R^N}(I_\alpha\ast |v|^p)|v|^p+\lambda\int_{\mathbb R^N}|v|^q \right. \right\},
$$
$$
\mathcal M_\mu:=\left\{ v\in H^1(\mathbb R^N)\setminus\{0\}  \ \left | \ \int_{\mathbb R^N}|\nabla v|^2+|v|^2=\mu\int_{\mathbb R^N}(I_\alpha\ast |v|^p)|v|^p+\int_{\mathbb R^N}|v|^q \right. \right\}.
$$

\vskip 5mm 

The ground state solutions of $(Q_\lambda)$ and $(Q_\mu)$ are denoted by $v_\lambda$ and $v_\mu$ respectively. The existence of these kind solutions are proved in \cite{Li-2, Li-1}.  More precisely,  the existence and symmetry of ground state solutions of $(Q_\lambda)_{\lambda>0}$ and $(Q_\mu)_{\mu>0}$ with $\frac{N+\alpha}{N}<p<\frac{N+\alpha}{N-2}$ and $2<q<2^*$ was  established in \cite{Li-2}, while by using the results in \cite{Li-2} and  the subcritical approximation technique,  the following theorems are proved in \cite{Li-1}.

\noindent{\bf Theorem A.} {\it 
 Let $N\geq 3$,\ $\alpha\in(0,N)$, $p=\frac{N+\alpha}{N-2}$ and
$\lambda>0$. Then there is a constant $\lambda_1>0$ such that $(Q_\lambda)$ admits a positive groundstate
$v_\lambda\in H^1(\mathbb{R}^N)$ which is radially symmetric and radially
nonincreasing if one of the following conditions holds:

(1) $N\geq 4$ and $q\in (2,\frac{2N}{N-2})$;

(2) $N=3$ and  $q\in (4,\frac{2N}{N-2})$;

(3) $N=3$,  $q\in (2,4]$ and $\lambda>\lambda_1$.
}

\noindent{\bf Theorem B. } {\it 
Let $N\geq 3$, $\alpha\in(0,N)$, $q=2^*$ and $\mu>0$. Then there are two  constants $\mu_0, \mu_1>0$ such that $(Q_\mu)$
admits a positive groundstate  $v_\mu\in H^1(\mathbb{R}^N)$ which is
radially symmetric and radially nonincreasing if one of the
following conditions holds:

(1) $N\geq 4$ and $p\in
(1+\frac{\alpha}{N-2},\frac{N+\alpha}{N-2})$;

(2) $N\geq 4$, $p\in (\frac{N+\alpha}{N},1+\frac{\alpha}{N-2}]$ and
$\mu>\mu_0$;

(3) $N=3$ and $p\in (2+\alpha,\frac{N+\alpha}{N-2})$;
 
(4) $N=3$, $p\in (\frac{N+\alpha}{N},2+\alpha]$ and $\mu>\mu_1$.
}

\vskip 3mm

To the best of our knowledge, few has been achieved for  the uniqueness of radial solutions of $(Q_\lambda)$ and $(Q_\mu)$.  In a classical paper,  Kwong \cite{Kwong-1}
established the uniqueness of the radially symmetric solution of $(Q_\mu)_{\mu=0}$ in the case $2<q<2^*$. Later, this result was extended to more general equations  
\begin{equation}\label{e15}
-\Delta v+v=f(v), \quad {\rm in} \ \mathbb R^N% \eqno(1.5)
\end{equation}
by Serrin and Tang \cite{Serrin-1}. More precisely, it was proved in \cite{Serrin-1} that the radial
solution of  \eqref{e15} is unique if there exists a $b > 0$ such that $(f(v) - v)(v-b) > 0$ for $v \not = b$ and
the quotient $(f'(v)v-v)/(f(v)-v)$ is a non-increasing function of $v\in (b,\infty)$.
However, $f(v) = v^{p-1} + \lambda v^{q-1}$ with $2<q\le p<2^*$ does not satisfy the latter condition for large $v$, unless $p = q$. In fact, by using  Lyapunov–Schmidt type arguments,    D\'avila, del Pino and Guerra \cite{Davila-1} constructed three  positive solutions of  \eqref{e15} for  sufficiently large $\lambda>0$, $2<q<4$  and slightly subcritical $p < 2^*$.
 Therefore, the uniqueness of positive solutions of   \eqref{e15} with $f(v)=v^{p-1} + \lambda v^{q-1}$, the most natural extension of the single power case, has remained conspicuously open. On the other hand, Lieb \cite{Lieb-1} proved the uniqueness of positive solution of $(Q_\lambda)_{\lambda=0}$ in the case $N=3, p=2$ and $\alpha=2$.
Tao Wang and Taishan  Yi \cite{Wang-1} proved  the uniqueness of positive solution of  $(Q_\lambda)_{\lambda=0}$ in the case $N\in \{4, 5\},  p=2$ and $\alpha=2$. 
C.-L. Xiang \cite{Xiang-1} proved the  uniqueness and nondegeneracy of ground states for $(Q_\lambda)_{\lambda=0}$ in the case $N=3, p>2,\alpha=2$ and $p-2$ is sufficiently small.
More recently, Zexing Li \cite{Li-10}  show the uniqueness and non-degeneracy  of positive solutions
for the Choquard equation $(Q_\lambda)_{\lambda=0}$ when $N\in \{3, 4, 5\}, p \ge 2$ and $(\alpha, p)$ close to $(2, 2)$. J. Seok  \cite{Seok-1} proved  the uniqueness and nondegeneracy of ground states of $(Q_\lambda)_{\lambda=0}$ provided $p\in [2,\frac{N}{N-2})$ and $\alpha$ sufficiently close to 0, or $p\in (2,\frac{2N}{N-2})$ and $\alpha$ sufficiently close to $N$.  However,  so far nothing has been done concerning the uniqueness and nondegeneracy of positive solutions for $(Q_\lambda)_{\lambda\not=0}$ and $(Q_\mu)_{\mu\not=0}$, and other Choquard type equations involving a critical exponent.
For more aspects about Choquard type  equations, we refer the readers to the survey paper \cite{Moroz-2}.

\vskip 3mm

In the present paper, we are interested in the non-existence and multiplicity of positive solutions of $(Q_\lambda)$ and $(Q_\mu)$ with $p$ being  a critical exponent in the sense of Hardy-Littlewood-Sobolev inequality or $q$ being the critical exponent in the sense of Sobolev inequality. These questions   are  closely related to the well known Brezis-Nirenberg problem and  can be viewed as a  counterpart of the Brezis-Nirenberg conjecture proposed in \cite{Brezis-2}   for nonlocal elliptic equation in the whole space.  

In the classical paper \cite{Brezis-2}, Brezis and Nirenberg  investigate the following  model 
problem
\begin{equation}\label{e16}
\begin{cases}
&-\Delta u=u^{2^*-1}+\lambda u^{q-1},\quad u>0 \   \ on \  \Omega,\\
&u=0 \qquad on\  \  \partial \Omega,
\end{cases}
%\eqno(1.6)
\end{equation}
where $N\ge 3$, $\Omega\subset \mathbb R^N$ is a bounded set, $q\in [2,2^*)$, and $\lambda>0$ is a parameter.  Surprisingly, the cases where 
$N = 3$ and $N\ge 4$  turn out to be quite different.  For example, if  $N\ge 4$ and $q=2$, then problem  \eqref{e16} has a solution for every
$\lambda\in (0, \lambda_1)$, where $\lambda_1$ denotes the first eigenvalue of $-\Delta$ with Dirichlet boundary condition on $\Omega$; morover, it has
no solution if $\lambda\not\in (0, \lambda_1)$ and $\Omega$ is starshaped. If  $N = 3$ and $p = 2$, then problem  \eqref{e16} is much more delicate and
we have a complete solution only when $\Omega$ is a ball. In that case, problem
 \eqref{e16} has a solution if and only if $\lambda\in (\frac{1}{4}\lambda_1,\lambda_1)$.
In \cite{Brezis-2},  the authors conjectured that for $N=3$ and any $q\in (2,4)$, there exists a constant $\lambda_0>0$ such that   \eqref{e16} has no positive solution for $\lambda\in (0,\lambda_0)$, has one positive solution for $\lambda=\lambda_0$ and has two positive solutions for $\lambda>\lambda_0$.
In 1986, Atkinson and  Peletier \cite{Atkinson-1} solved the Brezis-Nirenberg conjecture by using ODE technique  in the case that $\Omega$ is a ball. More precisely,  by the  shooting argument, Atkinson and  Peletier investigate the Brezis and Nirenberg problem  \eqref{e16} and obtain at least one solution for any $N\ge 4$ and any $\lambda>0$. However, it turns out that the case $N=3$ is quite different from the cases $N\ge 4$ and  the following result is obtained in \cite{Atkinson-1}.

\noindent{\bf Theorem C}. {\it Suppose $N = 3$ and $2 < q < 6$. If $\Omega$ is a ball in $\mathbb R^3$,  then 

 (i)  if $4 < q < 6$, then for any $\lambda>0$, the problem  \eqref{e16} has a solution,

 (ii) if $q=4$, then there exists a constant $\lambda_0>0$ such that the problem  \eqref{e16} has no solution for $\lambda\in (0,\lambda_0]$ and has a solution for $\lambda>\lambda_0$,

(ii) if $2 <q < 4$, then for some $\lambda_0>0$,  the problem  \eqref{e16} has no solution for $\lambda\in (0,\lambda_0)$, has one solution for $\lambda=\lambda_0$, and has at least two solutions for $\lambda > \lambda_0$.}

In 1994, under the condition $1<q<2$, by studying  the decomposition of the Nehari manifold, 
 Ambrosetti,  Brezis and Cerami \cite{Ambrosetti-1} proved that  there exists a constant $\Lambda>0$ such that  the problem  \eqref{e16} admits at least two positive solutions for $\lambda\in  (0,\Lambda)$, one positive solution for $\lambda= \Lambda$ and no positive solution for $\lambda>\Lambda$. In this case, the multiplicity of  positive solutions is mainly caused by  the combined effects of concave and convex nonlinearities.
 
In the excellent papers \cite{Wei-1, Wei-2}, Wei and Wu  investigate the existence and asymptotic behaviors of normalized solutions to the following Schr\"odinger equation with combined powers nonlinearity
\begin{equation}\label{e17}
\begin{cases}
&-\Delta u+u=u^{2^*-1}+\lambda u^{q-1},\quad u>0 \   \ on \  \mathbb R^N,\\
&\lim_{|x|\to +\infty}u=0,
\end{cases}
%\eqno(1.7)
\end{equation}
where $N\ge 3$, $q\in (2,2^*)$, and $\lambda>0$ is a parameter. For $N=3$, the authors in  \cite{Wei-2} observed that the situation is also quite different from the higher dimension cases $N\ge 4$.  In particular, when $N=3$ and $q\in (2,4]$, Wei and Wu \cite{Wei-2} obtain  a threshold number $\lambda_q>0$ such that for any $\lambda\in (0, \lambda_q)$,  \eqref{e17} has no ground state solution, for $\lambda\ge \lambda_q$,  \eqref{e16} has a ground state solution ( possibly $\lambda>\lambda_q$ for $q=4$ ). Besides, for any $q\in (2,4)$, there exists a lager number $\hat\lambda_q>\lambda_q$ such that  \eqref{e17} admits   two positive solutions for $\lambda\ge \hat\lambda_q$.   
In the case that $N=3$ and $q\in (2,4]$, the non-existence of ground state solution for  \eqref{e17} with small parameter $\lambda>0$  has also been proved  by Akahori at al. \cite{Akahori-4}.  
The asymptotic profile of ground state of  \eqref{e17} has been investigated in \cite{Ma-1}.

It is worth mentioning  that ODE technique 
plays a key role in the proof of  the above mentioned papers  \cite{ Akahori-4,Wei-1, Wei-2}. In particular, Wei and Wu \cite{Wei-1} study the best decay estimates of rescaled solutions by using ODE technique, which is main component in proving the multiplicity of positive solutions.
However, in our setting,  
%due to the appearance  of the nonlocal term (
due to  the presence of the nonlocal term, ODE technique does not work any more.  To overcome this difficulties, we adopt the Moser iteration \cite{Akahori-2, GT,LiuXQ}  and a similar technique as used in \cite{Ma-2} to deduce the best uniform decay estimation of some suitable rescaled solutions.  Besides, we mention that  the authors in \cite{Wei-2} used the $L^\infty$ norms to distinguish different solutions. However,  in the present  paper, we use the
energy of the solutions to distinguish different solutions.

It is well known that the domain topology also affect the number of the positive
solutions for various elliptic problems \cite{Benci-1, GP, MP,  Rey-1}.  In \cite{Benci-1, Rey-1}, the number of the solutions
is estimated by using the category of the domain. In \cite{GP, MP}, the number of
the solutions is linked to the number of the holes in the domain. 
\vskip 3mm

In this paper, we study the effect of the nonlocal  term  on the number of positive solutions  and establish  the non-existence of  ground state solution and multiplicity of positive solutions which are mainly caused by the appearance  of the nonlocal term.  It turns out that  
both the nonlocal term and the dimension of the space affect the number of the positive solutions. A  rather striking fact is that $(Q_\mu)$ have  two positive solutions  for all $N\ge 3$ and large $\mu>0$ if $q=2^*$ and $p\in (\frac{N+\alpha}{N}, 1+\frac{\alpha}{N-2})$. Our main results  are as follows:

\vskip 3mm

%\noindent{\bf Theorem 1.1.} {\it If  $N=3,  p=3+\alpha$ and $q\in (2,4]$,  then $m_\lambda$ is non-increasing 

%with respect to $\lambda$ 
\begin{theorem} \label{t11}
 If $N=3,  p=3+\alpha$ and $q\in (2,4]$,  then $m_\lambda$ is non-increasing 
and 
 there exists  $\lambda_q>0$ such that for $\lambda\in (0,\lambda_q)$, $m_\lambda=\frac{2+\alpha}{2(N+\alpha)}S_\alpha^{\frac{N+\alpha}{2+\alpha}}$ and there is no ground state solution,  and for $\lambda>\lambda_q$,  $m_\lambda<\frac{2+\alpha}{2(N+\alpha)}S_\alpha^{\frac{N+\alpha}{2+\alpha}}$ and  $(Q_\lambda)$ admits  a ground state solution. Furthermore, if $q\in (2,4)$, then
 there exists $\tilde \lambda_q> \lambda_q$ such that $(Q_\lambda)$ admits two positive solution for $\lambda\ge \tilde \lambda_q$, 
 where 
\begin{equation}\label{e18}
S_\alpha:=\inf_{v\in D^{1,2}(\mathbb R^N)\setminus\{0\}}\frac{\int_{\mathbb R^N}|\nabla v|^2}{\left(\int_{\mathbb R^N}(I_\alpha\ast |v|^{\frac{N+\alpha}{N-2}})|v|^{\frac{N+\alpha}{N-2}}\right)^{\frac{N-2}{N+\alpha}}}.
%\eqno(1.8)
\end{equation}
%\vskip 5mm 
\end{theorem}

\begin{theorem} \label{t12}
If $q=2^*$ and 
$$
p\in  \left\{ \begin{array}{rcl}
(\frac{N+\alpha}{N}, 1+\frac{\alpha}{N-2}], \qquad \qquad \qquad \qquad  \quad   &if& N\ge 4,\\
(\frac{3+\alpha}{3}, 1+\alpha]\cup (\max\{2,1+\alpha\}, 2+\alpha],  &if&  N=3,\end{array}\right.
$$
 then $m_\mu$ is non-increasing 
%with respect to $\mu$ 
and 
 there exists  $\mu_p>0$ such that for $\mu\in (0,\mu_p)$, $m_\mu=\frac{1}{N}S^{\frac{N}{2}}$ and there is no ground state solution,  and for $\mu>\mu_p$,  $m_\mu<\frac{1}{N}S^{\frac{N}{2}}$ and  $(Q_\mu)$ admits  a ground state solution. Furthermore, for  any  $p\in (\frac{N+\alpha}{N}, 1+\frac{\alpha}{N-2})$ in the cases $N\ge 4$, 
 and  $p\in (\frac{3+\alpha}{3}, 1+\alpha)\cup (\max\{2,1+\alpha\}, 2+\alpha)$  in the case  $N=3$,
% (or   $p\in (\frac{3+\alpha}{3},\max\{1+\alpha, 1+\frac{2+\alpha}{3}\})$ if $N=3$,
  there exists $\tilde \mu_p> \mu_p$ such that $(Q_\mu)$ admits two positive solution for $\mu\ge \tilde \mu_p$, 
 where 
\begin{equation}\label{e19}
S:=\inf_{w\in D^{1,2}(\mathbb R^N)\setminus \{0\}}\frac{\int_{\mathbb R^N}|\nabla w|^2}{\left(\int_{\mathbb R^N}|w|^{2^*}\right)^{\frac{2}{2^*}}}.
%\eqno(1.9)
\end{equation}
\end{theorem}
\vskip 3mm 

\noindent {\bf Remark 1.1.}  The non-existence of ground state solution  and multiplicity of positive solutions  of  $(Q_\mu)$  with  $p\in (\frac{N+\alpha}{N}, 1+\frac{\alpha}{N-2})$  in Theorem \ref{t12} are  mainly caused by the presence of the nonlocal term. These  are new phenomena and reveal the special influence of the nonlocal term. The non-existence of ground state solution and  multiplicity of positive solutions of $(Q_\lambda)$ in Theorem \ref{t11}, and  $(Q_\mu)$ with $N=3$ and $p\in (\max\{2, 1+\alpha\}, 2+\alpha)$ in Theorem \ref{t12} are mainly caused by the dimension of the space, and a similar phenomenon was observed in the local equation  \eqref{e17} with $N=3$ and $q\in (2,4)$ by Wei and Wu \cite{Wei-2}. The non-existence of ground state solutions of  \eqref{e17} with $N=3$ and $q\in (2,4)$ also be found in \cite{Akahori-4} for sufficiently small $\lambda>0$. In the case $N=3$ and
 % $p\in (1+\frac{2+\alpha}{3}, 2]$ 
 $p\in [1+\alpha, \max\{2, 1+\alpha\}]$, $\alpha\in (0,3)$, the multiplicity of positive solutions for $(Q_\mu)$ is still  an open question. In the case $N=3$,  I guess that $p\not\in (1+\alpha,\max\{2,1+\alpha\}]$ in Theorem \ref{t12} is a technical  condition, but since the technique used in \cite[Proposition 4.13]{Ma-3} to establish the best decay estimate for some suitable rescaled solutions  is only valid in this case $ 2<p<3+\alpha,$
 whether or not this condition can be removed  seems to be a challenging problem.  Besides, in this case, I do not know whether or not the condition $p\not=1+\alpha$ in Theorem \ref{t12} is essential for our problem. 
\vskip 5mm

\noindent\textbf{Organization of the paper}. In Section 2, we give  some preliminary results which are needed in the proof of our main results. Sections 3--4 are devoted to the proofs of Theorems \ref{t11}--\ref{t12}, respectively.  Finally, in the last section, we present some optimal decay estimates for positive solutions of some elliptic equations with parameters, which are crucial in our paper and should be useful in other related context. 
%which are core in our paper and should be useful in other related context. 

\smallskip

\noindent\textbf{Basic notations}. Throughout this paper, we assume $N\geq
3$. $ C_c^{\infty}(\mathbb{R}^N)$ denotes the space of the functions
infinitely differentiable with compact support in $\mathbb{R}^N$.
$L^p(\mathbb{R}^N)$ with $1\leq p<\infty$ denotes the Lebesgue space
with the norms
$\|u\|_p=\left(\int_{\mathbb{R}^N}|u|^p\right)^{1/p}$.
 $ H^1(\mathbb{R}^N)$ is the usual Sobolev space with norm
$\|u\|_{H^1(\mathbb{R}^N)}=\left(\int_{\mathbb{R}^N}|\nabla
u|^2+|u|^2\right)^{1/2}$. $ D^{1,2}(\mathbb{R}^N)=\{u\in
L^{2^*}(\mathbb{R}^N): |\nabla u|\in
L^2(\mathbb{R}^N)\}$. $H_r^1(\mathbb{R}^N)=\{u\in H^1(\mathbb{R}^N):
u\  \mathrm{is\ radially \ symmetric}\}$.  $B_r$ denotes the ball in $\mathbb R^N$ with radius $r>0$ and centered at the origin,  $|B_r|$ and $B_r^c$ denote its Lebesgue measure  and its complement in $\mathbb R^N$, respectively.  As usual, $C$, $c$, etc., denote generic positive constants.
 For any  small $\epsilon>0$ and two nonnegative functions $f(\epsilon)$ and  $g(\epsilon)$, we write

(1)  $f(\epsilon)\lesssim g(\epsilon)$ or $g(\epsilon)\gtrsim f(\epsilon)$ if there exists a positive constant $C$ independedent of $\epsilon$ such that $f(\epsilon)\le Cg(\epsilon)$.

(2) $f(\epsilon)\sim g(\epsilon)$ if $f(\epsilon)\lesssim g(\epsilon)$ and $f(\epsilon)\gtrsim g(\epsilon)$.

If $|f(\epsilon)|\lesssim |g(\epsilon)|$, we write $f(\epsilon)=O((g(\epsilon))$.  Finally, if $\lim f(\epsilon)/g(\epsilon)=1$ as $\epsilon\to \epsilon_0$, then we write $f(\epsilon)\simeq g(\epsilon)$ as $\epsilon\to \epsilon_0$.

\section{ Preliminaries } \label{s2}
In this section, we present some preliminary results which are needed in the proof of our main results. Firstly,  we consider the following Choquard type  equation with combined nonlinearites:
$$
-\Delta u+u=\mu(I_{\alpha}\ast |u|^p)|u|^p+\lambda |u|^{q-2}u,  \quad \mathrm{in}\ \
\mathbb{R}^N,
\eqno(Q_{\mu,\lambda})
$$
where $N\geq 3,\ \alpha\in(0,N)$, $p\in (\frac{N+\alpha}{N}, \frac{N+\alpha}{N-2}], \ q\in (2, 2^*]$,
 $\mu>0$ and $\lambda>0$ are two parameters. 

It has been proved in \cite{Li-1} that any weak solution of $(Q_{\mu,\lambda})$ in
$H^1(\mathbb{R}^N)$ has additional regularity properties, which
allows us to establish the Poho\v{z}aev identity for all finite
energy solutions.

\begin{lemma}\label{l21}
If
$u\in H^1(\mathbb{R}^N)$ is a solution of $(Q_{\mu,\lambda})$, then $u\in
W_{\mathrm{loc}}^{2,r}(\mathbb{R}^N)$ for every $r>1$. Moreover, $u$
satisfies the Poho\v{z}aev identity
$$
P_{\mu, \lambda}(u):=\frac{N-2}{2}\int_{\mathbb{R}^N}|\nabla
u|^2+\frac{N}{2}\int_{\mathbb{R}^N}|u|^2-\frac{N+\alpha}{2p}\mu\int_{\mathbb{R}^N}(I_\alpha\ast
|u|^p)|u|^p-\frac{N}{q}\lambda\int_{\R^N}|u|^q=0.
$$
\end{lemma}

It is well known that any weak solution of  $(Q_{\mu,\lambda})$ corresponds to a critical point of the action  functionals $I_{\mu,\lambda}$ defined by 
\begin{equation}\label{e21}
I_{\mu,\lambda}(u):=\frac{1}{2}\int_{\mathbb R^N}|\nabla u|^2+|u|^2-\frac{\mu}{2p}\int_{\mathbb R^N}(I_\alpha\ast |u|^p)|u|^p-\frac{\lambda}{q}\int_{\mathbb R^N}|u|^q, 
%\eqno(2.1)
\end{equation}
which is well defined and is of $C^1$ in $H^1(\R^N)$.  A nontrivial solution $u_{\mu,\lambda}\in H^1(\R^N)$  is called a ground-state if 
\begin{equation}\label{e22}
I_{\mu,\lambda}(u_{\mu,\lambda})=m_{\mu, \lambda}:=\inf\{ I_{\mu,\lambda}(u): \ u\in H^1(\R^N)\setminus\{0\} \ {\rm and}  \  I'_{\mu,\lambda}(u)=0\}.
%\eqno(2.2)
\end{equation}

In \cite{Li-2, Li-1} (see also the proof of the main results in \cite{Li-1}),  it has been shown that 
\begin{equation}\label{e23}
m_{\mu, \lambda}=\inf_{u\in \mathcal M_{\mu, \lambda}}I_{\mu, \lambda}(u)=\inf_{u\in \mathcal P_{\mu, \lambda}}I_{\mu, \lambda}(u),
%\eqno(2.3)
\end{equation}
where $\mathcal M_{\mu, \lambda}$ and $\mathcal P_{\mu, \lambda}$ are the correspoding Nehari  and Poho\v{z}aev manifolds defined by
$$
\mathcal M_{\mu, \lambda}:=\left\{ u\in H^1(\mathbb R^N)\setminus\{0\}  \ \left | \ \int_{\mathbb R^N}|\nabla u|^2+|u|^2=\mu \int_{\mathbb R^N}(I_\alpha\ast |u|^p)|u|^p+\lambda\int_{\mathbb R^N}|u|^q \right. \right\}
$$
and 
$$
\mathcal P_{\mu, \lambda}:=\left\{ u\in H^1(\mathbb R^N)\setminus\{0\}  \ \left | \ P_{\mu,\lambda}(u)=0 \right. \right\},
$$
respectively.
Moreover, the following min-max descriptions are valid:

\begin{lemma}\label{l22}  Let
 $$
u_t(x)=\left\{\begin{array}{rcl} u(\frac{x}{t}), \quad if \  t>0,\\
0, \quad \  if \  \ t=0,
\end{array}\right.
$$
then 
\begin{equation}\label{e24}
m_{\mu, \lambda}=\inf_{u\in H^1(\mathbb R^N)\setminus\{0\}}\sup_{t\ge 0}I_{\mu, \lambda}(tu)=\inf_{u\in H^1(\mathbb R^N)\setminus\{0\}}\sup_{t\ge 0}I_{\mu, \lambda}(u_t).
%\eqno(2.4)
\end{equation}
In particular, if $u_{\mu,\lambda}$ is a ground state for $(Q_{\mu, \lambda})$, then $m_{\mu, \lambda}=I_{\mu,\lambda}(u_{\mu, \lambda})=\sup_{t>0}I_{\mu, \lambda}(tu_{\mu,\lambda})=\sup_{t>0}I_{\mu, \lambda}((u_{\mu,\lambda})_t)$.
\end{lemma}

\vskip 3mm

Let $u_{\mu,\lambda}$ be the ground state for $(Q_{\mu, \lambda})$, then the following lemma  is proved in \cite{Ma-2}.

\begin{lemma}\label{l23}  The solution sequences  $\{u_{1,\lambda}\}$ and $\{u_{\mu,1}\}$ are  bounded in $H^1(\mathbb R^N)$.
\end{lemma}
 
  \vskip 3mm

 The following Gagliardo-Nirenberg inequality can be found in \cite{Weinstein-1}.
 
\begin{lemma}\label{l24}  Let $N\ge 1$ and $2<q<2^*$,  then the following sharp Galiardo-Nirenberg inequality
\begin{equation}\label{e25}
\|u\|_p \le  C_{Nq}\|\nabla u\|_2^{\frac{N(q-2)}{2q}}\|u\|_2^{\frac{2N-q(N-2)}{2q}}
%\eqno(2.5)
\end{equation}
holds for any $u \in H^1(\mathbb R^N)$,  where the sharp constant $C_{Nq}$  is
$$
C_{Nq}^q=\frac{2q}{2N-q(N-2)}\left(\frac{2N-q(N-2)}{N(q-2)}\right)^{\frac{N(q-2)}{4}}\frac{1}{\|Q_q\|_2^{q-2}}
$$
and $Q_q$ is the unique positive radial solution of equation
$$
-\Delta Q + Q = |Q|^{q-2}Q.
$$
\end{lemma}

The following well
known Hardy-Littlewood-Sobolev inequality can be found in
\cite{Lieb-Loss 2001}. 

\begin{lemma}\label{l25}
Let $p, r>1$ and $0<\alpha<N$ with $1/p+(N-\alpha)/N+1/r=2$. Let
$u\in L^p(\mathbb{R}^N)$ and $v\in L^r(\mathbb{R}^N)$. Then there
exists a sharp constant $C(N,\alpha,p)$, independent of $u$ and $v$,
such that
\begin{equation}\label{e26}
\left|\int_{\mathbb{R}^N}\int_{\mathbb{R}^N}\frac{u(x)v(y)}{|x-y|^{N-\alpha}}\right|\leq
C(N,\alpha,p)\|u\|_p\|v\|_r.
\end{equation}
If $p=r=\frac{2N}{N+\alpha}$, then
$$
C(N,\alpha,p)=C_\alpha(N)=\pi^{\frac{N-\alpha}{2}}\frac{\Gamma(\frac{\alpha}{2})}{\Gamma(\frac{N+\alpha}{2})}\left\{\frac{\Gamma(\frac{N}{2})}{\Gamma(N)}\right\}^{-\frac{\alpha}{N}}.
$$
\end{lemma}

\noindent{\bf Remark 2.1. }  By the Hardy-Littlewood-Sobolev inequality, for any $v\in L^s(\mathbb R^N)$ with $s\in (1,\frac{N}{\alpha})$, $I_\alpha\ast v\in L^{\frac{Ns}{N-\alpha s}}(\mathbb R^N)$ and 
\begin{equation}\label{e27}
\|I_\alpha\ast v\|_{\frac{Ns}{N-\alpha s}}\le A_\alpha(N)C(N,\alpha, s)\|v\|_{s}.
%\eqno(2.7)
\end{equation}

\begin{lemma} {\rm ( P. L. Lions  \cite{Lions-1} )} \label{l26}
 Let $r>0$ and $2\leq q\leq 2^{*}$. If $(u_{n})$ is bounded in $H^{1}(\mathbb{R}^N)$ and if
$$\sup_{y\in\mathbb{R}^N}\int_{B_{r}(y)}|u_{n}|^{q}\to0,\,\,\textrm{as\ }n\to\infty,$$
then $u_{n}\to0$ in $L^{s}(\mathbb{R}^N)$ for $2<s<2^*$. Moreover, if $q=2^*$, then $u_{n}\to0$ in $L^{2^{*}}(\mathbb{R}^N)$.
\end{lemma}
 
 The following lemma is proved in \cite{Ma-2}.
 
\begin{lemma}\label{l27}   Let $r>0$, $N\ge 3$, $\alpha\in (0,N)$ and $\frac{N+\alpha}{N}\le p\le \frac{N+\alpha}{N-2}$. If  $(u_n)$ be bounded in $H^1(\mathbb R^N)$ and if 
$$
\lim_{n\to \infty}\sup_{z\in \mathbb R^N} \int_{B_r(z)}\int_{B_r(z)}\frac{|u_n(x)|^p|u_n(y)|^p}{|x-y|^{N-\alpha}} dx dy = 0, 
$$
then 
$$
\lim_{n\to \infty}\int_{\mathbb R^N}|u_n|^s dx = \lim_{n\to\infty} \int_{\mathbb R^N}(I_\alpha\ast |u_n|^t)|u_n|^t dx = 0,
$$
for any $2<s<2^*$ and $\frac{N+\alpha}{N}<t<\frac{N+\alpha}{N-2}$. Moreover, if $p=\frac{N+\alpha}{N-2}$, then 
$$
\lim_{n\to \infty}\int_{\mathbb R^N}|u_n|^{2^*} dx = \lim_{n\to\infty}\int_{\mathbb R^N} (I_\alpha\ast |u_n|^{\frac{N+\alpha}{N-2}})|u_n|^{\frac{N+\alpha}{N-2}} dx = 0.
$$
\end{lemma}

  The following Moser iteration lemma  is given in \cite[Proposition B.1]{Akahori-2}. See also  \cite{LiuXQ}  and \cite{GT}.
%\smallskip

\begin{lemma}\label{l28}  Assume $N\ge 3$. Let $a(x)$ and $b(x)$ be functions on $B_4$, and let $u\in H^1(B_4)$ be a weak solution to 
\begin{equation}\label{e28}
-\Delta u+a(x)u=b(x)u \qquad  in \  \ B_4.
%\eqno(2.8)
\end{equation}
Suppose that $a(x)$ and $u$ satisfy that 
\begin{equation}\label{e29}
a(x)\ge 0 \quad for \ a. e. \ x\in B_4, 
%\eqno(2.9)
\end{equation}
and 
\begin{equation}\label{e210}
 \int_{B_4}a(x)|u(x)v(x)|dx<\infty \quad for \ each \ v\in H_0^1(B_4).
%\eqno(2.10)
\end{equation}
(i) Assume that for any $\varepsilon\in (0,1)$, there exists $t_\varepsilon>0$ such  that
$$
\|\chi_{[|b|>t_\varepsilon]}b\|_{L^{N/2}(B_4)}\le \varepsilon,
$$
where $[|b|>t]:=\{x\in B_4: \ |b(x)|>t\},$ and $\chi_A(x)$ denotes the characteristic function of $A\subset \mathbb R^N$. Then for any $r\in (0,\infty)$, there exists a constant $C(N,r,t_\varepsilon)$ such that 
$$
\||u|^{r+1}\|_{H^1(B_1)}\le C(N, r,t_\varepsilon)\|u\|_{L^{2^*}(B_4)}.
$$
(ii) \ Let $s>N/2$ and assume that $b\in L^s(B_4)$. Then there exits a constant $C(N,s,\|b\|_{L^s(B_4)})$ such that
$$
\|u\|_{L^\infty(B_1)}\le C(N,s,\|b\|_{L^s(B_4)})\|u\|_{L^{2^*}(B_4)}.
$$
Here, the constants $C(N,r, t_\varepsilon)$ and $C(N,s,\|b\|_{L^s(B_4)})$ in (i) and (ii) remain bounded as long as $r, t_\varepsilon$ and $\|b\|_{L^s(B_4)}$ are bounded.
\end{lemma}

  \vskip 3mm 
    
 Consider the case $\lambda\to \infty$ in $(Q_\lambda)$ and $\mu\to \infty$ in $(Q_\mu)$. The following two results are proved in \cite{Ma-2}.

\begin{theorem}\label{t29}   Let $N\ge 3$, $p\in (\frac{N+\alpha}{N}, \frac{N+\alpha}{N-2}]$, $q\in (2,2^*)$ and  $ v_\lambda$ be  the ground state of $(Q_\lambda)$, then as $\lambda\to \infty$, 
the rescaled family of ground states $\tilde w_\lambda=\lambda^{\frac{1}{q-2}}v_\lambda$ converges in $H^1(\mathbb R^N)$ to the unique positive solution $w_0\in H^1(\mathbb R^N)$ of the equation
$$
-\Delta w+w=w^{q-1}.
$$
Moreover,  as $\lambda\to \infty$, there holds
$$
\|v_\lambda\|_2^2=
\lambda^{-\frac{2}{q-2}}\left(\frac{2N-q(N-2)}{2q}S_q^{\frac{q}{q-2}}+O(\lambda^{-\frac{2(p-1)}{q-2}})\right),
$$
$$
\|\nabla v_\lambda\|_2^2=\lambda^{-\frac{2}{q-2}}\left(
\frac{N(q-2)}{2q}S_q^{\frac{q}{q-2}}+O(\lambda^{-\frac{2(p-1)}{q-2}})\right),
$$
and the least energy $m_\lambda$ of the ground state satisfies 
$$
\frac{q-2}{2q}S_q^{\frac{q}{q-2}}-\lambda^{\frac{2}{q-2}}m_\lambda \sim\lambda^{-\frac{2(p-1)}{q-2}},
$$
as $\lambda\to \infty$, where $S_q$ is given by
\begin{equation}\label{e211}
S_q=\inf_{v\in H^1(\mathbb R^N)\setminus \{0\}}\frac{\int_{\mathbb R^N}|\nabla v|^2+|v|^2}{(\int_{\mathbb R^N}|v|^q)^{\frac{2}{q}}}.
%\eqno(2.11)
\end{equation}
\end{theorem}
 
\begin{theorem}\label{t210}  Let $N\ge 3$, $p\in (\frac{N+\alpha}{N}, \frac{N+\alpha}{N-2})$, $q\in (2,2^*]$ and $v_\mu$ be  the ground state of $(Q_\mu)$, then as $\mu\to \infty$, 
the rescaled family of ground states $\tilde w_\mu=\mu^{\frac{1}{2(p-1)}}v_\mu$ converges up to a subsequence  in $H^1(\mathbb R^N)$ to a positive solution $w_0\in H^1(\mathbb R^N)$ of the equation
$$
-\Delta v+v=(I_\alpha \ast |v|^p)v^{p-1}.
$$
Moreover, as $\mu\to \infty$, there holds
$$
\|v_\mu\|_2^2=
\mu^{-\frac{1}{p-1}}\left(\frac{N+\alpha-p(N-2)}{2p}S_p^{\frac{p}{p-1}}+O(\mu^{-\frac{q-2}{2(p-1)}})\right), 
$$
$$
\|\nabla v_\mu\|_2^2=\mu^{-\frac{1}{p-1}}\left(\frac{N(p-1)-\alpha}{2p}S_p^{\frac{p}{p-1}}+O(\mu^{-\frac{q-2}{2(p-1)}})\right), 
$$
and  the least energy $m_\mu$ of the ground state satisfies 
$$
\frac{p-1}{2p}S_p^{\frac{p}{p-1}}-\mu^{\frac{1}{p-1}}m_\mu\sim \mu^{-\frac{q-2}{2(p-1)}},
$$
as $\mu\to \infty$, where $S_p$ is given by
\begin{equation}\label{e212}
S_p=\inf_{v\in H^1(\mathbb R^N)\setminus \{0\}}\frac{\int_{\mathbb R^N}|\nabla v|^2+|v|^2}{(\int_{\mathbb R^N}(I_\alpha\ast |v|^p)|v|^p)^{\frac{1}{p}}}.
%\eqno(2.12)
\end{equation}
\end{theorem}

 \vskip 3mm

\section{Proof of Theorem 1.1}\label{s4}

 In this section, we always assume that $2_\alpha^*:=\frac{N+\alpha}{N-2}$ and  $q\in (2, 2^*)$.
 
 \subsection{Non-existence}

  It is easy to see that under the rescaling 
\begin{equation}\label{e31}
w(x)=\lambda^{\frac{1}{q-2}}v(\lambda^{\frac{2^*-2}{2(q-2)}}x), 
%\eqno(3.1)
\end{equation}
the equation $(Q_\lambda)$ is reduced to 
$$
-\Delta w+ \lambda^\sigma w=(I_\alpha\ast |w|^{2_\alpha^*})|w|^{2_\alpha^*-2}w+\lambda^\sigma |w|^{q-2}w, 
\eqno(\bar Q_\lambda)
$$
where $\sigma:=\frac{2^*-2}{q-2}>1.$ 

The associated functional is defined by 
$$
J_\lambda(w):=\frac{1}{2}\int_{\mathbb R^N}|\nabla w|^2+\lambda^\sigma|w|^2-\frac{1}{22_\alpha^*}\int_{\mathbb R^N}(I_\alpha\ast |w|^{2_\alpha^*})|w|^{2_\alpha^*}-\frac{1}{q}\lambda^\sigma\int_{\mathbb R^N}|w|^q.
$$

We define the Nehari manifolds as follows.
$$
\mathcal{N}_\lambda=
\left\{w\in H^1(\mathbb R^N)\setminus\{0\} \ \left | \ \int_{\mathbb R^N}|\nabla w|^2+\lambda^\sigma\int_{\mathbb R^N}|w|^2=\int_{\mathbb R^N}(I_\alpha\ast |w|^{2_\alpha^*})|w|^{2_\alpha^*}+\lambda^\sigma
\int_{\mathbb R^N}|w|^q\  \right. \right\}
$$
and 
$$
\mathcal{N}_0=
\left\{w\in H^1(\mathbb R^N)\setminus\{0\} \ \left | \ \int_{\mathbb R^N}|\nabla w|^2=\int_{\mathbb R^N}(I_\alpha\ast |w|^{2_\alpha^*})|w|^{2_\alpha^*}\  \right. \right\}. 
$$
Then 
$$
m_\lambda:=\inf_{w\in \mathcal {N}_\lambda}J_\lambda(w), \quad {\rm and} \quad 
m_0:=\inf_{w\in \mathcal {N}_0}J_0(w)
$$
are well-defined and positive. Moreover, $J_0$ is attained on $\mathcal N_0$.

Define the Poho\v zaev manifold as follows
$$
\mathcal P_\lambda:=\{w\in H^1(\mathbb R^N)\setminus\{0\} \   | \  P_\lambda(w)=0 \},
$$
where 
$$
\begin{array}{rcl}
P_\lambda(w):&=&\frac{N-2}{2}\int_{\mathbb R^N}|\nabla w|^2+\frac{ \lambda^\sigma N}{2}\int_{\mathbb R^N}|w|^2\\ \\
&\quad &-\frac{N+\alpha}{2p}\int_{\mathbb R^N}(I_\alpha\ast |w|^{2_\alpha^*})|w|^{2_\alpha^*}-\frac{\lambda^\sigma N}{q}\int_{\mathbb R^N}|w|^q.
\end{array}
$$
Then by Lemma \ref{l21}, $w_\lambda\in \mathcal P_\lambda$. Moreover,  we have a similar minimax characterizations for the least energy level $m_\lambda$ as in Lemma \ref{l22}.

The following lemma is proved in \cite{Ma-2}.

\begin{lemma} \label{l31}  The rescaled family of solutions $\{w_\lambda\}$ is bounded in $H^1(\mathbb R^N)$.
\end{lemma}

\vskip 3mm

In the lower dimension cases $N=4$ and $N=3$, it has also been shown in \cite{Ma-2} that there exists $\xi_\lambda>0$ satisfying $\xi_\lambda\to 0$ such that  the scaled family $\tilde w_\lambda(x)=\xi_\lambda^{\frac{N-2}{2}} w_\lambda(\xi_\lambda x)$ satisfies 
\begin{equation}\label{e32}
\|\nabla(\tilde w_\lambda-W_1)\|_2\to 0, \qquad \|\tilde w_\lambda-W_1\|_{2^*}\to 0,   \qquad {\rm as} \  \ \lambda\to 0,
%\eqno(3.2)
\end{equation}
where and in what follows
$$
W_1(x)=[N(N-2)]^{\frac{N-2}{4}}\left(\frac{1}{1+|x|^2}\right)^{\frac{N-2}{2}}.
$$
Under the rescaling 
\begin{equation}\label{e33}
\tilde w(x)=\xi_\lambda^{\frac{N-2}{2}} w(\xi_\lambda x), 
%\eqno(3.3)
\end{equation}
the equation $(\bar Q_\lambda)$ reads as  
$$
-\Delta \tilde w+\lambda^\sigma\xi_\lambda^{2}\tilde w=(I_\alpha\ast |\tilde w|^{2_\alpha^*})|\tilde w|^{2_\alpha^*-2}\tilde w+\lambda^\sigma\xi_\lambda^{N-\frac{N-2}{2}q}|\tilde w|^{q-2}\tilde w.
\eqno(R_\lambda)
$$
The  corresponding energy functional is given by
$$
\tilde J_\lambda(\tilde w):=\frac{1}{2}\int_{\mathbb R^N}|\nabla\tilde  w|^2+\lambda^\sigma\xi_\lambda^2|\tilde w|^2-\frac{1}{2{2_\alpha^*}}\int_{\mathbb R^N}(I_\alpha\ast |\tilde w|^{{2_\alpha^*}})|\tilde w|^{2_\alpha^*}-\frac{1}{q}\lambda^\sigma\xi_\lambda^{N-\frac{N-2}{2}q}\int_{\mathbb R^N}|\tilde w|^q.
$$
Clearly, we have $\tilde J_\lambda(\tilde w)=J_\lambda(w)=I_\lambda(v)$ and  the following lemma holds true \cite{Ma-2}.

\begin{lemma}\label{l32}    Let $v,w,\tilde w\in H^1(\mathbb R^N)$ satisfy  \eqref{e31}  and  \eqref{e33}, then the following statements hold true

(1) $ \ \|\nabla \tilde w\|_2^2= \|\nabla w\|_{2}^{2}=\|\nabla v\|_{2}^{2}, \  \int_{\mathbb R^N}(I_\alpha\ast |\tilde w|^{2_\alpha^*})|\tilde w|^{2_\alpha^*}=\int_{\mathbb R^N}(I_\alpha\ast |w|^{2_\alpha^*})|w|^{2_\alpha^*}=\int_{\mathbb R^N}(I_\alpha\ast |v|^{2_\alpha^*})|v|^{2_\alpha^*},$

(2)  $\xi_\lambda^{2}\|\tilde w\|^2_2=\|w\|_2^2=\lambda^{-\sigma}\| v\|_2^2, \   \   \xi_\lambda^{N-\frac{N-2}{2}q}\|\tilde w\|^q_q=\|w\|_q^q=\lambda^{1-\sigma} \|v\|_q^q$.
\end{lemma}

Note that the corresponding Nehari and Poho\v zaev's identities are as follows
$$
\int_{\mathbb R^N}|\nabla \tilde w_\lambda|^2+\lambda^\sigma\xi_\lambda^{2}\int_{\mathbb R^N}|
\tilde w_\lambda|^2=\int_{\mathbb R^N}(I_\alpha\ast |\tilde w_\lambda|^{{2_\alpha^*}})|\tilde w_\lambda|^{2_\alpha^*}+\lambda^\sigma\xi_\lambda^{N-\frac{N-2}{2}q}\int_{\mathbb R^N}|\tilde w_\lambda|^q
$$
and 
$$
\frac{1}{2^*}\int_{\mathbb R^N}|\nabla \tilde w_\lambda|^2+\frac{1}{2}
\lambda^\sigma\xi_\lambda^{2}\int_{\mathbb R^N}|\tilde w_\lambda|^2=\frac{1}{2^*}\int_{\mathbb R^N}(I_\alpha\ast |\tilde w_\lambda|^{{2_\alpha^*}})|\tilde w_\lambda|^{2_\alpha^*}+\frac{1}{q}\lambda^\sigma\xi_\lambda^{N-\frac{N-2}{2}q}\int_{\mathbb R^N}|\tilde w_\lambda|^q,
$$
it follows that 
$$
\left(\frac{1}{2}-\frac{1}{2^*}\right)\lambda^\sigma\xi_\lambda^{2}\int_{\mathbb R^N}|\tilde w_\lambda|^2=\left(\frac{1}{q}-\frac{1}{2^*}\right)\lambda^\sigma\xi_\lambda^{N-\frac{N-2}{2}q}\int_{\mathbb R^N}|\tilde w_\lambda|^q.
$$
Thus, we obtain
\begin{equation}\label{e34}
\xi_\lambda^{\frac{(N-2)(q-2)}{2}}\int_{\mathbb R^N}|\tilde w_\lambda|^2=\frac{2(2^*-q)}{q(2^*-2)}\int_{\mathbb R^N}|\tilde w_\lambda|^q.
%\eqno(3.4)
\end{equation}

To control the norm $\|\tilde w_\lambda\|_2$, we note that  for any $\lambda>0$, $\tilde w_\lambda>0$ satisfies the linear inequality
$$
-\Delta \tilde w_\lambda+\lambda^\sigma\xi_\lambda^{2}\tilde w_\lambda
%=(I_\alpha\ast |\tilde w_\lambda|^{{2_\alpha^*}})\tilde w_\lambda^{{2_\alpha^*}-1}+\lambda^\sigma\xi_\lambda^{N-\frac{N-2}{2}q}\tilde w_\lambda^{q-1}
>0,  \qquad x\in \mathbb R^N.
$$
Arguing as in  that of   \cite[Lemma 4.8]{Moroz-1}, we can prove the following 

\begin{lemma}\label{l33}   There exists a constant $c>0$ such that 
\begin{equation}\label{e35}
\tilde w_\lambda(x)\ge c|x|^{-(N-2)}\exp({-\lambda^{\sigma/2}\xi_\lambda}|x|),  \quad |x|\ge 1.
%\eqno(3.5)
\end{equation}
\end{lemma}

Choosing $\lambda_\nu=\nu$ and $\beta_\nu=\gamma_\nu=1$ in Proposition \ref{p51},  we obtain the  following result, which establish the optimal decay estimate of $\tilde w_\lambda$ at infinity.

\begin{lemma} \label{l34}  Assume $N=3,4$ and $ 2<q<2^*$. Then there exists a constant $C>0$ such that for small $\lambda>0$, there holds 
\begin{equation}\label{e36}
\tilde w_\lambda(x)\le C(1+|x|)^{-(N-2)}, \qquad x\in \mathbb R^N.
%\eqno(3.6)
\end{equation}
\end{lemma}

\begin{lemma}\label{l35}    Assume $N=3,4$ and $ 2<q<2^*$. Then there exist constants $L_0>0$ and $C_0>0$ such that for any small $\lambda>0$ and $|x|\ge L_0\lambda^{-\sigma/2}\xi_\lambda^{-1}$, 
\begin{equation}\label{e37}
\tilde w_\lambda(x)\le C_0|x|^{-(N-2)}\exp(-\frac{1}{2}\lambda^{\sigma/2}\xi_\lambda |x|).
%\eqno(3.7)
\end{equation}
\end{lemma}
\begin{proof}  In \cite{Ma-2}, it has been shown that there exists a large $L_0>0$ such that
 $$
-\Delta\tilde w_\lambda(x)+\frac{1}{2}\lambda^{\sigma}\xi_\lambda^2\tilde w_\lambda(x)\le 0, \quad 
{\rm for \ all} \  |x|\ge L_0\lambda^{-\sigma/2}\xi_\lambda^{-1}.
$$ 
We  introduce a positive function 
$$
\psi(x):=C_1|x|^{-(N-2)}\exp(-\frac{1}{2}\lambda^{\sigma/2}\xi_\lambda |x|).
$$
A direct computation shows that 
$$
-\Delta \psi (x)+\frac{1}{2}\lambda^{\sigma}\xi_\lambda \psi(x)=C[\frac{1}{2}(N-3)\lambda^{\sigma/2}\xi_\lambda|x|^{-N+1} +\frac{3}{4}\lambda^{\sigma}\xi_\lambda^2|x|^{-N+2}]\exp(-\frac{1}{2}\lambda^{\sigma/2}\xi_\lambda |x|)\ge 0.
$$
By virtue of Lemma \ref{l34}, it follows that $\tilde w_\lambda(x)\le \psi(x)$ for $|x|=L_0\lambda^{-\sigma/2}\xi_\lambda^{-1}$ provided $C_1>0$ is large enough. 
Hence, the comparison principle implies that if $|x|\ge L_0\lambda^{-\sigma/2}\xi_\lambda^{-1}$, then
$$
\tilde w_\lambda(x)\le C_1|x|^{-(N-2)}\exp(-\frac{1}{2}\lambda^{\sigma/2}\xi_\lambda |x|).
$$
The proof is complete.
\end{proof}

Based on Lemma \ref{l33} and Lemma \ref{l35}, we obtain the following  precise asymptotic behavior of $\|\tilde w_\lambda\|_q^q$ as $\lambda\to 0$.

\begin{lemma}\label{l36}  Assume $N=3,4$ and $ 2<q<2^*$. Then as $\lambda\to 0$, there holds
\begin{equation}\label{e38}
\|\tilde w_\lambda\|_q^q\sim \left\{\begin{array}{cl}
\lambda^{-\frac{\sigma}{2}[N-(N-2)q]}\xi_\lambda^{-[N-(N-2)q]}, \  \ &{if} \  2<q<\frac{N}{N-2},\\
\ln(\lambda^{-\frac{\sigma}{2}}\xi_\lambda^{-1}), \ \ &{if} \  \ q=\frac{N}{N-2},\\
1, \    \    \     &{if}   \    \  q>\frac{N}{N-2}.
\end{array}\right.
%\eqno(3.8)
\end{equation}
\end{lemma}

\begin{proof}  By virtue of Lemma \ref{l33},  if $q<\frac{N}{N-2}$, we then have  
$$
\begin{array}{rcl}
\|\tilde w_\lambda\|_q^q&\gtrsim &\int_{\lambda^{-\sigma/2}\xi_\lambda^{-1}}^{+\infty}r^{-(N-2)q+N-1}\exp(-q\lambda^{\sigma/2}\xi_\lambda r)dr\\
&= &\lambda^{\frac{\sigma}{2}[(N-2)q-N]}\xi_\lambda^{(N-2)q-N}\int_1^{+\infty}s^{-(N-2)q+N-1}e^{-qs}ds\\
&\gtrsim &\lambda^{\frac{\sigma}{2}[(N-2)q-N]}\xi_\lambda^{(N-2)q-N}.
\end{array}
$$
If $q=\frac{N}{N-2}$, then
$$
\begin{array}{rcl}
\|\tilde w_\lambda\|_q^q&\gtrsim &\int_1^{\lambda^{-\sigma/2}\xi_\lambda^{-1}}r^{-(N-2)q+N-1}\exp(-\lambda^{\sigma/2}\xi_\lambda r)dr\\
&\gtrsim&  \int_1^{\lambda^{-\sigma/2}\xi_\lambda^{-1}}r^{-1}dr\\
&\sim&\ln (\lambda^{-\sigma/2}\xi_\lambda^{-1}).
\end{array}
 $$
If $q>\frac{N}{N-2}$, then
$$
\begin{array}{rcl}
\|\tilde w_\lambda\|_q^q&\gtrsim &\int_1^{\lambda^{-\sigma/2}\xi_\lambda^{-1}}r^{-(N-2)q+N-1}\exp(-q\lambda^{\sigma/2}\xi_\lambda r)dr\\
&\gtrsim&  \int_1^{\lambda^{-\sigma/2}\xi_\lambda^{-1}}r^{-(N-2)q+N-1}dr\\
&=&\frac{1}{(N-2)q-N}-\frac{1}{(N-2)q-N}\lambda^{\frac{\sigma}{2}[(N-2)q-N]}\xi_\lambda^{(N-2)q-N}\\
&\gtrsim& 1.
\end{array}
$$

On the other hand, by Lemma \ref{l34} and Lemma \ref{l35}, we have 
$$\begin{array}{rcl}
\|\tilde w_\lambda\|_q^q&\lesssim &1+\int_1^{L_0\lambda^{-\sigma/2}\xi_\lambda^{-1}}r^{-(N-2)q+N-1}dr\\
&\mbox{}&+\int_{L_0\lambda^{-\sigma/2}\xi_\lambda^{-1}}^{+\infty}r^{-(N-2)q+N-1}\exp(-\frac{q}{2}\lambda^{\sigma/2}\xi_\lambda r)dr.
\end{array}
$$
If $2<q<\frac{N}{N-2}$, then
$$
\begin{array}{rcl}
\int_1^{L_0\lambda^{-\sigma/2}\xi_\lambda^{-1}}r^{-(N-2)q+N-1}dr&=&\left.\frac{1}{N-(N-2)q}r^{N-(N-2)q}\right |_1^{L_0\lambda^{-\sigma/2}\xi_\lambda^{-1}}\\
&\lesssim &\lambda^{-\frac{\sigma}{2}[N-(N-2)q]}\xi_\lambda^{-N+(N-2)q}.
\end{array}
$$
If $q=\frac{N}{N-2}$, then 
$$
\int_1^{L_0\lambda^{-\sigma/2}\xi_\lambda^{-1}}r^{-(N-2)q+N-1}dr=\left.\ln r\right |_1^{L_0\lambda^{-\sigma/2}\xi_\lambda^{-1}}
\lesssim \ln(\lambda^{-\frac{\sigma}{2}}\xi_\lambda^{-1}).
$$
If $q>\frac{N}{N-2}$, then 
$$
\int_1^{L_0\lambda^{-\sigma/2}\xi_\lambda^{-1}}r^{-(N-2)q+N-1}dr=\left.\frac{1}{(N-2)q-N}r^{N-(N-2)q}\right |^1_{L_0\lambda^{-\sigma/2}\xi_\lambda^{-1}}
\sim  1.
$$
We also have 
$$
\begin{array}{rl}
&\int_{L_0\lambda^{-\sigma/2}\xi_\lambda^{-1}}^{+\infty}r^{-(N-2)q+N-1}\exp(-\frac{q}{2}\lambda^{\sigma/2}\xi_\lambda r)dr\\
&=\lambda^{\frac{\sigma}{2}[(N-2)q-N]}\xi_\lambda^{(N-2)q-N}\int_{L_0}^{+\infty}s^{-(N-2)q+N-1}e^{-\frac{q}{2}s}ds\\
&\lesssim \left\{\begin{array}{rcl}
\lambda^{\frac{\sigma}{2}[(N-2)q-N]}\xi_\lambda^{(N-2)q-N}, \quad &if & \   2<q<\frac{N}{N-2},\\
1, \qquad \qquad\qquad \qquad  &if &   q=\frac{N}{N-2},\\
1, \qquad \qquad\qquad \qquad  &if &   q>\frac{N}{N-2}.
\end{array}\right.
\end{array}
$$
Thus, combining  the above inequalities, we conclude the proof. \end{proof}

\begin{proof}[Proof of Theorem \ref{t11}]  Assume $N=3$. Firstly, since $m_\lambda$ is the minimum of $I_\lambda(v)$ on the Nehari manifold $\mathcal M_\lambda$,  by using the associated fibering map, it is easy to show that $m_\lambda$ is nonincreasing  with respect to $\lambda$.  It is well known that $m_0=\frac{2+\alpha}{2(N+\alpha)}S_\alpha^{\frac{N+\alpha}{2+\alpha}}$, therefore, we have $m_\lambda\le \frac{2+\alpha}{2(N+\alpha)}S_\alpha^{\frac{N+\alpha}{2+\alpha}}$ for all $\lambda>0$.  On the other hand, if $m_\lambda<\frac{2+\alpha}{2(N+\alpha)}S_\alpha^{\frac{N+\alpha}{2+\alpha}}$ for some $\lambda>0$, arguing as in \cite{Li-1} , we can show that $(Q_\lambda)$ has a ground state solution.  

Suppose for the contrary that $(Q_\lambda)$ has a ground state for all $\lambda>0$. If $2<q<3$, then
$$
\xi_\lambda\lambda^{-\sigma/2}\xi_\lambda^{-1}\lesssim \xi_\lambda^{\frac{q-2}{2}}\|\tilde w_\lambda\|_2^2\lesssim \|\tilde w_\lambda\|_q^q\lesssim \lambda^{\frac{\sigma}{2}(q-3)}\xi_\lambda^{q-3},
$$
and hence
$$
1\lesssim \lambda^{\frac{\sigma}{2}(q-2)}\xi_\lambda^{\frac{q-2}{2}},
$$
which is a contradiction.
If $q=3$, then
$$
\xi_\lambda^{-\frac{1}{2}}\lambda^{-\frac{\sigma}{4}}=\lambda^{\frac{\sigma}{4}}(\lambda^{-\frac{\sigma}{2}}\xi_\lambda^{-\frac{1}{2}})\lesssim \lambda^{-\frac{\sigma}{2}}\xi_\lambda^{-\frac{1}{2}}\lesssim \ln(\lambda^{-\frac{\sigma}{2}}\xi_\lambda^{-1}).
$$
This is also a contradiction.
If $3<q\le 4$, then 
$$
\lambda^{-\frac{\sigma}{2}}\xi_\lambda^{\frac{q-4}{2}}\lesssim 1.
$$
This is also a contradiction.  Therefore, $(Q_\lambda)$ has no ground state solution for small $\lambda>0$. 
Therefore, we conclude that  $m_\lambda\ge \frac{2+\alpha}{2(N+\alpha)}S_\alpha^{\frac{N+\alpha}{2+\alpha}}$ for small $\lambda>0$. Thus, we get
$m_\lambda=\frac{2+\alpha}{2(N+\alpha)}S_\alpha^{\frac{N+\alpha}{2+\alpha}}$  for small  $\lambda>0.$ Let
$$
\lambda_q:=\sup\left\{\lambda>0 \left | \  m_\lambda=\frac{2+\alpha}{2(N+\alpha)}S_\alpha^{\frac{N+\alpha}{2+\alpha}} \right.\right\}.
$$
Then $\lambda_q>0$ for $2<q\le 4$. It is well known (cf.\cite{Li-1}) that $m_\lambda<\frac{2+\alpha}{2(N+\alpha)}S_\alpha^{\frac{N+\alpha}{2+\alpha}}$ for sufficiently large $\lambda>0$. Therefore, we have $0<\lambda_q<+\infty$. Hence, $m_\lambda=\frac{2+\alpha}{2(N+\alpha)}S_\alpha^{\frac{N+\alpha}{2+\alpha}}$  for $\lambda\in (0,\lambda_q)$ and $m_\lambda<\frac{2+\alpha}{2(N+\alpha)}S_\alpha^{\frac{N+\alpha}{2+\alpha}}$  for $\lambda>\lambda_q.$  If there exists $\lambda\in (0,\lambda_q)$ such that $m_\lambda$ is attained, then arguing as in \cite[Lemma 3.3]{Wei-1}, we obtain  that $m_{\lambda'}<m_\lambda=\frac{2+\alpha}{2(N+\alpha)}S_\alpha^{\frac{N+\alpha}{2+\alpha}}$  for $\lambda'>\lambda,$ which contradicts to the definition of $\lambda_q$.  Thus, we conclude that $(Q_\lambda)$ has no ground state solution for $\lambda\in (0,\lambda_q)$ and admits a ground state solution for $\lambda>\lambda_q$.   This completes the proof of first part of Theorem \ref{t11}.
 \end{proof}
\vskip 3mm

\subsection{Multiplicity}\label{s4}

In order to obtain the second positive solution of $(Q_\lambda)$, we search for solutions to $(P_\varepsilon)$ having prescribed mass, and in this case $\varepsilon\in \mathbb R$ is part of the unknown.  That is, for a fixed $a>0$,  search for $u\in H^1(\mathbb R^N)$ and $\lambda\in \mathbb R$ satisfying
\begin{equation}\label{e39}
\left\{\begin{array}{rl}
&-\Delta u+\lambda u=(I_\alpha\ast |u|^{2_\alpha^*})|u|^{2_\alpha^*-2}u+\nu|u|^{q-2}u, \  \ in  \  \mathbb R^N,\\
&u\in H^1(\mathbb R^N),  \   \    \   \int_{\mathbb R^N}|u|^2=a^2.
\end{array}\right.
%\eqno(3.9)
\end{equation}
The solution of \eqref{e39} is usually denoted by a pair $(u,\lambda)$ and  called a normalized solution.

  It is standard to check that the energy functional 
\begin{equation}\label{e310}
E_{a,\nu}(u)=\frac{1}{2}\int_{\mathbb R^N}|\nabla u|^2-\frac{1}{22_\alpha^*}\int_{\mathbb R^N}(I_\alpha\ast |u|^{2_\alpha^*})|u|^{2_\alpha^*}-\frac{\nu}{q}\int_{\mathbb R^N}|u|^{q}
%\eqno(3.10)
\end{equation}
is of class $C^1$ and that a critical point of $E$ restricted to the (mass) constraint 
\begin{equation}\label{e311}
S_a=\{ u\in H^1(\mathbb R^N): \ \|u\|_2^2=a^2\}
%\eqno(3.11)
\end{equation}
gives a solution to \eqref{e39}.  Here   $\lambda\in \mathbb R$ arises as a Lagrange multiplier. We refer the readers to \cite{Jeanjean-2,Jeanjean-3, Li-3,Li-4, Soave-1, Soave-2, Sun-1, Sun-2} and the  references therein.

Let 
$$
P_\nu(u)=\int_{\mathbb R^N}|\nabla u|^2-\int_{\mathbb R^N}(I_\alpha\ast |u|^{2_\alpha^*})|u|^{2_\alpha^*}-\nu\gamma_q\int_{\mathbb R^N}|u|^{q},
$$
where $\gamma_q:=\frac{N(q-2)}{2q}\in (0,1)$. 
Define 
$$
\mathcal{P}_{a,\nu}:=\{u\in S_a \  | \ P_\nu(u)=0\},
$$
$$
\mathcal{P}_+^{a,\nu}:=\{u\in S_a\cap \mathcal{P}_{a,\nu} \ | \ 2\|\nabla u\|_2^2>22_\alpha^*\int_{\mathbb R^N}(I_\alpha\ast |u|^{2_\alpha^*})|u|^{2_\alpha^*}+\nu q\gamma_q^2\|u\|_{q}^{q}\},
$$
$$
\mathcal{P}_0^{a,\nu}:=\{u\in S_a\cap \mathcal{P}_{a,\nu} \ | \ 2\|\nabla u\|_2^2=22_\alpha^*\int_{\mathbb R^N}(I_\alpha\ast |u|^{2_\alpha^*})|u|^{2_\alpha^*}+\nu q\gamma_q^2\|u\|_{q}^{q}\},
$$
$$
\mathcal{P}_-^{a,\nu}:=\{u\in S_a\cap \mathcal{P}_{a,\nu} \ | \ 2\|\nabla u\|_2^2<22_\alpha^*\int_{\mathbb R^N}(I_\alpha\ast |u|^{2_\alpha^*})|u|^{2_\alpha^*}+\nu q\gamma_q^2\|u\|_{q}^{q}\}.
$$
Put
$$
m_\nu:=\inf_{u\in \mathcal{P}_+^{a,\nu}}E_{a,\nu}(u), \qquad \mathcal E_\nu:=\inf_{u\in \mathcal{P}_-^{a,\nu}}E_{a,\nu}(u).
$$

Let
$$
K_q:=\frac{22_\alpha^*-q\gamma_q}{22_\alpha^*(2-q\gamma_q)S_\alpha^{2_\alpha^*}}\left(\frac{2_\alpha^*(2-q\gamma_q)C_{Nq}^qS_\alpha^*}{q(2_\alpha^*-1)}\right)^{\frac{2(2_\alpha^*-1)}{22_\alpha^*-q\gamma_q}}, \qquad \gamma_q=\frac{N(q-2)}{2q}.
$$
Then the following results are proved in Li \cite[Theorem 1.3 and Theorem 1.5]{Li-4} and Li  \cite[Theorem 1.2]{Li-5}, respectively. 

\begin{theorem}\label{t37}\cite{Li-4} Let $N\ge 3$, $\alpha\in (0,N), \nu>0, a>0$ and  $q\in (2,2+\frac{4}{N})$. If  
\begin{equation}\label{e312}
\nu a^{q(1-\gamma_q)}\le \left(2K_q\right)^{-\frac{2(N+\alpha)-q\gamma_q(N-2)}{2(2+\alpha)}},
%\eqno(3.12)
\end{equation}
then  the following statements hold true:

(1)  $E_{a,\nu}|_{S_a}$ has a critical point $\tilde u_\nu$ at negative level $m_\nu<0$, which is radially symmetric and non-increasing ground state to \eqref{e39}  with some $\lambda_\nu>0$.

(2)   there exists a mountain pass solution $u_\nu$ to \eqref{e39}  with some $\lambda_\nu>0$, which is radially symmetric and  non-increasing, and  satisfies 
\begin{equation}\label{e313}
0<E_{a,\nu}(u_\nu)=\mathcal E_\nu= \inf_{u\in P_{a,-}}E_{a,\nu}(u)<m_\nu+\frac{2+\alpha}{2(N+\alpha)}S_\alpha^{\frac{N+\alpha}{2+\alpha}}.
%\eqno(3.13)
\end{equation}
In particular, $u_\nu$ is not a ground state.    
\end{theorem}

\begin{theorem}\label{t38} \cite{Li-5} Let $N\ge 3$, $\alpha\in (0,N), \nu>0, a>0$ and  $q\in [2+\frac{4}{N}, 2^*)$. If  $q=2+\frac{4}{N}$,  we further assume that
\begin{equation}\label{e314}
\nu a^{4/N}<\frac{q}{2C_{Nq}^q},
%\eqno(3.14)
\end{equation}
then  the equation \eqref{e39} has a mountain pass type normalized ground state $u_\nu$, which is radially symmetric and  non-increasing, and 
 satisfies
\begin{equation}\label{e315}
0<E_{a,\nu}(u_\nu)=\mathcal E_\nu=\inf_{u\in P_{a,-}}E_{a,\nu}(u)<\frac{2+\alpha}{2(N+\alpha)}S_\alpha^{\frac{N+\alpha}{2+\alpha}}.
%\eqno(3.15)
\end{equation}
 Moreover, $u_\nu$ solves \eqref{e39} with some $\lambda_\nu>0$. 
 \end{theorem}

Assume that $u_\nu$ is a mountain pass type solution of \eqref{e39}, then
\begin{equation}\label{e316}
-\Delta u_\nu +\lambda_\nu u_\nu =(I_\alpha\ast |u_\nu|^{2_\alpha^*})|u_\nu|^{2_\alpha^*-2}u_\nu+\nu|u_\nu|^{q-2}u_\nu.
%\eqno(3.16)
\end{equation}
Then, by the corresponding Nehari and Poho\v zaev identities,  it is easy to show that
\begin{equation}\label{e317}
\lambda_{\nu}a^2=\frac{2(2^*-q)}{q(2^*-2)}\nu\int_{\mathbb R^N}|u_{\nu}|^q.
%\eqno(3.17)
\end{equation}
Therefore,  $\lambda_\nu$ is continuous  respect to $\nu$.

Let $u\in S_a$ and $u^t(x)=t^{\frac{N}{2}}u(tx)$ for all $t>0$, then $u^t\in S_a$ and 
\begin{equation}\label{e318}
\begin{array}{rcl}
E_{a,\nu}(u^t)&=&\frac{t^2}{2}\|\nabla u\|_2^2-\frac{t^{22_\alpha^*}}{22_\alpha^*}\int_{\mathbb R^N}(I_\alpha\ast |u|^{2_\alpha^*})|u|^{2_\alpha^*}-\frac{\nu t^{q\gamma_q}}{q}\|u\|^q_q\\
&\ge &\frac{t^2}{2}\|\nabla u\|_2^2-\frac{t^{22_\alpha^*}}{22_\alpha^*}S_\alpha^{-2_\alpha^*}\|\nabla u\|_2^{22_\alpha^*}-\frac{\nu t^{q\gamma_q}}{q}C_{Nq}^qa^{q(1-\gamma_q)}\|\nabla u\|_2^{q\gamma_q}\\
&=&h(\|\nabla u^t\|_2^2),
\end{array}
%\eqno(3.18)
\end{equation}
where 
$$
h(\rho):=\rho\left\{\frac{1}{2}-\frac{S_\alpha^{-2_\alpha^*}}{22_\alpha^*}\rho^{2_\alpha^*-1}-\frac{\nu}{q}a^{q(1-\gamma_q)}\rho^{\frac{q\gamma_q}{2}-1}\right\}.
$$
Clearly, there exists a unique $\rho_\nu>0$ such that $h'(\rho_\nu)=0$, which  implies that
$$
S_\alpha^{-2_\alpha^*}\rho_\nu^{2_\alpha^*-1}+\nu\gamma_q a^{q(1-\gamma_q)}\rho_\nu^{\frac{q\gamma_q-2}{2}}=1.
$$
Therefore, $\rho_\nu\le S_\alpha^{\frac{2_\alpha^*}{2_\alpha^*-1}}=S_\alpha^{\frac{N+\alpha}{2+\alpha}}$.  If $q\in [2+\frac{4}{N}, 2^*)$, then $q\gamma_q\ge 2$.  Thus we get
$$
\rho_\nu=S_\alpha^{\frac{N+\alpha}{2+\alpha}}\left(1-\nu\gamma_q a^{q(1-\gamma_q)}\rho_\nu^{\frac{q\gamma_q-2}{2}}\right)^{\frac{N-2}{2+\alpha}}
=S_\alpha^{\frac{N+\alpha}{2+\alpha}}(1+O(\nu a^{q(1-\gamma_q)})), \quad {\rm as} \  \nu\to 0,
$$
which together with \eqref{e318} implies  that
\begin{equation}\label{e319}
\begin{array}{rcl}
\mathcal E_\nu&=&\inf_{u\in S_a}\sup_{s\in (0,+\infty)}E(u^s)\\
&\ge& h(\rho_\nu)=\frac{2+\alpha}{2(N+\alpha)}S_\alpha^{\frac{N+\alpha}{2+\alpha}}(1+O(\nu a^{q(1-\gamma_q)})),
 \ for \  \ q\in [2+\frac{4}{N}, 2^*).
 \end{array}
%\eqno(3.19)
\end{equation}
Thus,  by \eqref{e315}, for $q\in [2+\frac{4}{N}, 2^*)$, we get 
$$
\mathcal E_\nu= \frac{2+\alpha}{2(N+\alpha)}S_\alpha^{\frac{N+\alpha}{2+\alpha}}+O(\nu a^{q(1-\gamma_q)}), \quad  as \  \  \nu a^{q(1-\gamma_q)}\to 0.
$$

\vskip 3mm 
Let 
$$
\quad W_\varepsilon(x)=\varepsilon^{-\frac{N-2}{2}}W_1(x/\varepsilon), \quad \varepsilon>0, \quad 
V_\varepsilon(x)=W_\varepsilon(x)\varphi(R_\varepsilon^{-1}x), 
$$
where $\varphi\in C_0^{\infty}(\mathbb R^N)$ is a radial cut-off function with $\varphi\equiv 1$ for $|x|\le 1$ and $\varphi\equiv 0$ for $|x|\ge 2$, and $R_\varepsilon$ is chosen such that $V_\varepsilon\in S_a$. More precisely, we choose $\varepsilon>0$ sufficiently small such that
$$
a^2=\int_{\mathbb R^N}(W_\varepsilon(x)\varphi(R^{-1}_\varepsilon x))^2\sim \varepsilon^2\int^{R_\varepsilon \varepsilon^{-1}}_1r^{3-N}\sim \left\{\begin{array}{rcl}
\varepsilon^2\ln (R_\varepsilon \varepsilon^{-1}), \quad  &for& \ N=4,\\
\varepsilon R_\varepsilon, \qquad \quad &for & \  N=3,
\end{array}\right.
$$
which implies that
$$
R_\varepsilon\varepsilon^{-1}\sim \left\{\begin{array}{rcl}
e^{a^2\varepsilon^{-2}},\quad &for& \  N=4,\\
\varepsilon^{-2}, \quad &for & \   N=3.
\end{array}\right.
$$
Discussed as in Soave \cite[Lemma 4.2, 5.1 and 6.1]{Soave-1},  there exists $t_{\nu,\varepsilon}>0$ such that 
$$
\|\nabla V_\varepsilon\|_2^2=\nu\gamma_qt_{\nu,\varepsilon}^{q\gamma_q-2}\|V_\varepsilon\|^q_q +t_{\nu,\varepsilon}^{\frac{2(2+\alpha)}{N-2}}\int_{\mathbb R^N}(I_\alpha\ast |V_\varepsilon|^{2_\alpha^*})|V_\varepsilon|^{2_\alpha^*},
$$
and 
$$
2\|\nabla V_\varepsilon\|_2^2<\nu q\gamma_q^2t_{\nu,\varepsilon}^{q\gamma_q-2}\|V_\varepsilon\|_q^q +\frac{2(N+\alpha)}{N-2}t_{\nu,\varepsilon}^{\frac{2(2+\alpha)}{N-2}}\int_{\mathbb R^N}(I_\alpha\ast |V_\varepsilon|^{2_\alpha^*})|V_\varepsilon|^{2_\alpha^*}.
$$
It follows that $\{t_{\nu,\varepsilon}\}$ is uniformly bounded and bounded below from 0 for all $\nu, \varepsilon>0$ sufficiently small. 
Since
$$
\|\nabla V_\varepsilon\|_2^2=S_\alpha^{\frac{N+\alpha}{2+\alpha}}+O((R_\varepsilon\varepsilon^{-1})^{2-N})
$$
and
$$
\int_{\mathbb R^N}(I_\alpha\ast |V_\varepsilon|^{2_\alpha^*})|V_\varepsilon|^{2_\alpha^*}
=S_\alpha^{\frac{N+\alpha}{2+\alpha}}+O((R_\varepsilon\varepsilon^{-1})^{1-2N})
$$
for small $\varepsilon>0$, we have
$$
S_\alpha^{\frac{N+\alpha}{2+\alpha}}(1-t_{\nu,\varepsilon}^{\frac{2(2+\alpha)}{N-2}})=\nu\gamma_qt_{\nu,\varepsilon}^{q\gamma_q-2}\|V_\varepsilon\|_q^q
+O((R_\varepsilon\varepsilon^{-1})^{2-N}).
$$
Therefore, we get
$$
t_{\nu,\varepsilon}=1-(1+o_\nu(1))\frac{\nu\gamma_q\|V_\varepsilon\|_q^q+O((R_\varepsilon\varepsilon^{-1})^{2-N})
}{\frac{2(2+\alpha)}{N-2}S_\alpha^{\frac{N+\alpha}{2+\alpha}}}.
$$
For small $\varepsilon>0$,  it is proved in \cite{Moroz-1} that 
\begin{equation}\label{e320}
\|V_\varepsilon\|_q^q\sim \left\{\begin{array}{rcl}
 \varepsilon^{N-\frac{N-2}{2}q},  \  \  &if& \  \   N=3, 4, \ \frac{N}{N-2}<q<2^*,\\
 \varepsilon^{\frac{3}{2}}\ln\frac{1}{\varepsilon}, \quad \  &if & \  \  N=3, \ q=3,\\
 \varepsilon^{\frac{3}{2}q-3}, \qquad  &if& \   \   N=3, \  2<q<3.
 \end{array}\right.
%\eqno(3.20)
\end{equation}

In what follows, we fix $a>0$ and give a precise asymptotic description  of $\lambda_\nu$ in the case that $N=3$.  
 Arguing as in \cite{Wei-1}, we prove that
$$
\|u_\nu\|_q^q\ge (1+o_\nu(1))\left(\|V_\varepsilon\|_q^q-C\nu^{-1}(R_\varepsilon\varepsilon^{-1})^{2-N}\right),
$$
which together with \eqref{e320} implies  that
$$
\|u_\nu\|_q^q\gtrsim \left\{\begin{array}{rcl}
 \varepsilon^{3-\frac{1}{2}q}-C\nu^{-1}\varepsilon^2, \qquad  \quad   \  &if& \ \  N=3, \ 3<q<6,\\
 \varepsilon^{\frac{3}{2}}\ln\frac{1}{\varepsilon}-C\nu^{-1}\varepsilon^2, \qquad \ \  &if & \  \  N=3, \ q=3,\\
 \varepsilon^{\frac{3}{2}q-3}-C\nu^{-1}\varepsilon^2, \qquad \quad  &if& \   \   N=3, \  2<q<3.
 \end{array}\right.
$$
Clearly, we can choose $\varepsilon_\nu>0$ such that
\begin{equation}\label{e321}
\left\{\begin{array}{rcl}\varepsilon_\nu=(\frac{3-q/2}{2C}\nu)^{\frac{2}{q-2}}, \qquad \  \  &if& \  N=3, \ 3<q<6,\\
-(\frac{3}{2}\ln\varepsilon_\nu +1)=2C\nu^{-1}\varepsilon_\nu^{\frac{1}{2}}, \ &if& \ N=3, \ q=3,\\
\varepsilon_\nu=(\frac{3(q-2)}{4C}\nu)^{\frac{2}{10-3q}},\qquad &if& \ N=3, \ 2<q<3.
\end{array}\right.
%\eqno(3.21)
\end{equation}
Then $\varepsilon_\nu\to 0$ as $\nu\to 0$ and the right hand sides of the  above estimates take the maximum  at $\varepsilon=\varepsilon_\nu$,  and hance
\begin{equation}\label{e322}
\|u_\nu\|_q^q\gtrsim \left\{\begin{array}{rcl}
 \varepsilon_\nu^{3-\frac{1}{2}q},  \quad\quad     &if& \  \   N=3, \  3<q<6,\\
 \varepsilon_\nu^{\frac{3}{2}}\ln\frac{1}{\varepsilon_\nu}, \quad \ \  &if & \  \  N=3, \ q=3,\\
 \varepsilon_\nu^{\frac{3}{2}q-3}, \quad \quad  \  &if& \   \   N=3, \  2<q<3.
 \end{array}\right.
%\eqno(3.22)
\end{equation}
Since $a>0$ is fixed, $\lambda_\nu\sim \lambda_\nu a^2\sim \nu\|u_\nu\|_q^q$, it follows from \eqref{e321} and \eqref{e322} that
\begin{equation}\label{e323}
\lambda_\nu\gtrsim \left\{\begin{array}{rcl}
\nu^{\frac{4}{q-2}}, \quad\quad \ &if& \    N=3, \  3<q<6,\\
\varepsilon_\nu^2, \quad\qquad \  &if & \   N=3,\  q=3,\\
\nu^{\frac{4}{10-3q}}, \quad \   \ &if& \  N=3, \  2<q<3,
\end{array}\right.
%\eqno(3.23)
\end{equation}
where $\varepsilon_\nu>0$ is such that $\nu\ln \frac{1}{\varepsilon_\nu}\sim \varepsilon_\nu^{\frac{1}{2}}$ as $\nu\to 0$.

Arguing as in \cite[Proposition 4.1]{Wei-1}, we  also show that  as $\nu\to 0$, 
$$
\|\nabla u_\nu\|_2^2, \   \   \ \int_{\mathbb R^N}(I_\alpha\ast |u_\nu|^{2_\alpha^*})|u_\nu|^{2_\alpha^*}\to S_\alpha^{\frac{N+\alpha}{2+\alpha}}.
$$
Moreover, up to a subsequence, $\{u_\nu\}$ is a minimizing sequence of the following minimizing problem:
$$
S_\alpha=\inf_{u\in D^{1,2}(\mathbb R^N)\setminus \{0\}}\frac{\|\nabla u\|_2^2}{(\int_{\mathbb R^N}(I_\alpha\ast |u|^{2_\alpha^*})|u|^{2_\alpha^*})^{\frac{1}{2_\alpha^*}}}.
$$
Since $u_\nu$ is radially symmetric,  by Lemma \ref{l27}, there exists $\sigma_\nu>0$ such that 
\begin{equation}\label{e324}
v_\nu(x)=\sigma_\nu^{\frac{N-2}{2}}u_\nu(\sigma_\nu x)\to W_1 \quad \ in \  \ D^{1,2}(\mathbb R^N), \qquad as \ \ \nu\to 0.
%\eqno(3.24)
\end{equation}
Clearly,  $v=v_\nu$ satisfies 
\begin{equation}\label{e325}
-\Delta v_\nu +\lambda_\nu\sigma_\nu^2 v_\nu=(I_\alpha\ast|v_\nu|^{2_\alpha^*})|v_\nu|^{2_\alpha^*-2}v_\nu+\nu\sigma_\nu^{N-\frac{N-2}{2}q}|v_\nu|^{q-2}v_\nu.
%\eqno(3.25)
\end{equation}
We remark that since $W_{1}\not\in L^2(\mathbb R^N)$ for $N=3,4$ and $\sigma_\nu^2\|v_\nu\|_2^2=a^2$, by the Fatou's lemma, we have $\lim_{\nu\to 0}\sigma_\nu=0$. 

Since $a>0$ is fixed, by \eqref{e317}, we have 
\begin{equation}\label{e326}
\lambda_\nu\sim \lambda_\nu a^2=\lambda_\nu\|u_\nu\|_2^2=\nu \frac{2(2^*-q)}{q(2^*-2)}\|u_\nu\|_q^q\sim \nu\sigma_\nu^{N-\frac{N-2}{2}q}\|v_\nu\|_q^q.
%\eqno(3.26)
\end{equation}
Since $v_\nu\to W_1$ in $L^{2^*}(\mathbb R^N)$, as in \cite[Lemma 4.6]{Moroz-1}, using the embeddings $L^{2^*}(B_1)\hookrightarrow L^q(B_1)$ we prove that
$\liminf_{\nu\to 0}\|v_\nu\|_q^q>0$.  Therefore, it follows from \eqref{e323} and \eqref{e326} that
\begin{equation}\label{e327}
\nu\lesssim \min\{\lambda_\nu^{\frac{q-2}{4}}, \lambda_\nu^{\frac{10-3q}{4}}\}, \qquad \lambda_\nu\gtrsim \nu\sigma_\nu^{N-\frac{N-2}{2}q}.
%\eqno(3.27)
\end{equation}
By virtue of Proposition \ref{p51} with $\beta_\nu=\gamma_\nu=1$, we obtain 
$$
v_\nu(x)\lesssim (1+|x|)^{-(N-2)}, \quad x\in \mathbb R^N.
$$

Arguing as in the proof of Lemma \ref{l36}, we can prove the following

\begin{lemma}\label{l39}  Assume $N=3$ and $ 2<q<2^*$. Then 
\begin{equation}\label{e328}
\|v_\nu\|_q^q\sim \left\{\begin{array}{cl}
\lambda_\nu^{-\frac{1}{2}[N-(N-2)q]}\sigma_\nu^{-[N-(N-2)q]}, \  \ &{if} \ \  2<q<\frac{N}{N-2},\\
\ln(\lambda_\nu^{-\frac{1}{2}}\sigma_\nu^{-1}), \ \ &{if} \  \ q=\frac{N}{N-2},\\
1, \    \    \     &{if}   \    \  \frac{N}{N-2}<q<2^*.
\end{array}\right.
%\eqno(3.28)
\end{equation}
\end{lemma}

By virtue of Lemma \ref{l39}, we obtain 

\begin{lemma}\label{l310}  Assume $N=3$ and $ 2<q<2^*$. Then 
\begin{equation}\label{e329}
\lambda_\nu\sim \left\{\begin{array}{cl}
\nu\lambda_\nu^{-\frac{1}{2}[N-(N-2)q]}\sigma_\nu^{\frac{N-2}{2}q}, \  \ &{if} \ \   2<q<\frac{N}{N-2},\\
\nu\sigma_\nu^{N-\frac{N-2}{2}q} \ln(\lambda_\nu^{-\frac{1}{2}}\sigma_\nu^{-1}), \ \ &{if} \  \ q=\frac{N}{N-2},\\
\nu\sigma_\nu^{N-\frac{N-2}{2}q}, \    \    \     &{if}   \   \  \frac{N}{N-2}<q<2^*.
\end{array}\right.
%\eqno(3.29)
\end{equation}
\end{lemma}

\noindent{\bf Remark 3.1.}  It is easy to check that
$$
\begin{array}{rcl}
\lambda_\nu&\sim& \left\{\begin{array}{cl}
\nu^{\frac{2}{N+2-(N-2)q}}\sigma_\nu^{\frac{(N-2)q}{N+2-(N-2)q}}, \  \ &{if} \ \   2<q<\frac{N}{N-2},\\
\nu\sigma_\nu^{N-\frac{N-2}{2}q} \ln(\lambda_\nu^{-\frac{1}{2}}\sigma_\nu^{-1}), \ \ &{if} \  \ q=\frac{N}{N-2},\\
\nu\sigma_\nu^{N-\frac{N-2}{2}q}, \qquad    \    \     &{if}   \   \  \frac{N}{N-2}<q<2^*
\end{array}\right.\\ \\
&\sim& \left\{\begin{array}{cl}
\nu\sigma_\nu^{N-\frac{N-2}{2}q}\nu^{-\frac{N-(N-2)q}{N+2-(N-2)q}}\sigma_\nu^{-\frac{[2(N+2)-(N-2)q][N-(N-2)q]}{2[N+2-(N-2)q]}}, \ &{if}  \ 2<q<\frac{N}{N-2},\\
\nu\sigma_\nu^{N-\frac{N-2}{2}q} \ln(\lambda_\nu^{-\frac{1}{2}}\sigma_\nu^{-1}), \qquad\qquad\qquad  \qquad   &{if} \ q=\frac{N}{N-2},\\
\nu\sigma_\nu^{N-\frac{N-2}{2}q}, \qquad\qquad \qquad \qquad  \qquad \qquad    &{if} \ \frac{N}{N-2}<q<2^*.
\end{array}\right.
\end{array}
$$

We define $ w_\nu(x)=\varepsilon_\nu^{\frac{N-2}{2}}u_\nu(\varepsilon_\nu x)$, then $\|\nabla  w_\nu\|_2^2\sim \int_{\mathbb R^N}(I_\alpha\ast | w_\nu|^{2_\alpha^*})|w_\nu^{2_\alpha^*}\sim 1$ as $\nu\to 0$. 
It is easy to see that
\begin{equation}\label{e330}
\sigma_\nu^{N-\frac{N-2}{2}q}\|v_\nu\|_q^q= \|u_\nu\|_q^q=\varepsilon_\nu^{N-\frac{N-2}{2}q}\|w_\nu\|_q^q.
%\eqno(3.30)
\end{equation}
In what follows, we consider three cases: 

\noindent{\bf Case 1:  $N=3$ and $3<q<2^*$. } In this case, from \eqref{e322}, \eqref{e330} and  the fact that $ \|v_\nu\|_q\lesssim 1$, we get   
$$
\sigma_\nu^{N-\frac{N-2}{2}q}\gtrsim \sigma_\nu^{N-\frac{N-2}{2}}\|v_\nu\|_q^q=\|u_\nu\|_q^q\gtrsim \varepsilon_\nu^{N-\frac{N-2}{2}q},
$$
and hence
$$
\sigma_\nu\gtrsim \varepsilon_\nu, \qquad for \ \ N=3,  \  3<q<2^*.
$$
With $\varepsilon_\nu\lesssim \sigma_\nu$ in hand, we set 
\begin{equation}\label{e331}
\tilde w_\nu(x)=v_\nu\left((\frac{\varepsilon_\nu}{\sigma_\nu})^{\frac{N-2}{2}}x\right),\quad x\in \mathbb R^N.
%\eqno(3.31)
\end{equation}
Then 
\begin{equation}\label{e332}
\tilde w_\nu(x)=(\frac{\varepsilon_\nu}{\sigma_\nu})^{-\frac{N-2}{2}}w_\nu\left((\frac{\varepsilon_\nu}{\sigma_\nu})^{\frac{N-4}{2}}x\right), 
\quad x\in \mathbb R^N.
%\eqno(3.32)
\end{equation}
and  $w=\tilde w_\nu$ satisfies 
$$
-\Delta w+\lambda_\nu \sigma_\nu^2 (\frac{\varepsilon_\nu}{\sigma_\nu})^{N-2}w=\nu\sigma_\nu^{N-\frac{N-2}{2}q}(\frac{\varepsilon_\nu}{\sigma_\nu})^{N-2}|w|^{p-2}w+(\frac{\varepsilon_\nu}{\sigma_\nu})^{\frac{(N-2)(2+\alpha)}{2}}(I_\alpha\ast |w|^{2_\alpha^*})|w^{2_\alpha^*-2}w.
$$
Since $\lambda_\nu\lesssim \nu$ and $\varepsilon_\nu/\sigma_\nu\lesssim 1$, adopting the Moser iteration, by \eqref{e327} and Proposition \ref{p51} 
with $\beta_\nu=(\frac{\varepsilon_\nu}{\sigma_\nu})^{\frac{N-2}{2}}$ and $\gamma_\nu=\beta_\nu^{2+\alpha}$,  we obtain 
\begin{equation}\label{e333}
\tilde w_\nu(x)\lesssim (1+|x|)^{-(N-2)}, \qquad x\in \mathbb R^N.
%\eqno(3.33)
\end{equation}
From\eqref{e332}, \eqref{e333} and the fact that $3<q<2^*$, it follows that
\begin{equation}\label{e334}
\|w_\nu\|^q_q=(\frac{\varepsilon_\nu}{\sigma_\nu})^{\frac{N-2}{2}q-\frac{N(4-N)}{2}}\|\tilde w_\nu\|_q^q\lesssim (\frac{\varepsilon_\nu}{\sigma_\nu})^{\frac{N-2}{2}q-\frac{N(4-N)}{2}}.
\end{equation}
Hence, by \eqref{e330} and \eqref{e334}, we get
$$
\sigma_\nu^{N-\frac{N-2}{2}q}\|v_\nu\|_q^q=\varepsilon_\nu^{N-\frac{N-2}{2}q}\|w_\nu\|^q_q
\lesssim\varepsilon_\nu^{N-\frac{N-2}{2}q}(\frac{\varepsilon_\nu}{\sigma_\nu})^{\frac{N-2}{2}q-\frac{N(4-N)}{2}},
$$ 
which together with the fact that $\|v_\nu\|_q^q\gtrsim 1$  implies that
$$
\sigma_\nu^{\frac{N(N-2)}{2}}\lesssim \varepsilon_\nu^{\frac{N(N-2)}{2}}.
$$
That is, $\sigma_\nu\lesssim \varepsilon_\nu$ as $\nu\to 0$. Thus, $\sigma_\nu\sim \varepsilon_\nu$ as $\nu\to 0$ for $3<q<2^*$.

\noindent{\bf Case 2: $N=3$ and $2<q<3$.}  In this case,  by \eqref{e322}, \eqref{e326}  and Lemma \ref{l310}, we have
$$
\lambda_\nu\sim \lambda_\nu a^2=\lambda_\nu \|u_\nu\|_2^2\sim \nu\|u_\nu\|_q^q\gtrsim \nu \varepsilon_\nu^{\frac{3}{2}q-3},
\quad 
\lambda_\nu\sim (\nu\sigma_\nu^{\frac{q}{2}})^{\frac{2}{5-q}},
$$
it follows that 
$$
\sigma_\nu^{\frac{q}{5-q}}\sim \nu^{-\frac{2}{5-q}}\lambda_\nu\gtrsim \nu^{\frac{3-q}{5-q}}\varepsilon_\nu^{\frac{3}{2}q-3}.
$$
By \eqref{e321}, we have $\nu\sim \varepsilon_\nu^{5-\frac{3}{2}q}$, and hence
$$
\sigma_\nu^{\frac{q}{5-q}}\gtrsim \varepsilon_\nu^{(5-\frac{3}{2}q)\frac{3-q}{5-q}}\varepsilon_\nu^{\frac{3}{2}q-3}= \varepsilon_\nu^{\frac{q}{5-q}},
$$
which yields $\sigma_\nu\gtrsim \varepsilon_\nu$ as $\nu\to 0$.

On the other hand, since $-\Delta \tilde w_\nu+\lambda_\nu\sigma_\nu\varepsilon_\nu\tilde w_\nu>0$,  as in Lemma \ref{l33}, we have 
\begin{equation}\label{e335}
\tilde w_\nu(x)\gtrsim |x|^{-1}\exp(-(\lambda_\nu\sigma_\nu\varepsilon_\nu)^{\frac{1}{2}} |x|), \qquad |x|\ge 1.
%\eqno(3.35)
\end{equation}
Therefore, we find
$$
\|\tilde w_\nu\|_q^q\gtrsim \int_1^{(\lambda_\nu\sigma_\nu\varepsilon_\nu)^{-\frac{1}{2}}}r^{2-q}\exp(-q(\lambda_\nu\sigma_\nu\varepsilon_\nu)^{\frac{1}{2}} r)dr
\gtrsim \left.\frac{1}{3-q}r^{3-q}\right|_1^{(\lambda_\nu\sigma_\nu\varepsilon_\nu)^{-\frac{1}{2}}}\sim (\lambda_\nu\sigma_\nu\varepsilon_\nu)^{\frac{1}{2}(q-3)}.
$$

By using the uniform decay estimate \eqref{e333} and a comparison  argument, we also have 
\begin{equation}\label{e336}
\tilde w_\nu(x)\lesssim |x|^{-1}\exp(-\frac{1}{2}(\lambda_\nu\sigma_\nu\varepsilon_\nu)^{\frac{1}{2}} |x|), \qquad |x|\gtrsim (\lambda_\nu\sigma_\nu\varepsilon_\nu)^{-\frac{1}{2}}.
%\eqno(3.36)
\end{equation}
 Therefore, by \eqref{e333} and  \eqref{e336},  we have 
$$
\|\tilde w_\nu\|_q^q\lesssim \int_1^{(\lambda_\nu\sigma_\nu\varepsilon_\nu)^{-\frac{1}{2}}}r^{2-q}+\int_{(\lambda_\nu\sigma_\nu\varepsilon_\nu)^{-\frac{1}{2}}}^{+\infty}r^{2-q}\exp(-\frac{q}{2}(\lambda_\nu\sigma_\nu\varepsilon_\nu)^{\frac{1}{2}} r)dr
\sim  (\lambda_\nu\sigma_\nu\varepsilon_\nu)^{\frac{1}{2}(q-3)}.
$$
Thus we conclude that 
\begin{equation}\label{e337}
\|\tilde w_\nu\|_q^q\sim( \lambda_\nu\sigma_\nu\varepsilon_\nu)^{\frac{1}{2}(q-3)}.
%\eqno(3.37)
\end{equation}
By \eqref{e322},  \eqref{e330},  \eqref{e334} and \eqref{e337}, it follows that 
$$
\varepsilon_\nu^{\frac{3}{2}q-3}\lesssim \|u_\nu\|_q^q=\varepsilon_\nu^{3-\frac{q}{2}}\|w_\nu\|_q^q
=\varepsilon_\nu^{\frac{3}{2}}\sigma_\nu^{-\frac{1}{2}(q-3)}\|\tilde w_\nu\|_q^q\sim \lambda_\nu^{\frac{1}{2}(q-3)}\varepsilon_\nu^{\frac{q}{2}},
$$
which implies that $\lambda_\nu\lesssim \varepsilon_\nu^2$.  By \eqref{e321}, we have $\nu\sim \varepsilon_\nu^{\frac{10-3q}{2}}$, and then it follows from Lemma \ref{l310} that
$$
\lambda_\nu\sim (\nu\sigma_\nu^{\frac{q}{2}})^{\frac{2}{5-q}}\sim \varepsilon_\nu^{\frac{10-3q}{5-q}}\sigma_\nu^{\frac{q}{5-q}}\lesssim \varepsilon_\nu^2,
$$
which implies that $\sigma_\nu^{\frac{q}{5-q}}\lesssim \varepsilon_\nu^{\frac{q}{5-q}}$.   Thus, we also have $\sigma_\nu\sim \varepsilon_\nu$ as $\nu\to 0$ in the case of $N=3$ and $2<q<3.$

\noindent{\bf Case 3: $N=3$ and $q=3$.} In this case,  by \eqref{e322},  \eqref{e326} and  Lemma \ref{l310}, we have 
$$
\lambda_\nu\sim \nu\sigma_\nu^{\frac{3}{2}}\ln(\frac{1}{\sqrt{\lambda_\nu}\sigma_\nu}), \quad \lambda_\nu\sim \nu\|u_\nu\|_q^q\gtrsim \nu\varepsilon_\nu^{\frac{3}{2}}\ln(\frac{1}{\varepsilon_\nu}).
$$
Therefore, it follows that
\begin{equation}\label{e338}
\sigma_\nu^{\frac{3}{2}}\ln(\frac{1}{\sqrt{\lambda_\nu}\sigma_\nu})\gtrsim \varepsilon_\nu^{\frac{3}{2}}\ln(\frac{1}{\varepsilon_\nu}).
%\eqno(3.38)
\end{equation}
That is,
\begin{equation}\label{e339}
\frac{1}{2}\sigma_\nu^{\frac{3}{2}}\ln\frac{1}{\lambda_\nu}+\sigma_\nu^{\frac{3}{2}}\ln\frac{1}{\sigma_\nu}\gtrsim \varepsilon_\nu^{\frac{3}{2}}\ln\frac{1}{\varepsilon_\nu}.
%\eqno(3.39)
\end{equation}
Since $\varepsilon_\nu^{\frac{1}{2}}\sim \nu\ln\frac{1}{\varepsilon_\nu}$ by  \eqref{e321}, we have 
\begin{equation}\label{e340}
\lambda_\nu\sim \nu\sigma_\nu^{\frac{3}{2}}\ln(\frac{1}{\sqrt{\lambda_\nu}\sigma_\nu})\gtrsim \nu\varepsilon_\nu^{\frac{3}{2}}\ln(\frac{1}{\varepsilon_\nu})\sim  \varepsilon_\nu^2,
%\eqno(3.40)
\end{equation}
and hence
$$
\ln(\frac{1}{\lambda_\nu})\lesssim \ln(\frac{1}{\varepsilon_\nu^2})\sim \ln(\frac{1}{\varepsilon_\nu}),
$$ 
and hence, it follows from \eqref{e339} that
\begin{equation}\label{e341}
\frac{1}{2}\sigma_\nu^{\frac{3}{2}}\ln\frac{1}{\varepsilon_\nu}+\sigma_\nu^{\frac{3}{2}}\ln\frac{1}{\sigma_\nu}\gtrsim \varepsilon_\nu^{\frac{3}{2}}\ln\frac{1}{\varepsilon_\nu}.
%\eqno(3.41)
\end{equation}
By the above inequality,  one of the following two cases must occur:

(a)   \  $\sigma_\nu^{\frac{3}{2}}\ln\frac{1}{\varepsilon_\nu}\gtrsim \varepsilon_\nu^{\frac{3}{2}}\ln\frac{1}{\varepsilon_\nu}$, \qquad 
(b) \   $\sigma_\nu^{\frac{3}{2}}\ln\frac{1}{\sigma_\nu}\gtrsim \varepsilon_\nu^{\frac{3}{2}}\ln\frac{1}{\varepsilon_\nu}$.

If (a) holds, then $\varepsilon_\nu\lesssim \sigma_\nu$ as $\nu\to 0$, we are done. If (b) holds,  we claim that $\varepsilon_\nu\lesssim \sigma_\nu$ as $\nu\to 0$. 
Otherwise, there exists a subsequence $\nu_n\to 0$ such that $\varepsilon_{\nu_n}/\sigma_{\nu_n}\to \infty$ as $n\to \infty$.
Since $\varepsilon_\nu^{3/2}\ln (\frac{1}{\varepsilon_\nu})\le C\sigma_\nu^{3/2}\ln(\frac{1}{\sigma_\nu})$ for some $C>0$ and the function $x\ln\frac{1}{x}$ is increasing on $(0,e^{-1})$, we have 
$$
(\frac{\varepsilon_{\nu_n}}{\sigma_{\nu_n}})^{1/2}\le C\frac{\sigma_{\nu_n}\ln(\frac{1}{\sigma_{\nu_n}})}{\varepsilon_{\nu_n}\ln(\frac{1}{\varepsilon_{\nu_n}})}\le C.
$$
This contradicts to the fact that $\varepsilon_{\nu_n}/\sigma_{\nu_n}\to \infty$ and establish the claim. 

With $\varepsilon_\nu\lesssim \sigma_\nu$ in hand, adopting the Moser iteration, by \eqref{e327}  and Proposition \ref{p51}, and arguing as above, we obtain
$$
\|\tilde w_\nu\|_q^q\gtrsim \int_1^{(\lambda_\nu\sigma_\nu\varepsilon_\nu)^{-\frac{1}{2}}}r^{-1}\exp(-3(\lambda_\nu\sigma_\nu\varepsilon_\nu)^{\frac{1}{2}} r)dr
\gtrsim \left.\ln r\right|_1^{(\lambda_\nu\sigma_\nu\varepsilon_\nu)^{-\frac{1}{2}}}\sim \ln(\frac{1}{\lambda_\nu\sigma_\nu\varepsilon_\nu}),
$$
and 
$$
\|\tilde w_\nu\|_q^q\lesssim \int_1^{(\lambda_\nu\sigma_\nu\varepsilon_\nu)^{-\frac{1}{2}}}r^{-1}+\int_{(\lambda_\nu\sigma_\nu\varepsilon_\nu)^{-\frac{1}{2}}}^{+\infty}r^{-1}\exp(-\frac{3}{2}(\lambda_\nu\sigma_\nu\varepsilon_\nu)^{\frac{1}{2}} r)dr
\sim  \ln(\frac{1}{\lambda_\nu\sigma_\nu\varepsilon_\nu}).
$$
Thus, noting that $\varepsilon_\nu\lesssim\sigma_\nu$, we conclude that 
$$
\|\tilde w_\nu\|_q^q\sim\ln ( \frac{1}{\lambda_\nu\sigma_\nu\varepsilon_\nu})\lesssim \ln(\frac{1}{\lambda_\nu\varepsilon_\nu^2})\sim \ln(\frac{1}{\sqrt{\lambda_\nu}\varepsilon_\nu}).
$$
Thus, by Lemma \ref{l39},  \eqref{e330} and \eqref{e334},  we get
$$
\sigma_\nu^{\frac{3}{2}}\ln\left(\frac{1}{\sqrt{\lambda_\nu}\sigma_\nu}\right)\sim \sigma_\nu^{\frac{3}{2}}\|v_\nu\|_q^q=\|u_\nu\|_q^q=\varepsilon_\nu^{\frac{3}{2}}\|w_\nu\|_q^q=\varepsilon_\nu^{\frac{3}{2}}\|\tilde w_\nu\|_q^q\lesssim \varepsilon_\nu^{\frac{3}{2}}\ln\left(\frac{1}{\sqrt{\lambda_\nu}\varepsilon_\nu}\right).
$$
Therefore, there exists some $C>0$ such that
$$
\frac{\sigma_\nu^{3/2}\ln(\frac{1}{\sqrt{\lambda_\nu}\sigma_\nu})}{\varepsilon_\nu^{3/2}\ln(\frac{1}{\sqrt{\lambda_\nu}\varepsilon_\nu})}\le C.
$$
Suppose that there exists a subsequence $\nu_n\to 0$  such that $\sigma_{\nu_n}/\varepsilon_{\nu_n}\to +\infty$ as $n\to \infty$. Then, noting that $\lambda_\nu=O(\nu)$ as $\nu\to 0$,  we obtain
$$
(\frac{\sigma_{\nu_n}}{\varepsilon_{\nu_n}})^{1/2}\le C\frac{\varepsilon_{\nu_n}\ln(\frac{1}{\sqrt{\lambda_{\nu_n}}\varepsilon_{\nu_n}})}{\sigma_{\nu_n}\ln(\frac{1}{\sqrt{\lambda_{\nu_n}}\sigma_{\nu_n}})}=C\frac{\sqrt{\lambda_{\nu_n}}\varepsilon_{\nu_n}\ln(\frac{1}{\sqrt{\lambda_{\nu_n}}\varepsilon_{\nu_n}})}{\sqrt{\lambda_{\nu_n}}\sigma_{\nu_n}\ln(\frac{1}{\sqrt{\lambda_{\nu_n}}\sigma_{\nu_n}})}\le C.
$$
This is a contradiction. Therefore, there exists a constant $C>0$ such that $\sigma_\nu/\varepsilon_\nu\le C$ for small $\nu>0$, that is $\sigma_\nu\lesssim \varepsilon_\nu$ as $\nu\to 0$.

Thus, we conclude that $\sigma_\nu\sim \varepsilon_\nu$ as $\nu\to 0$ in the case of $N=3$ and $q=3$.

\vskip 3mm 

 Summing the above discussion, we obtain the following result.  

\begin{lemma} \label{l311}  The following estimates hold true
\begin{equation}\label{e342}
\lambda_\nu\sim \left\{\begin{array}{rcl}
\nu^{\frac{4}{q-2}}, \quad\quad \ &if& \    N=3, \ 3<q<6,\\
\varepsilon_\nu^2, \quad\qquad \  &if & \   N=3, \ q=3,\\
\nu^{\frac{4}{10-3q}}, \quad \   \ &if& \  N=3, \  2<q<3,
\end{array}\right.
%\eqno(3.42)
\end{equation}
where $\varepsilon_\nu>0$ is such that $\nu\ln \frac{1}{\varepsilon_\nu}\sim \varepsilon_\nu^{\frac{1}{2}}$ as $\nu\to 0$.
\end{lemma}
\begin{proof} We only need to show that $\lambda_\nu\sim \varepsilon_\nu^2$ as $\nu\to 0$ if $N=3$ and $q=3$.  Noting that $\sigma_\nu\sim \varepsilon_\nu$ and $\lambda_\nu\gtrsim \varepsilon_\nu^2$ as $\nu\to 0$, there exists $C>0$ such that $\sqrt{\lambda_\nu}\sigma_\nu\ge C\varepsilon_\nu^2$ for small $\nu>0$, therefore, by Lemma \ref{l310} and \eqref{e321}, we get
$$
\lambda_\nu\sim \nu\sigma_\nu^{\frac{3}{2}}\ln(\frac{1}{\sqrt{\lambda_\nu}\sigma_\nu})\le \nu\sigma_\nu^{\frac{3}{2}}\ln(\frac{1}{C\varepsilon_\nu^2})=\nu\sigma_\nu^{\frac{3}{2}}\ln(\frac{1}{C})+2\nu\sigma_\nu^{\frac{3}{2}}\ln(\frac{1}{\varepsilon_\nu})\lesssim \nu\sigma_\nu^{\frac{3}{2}}\ln(\frac{1}{\varepsilon_\nu})\sim \varepsilon_\nu^2.
$$
This completes the proof.
\end{proof}

\begin{proof} [Proof of Theorem \ref{t11} (completed)]
Assume that $N=3$ and $u_\nu$ is a mountain pass type solution of  \eqref{e39}. Let
$$
\tilde v_\lambda(x)=\lambda_\nu^{-\frac{N-2}{4}}u_\nu(\lambda_\nu^{-\frac{1}{2}}x).
$$
Then  $v=\tilde v_\lambda$ satisfies $(Q_\lambda)$ with $\lambda=\nu\lambda_\nu^{-\frac{2N-q(N-2)}{4}}$. That is, 
$$
-\Delta \tilde v_\lambda +\tilde v_\lambda=(I_\alpha\ast|\tilde v_\lambda|^{2_\alpha^*})|\tilde v_\lambda|^{2_\alpha^*-2}\tilde v_\lambda+\nu\lambda_\nu^{-\frac{2N-q(N-2)}{4}}|\tilde v_\lambda|^{q-2}\tilde v_\lambda.
$$
By Lemma \ref{l311}, we obtain 
\begin{equation}\label{e343}
\lambda=\nu\lambda_\nu^{-\frac{6-q}{4}}\sim \left\{\begin{array}{rcl}
\nu^{-\frac{2(4-q)}{q-2}}, \quad\quad \  \ &if& \     3<q<6,\\
\nu\varepsilon_\nu^{-\frac{3}{2}}
%\textcolor{red}{=\nu^{-2}}, 
 \quad \quad \quad \  &if & \   q=3,\\
\nu^{-\frac{2(q-2)}{10-3q}}, \qquad \  \  &if& \  2<q<3,
\end{array}\right.
%\eqno(3.43)
\end{equation}
Note that  by \eqref{e321}, we have 
$$
\nu\varepsilon_\nu^{-\frac{3}{2}}=\nu\ln(\frac{1}{\varepsilon_\nu})\frac{\varepsilon_\nu^{-\frac{3}{2}}}{-\ln\varepsilon_\nu}\sim\frac{1}{-\varepsilon_\nu\ln(\varepsilon_\nu)},
$$
it follows from \eqref{e343} that 
$\lim_{\nu\to 0}\lambda=\lim_{\nu\to 0}\nu\lambda_\nu^{-\frac{6-q}{4}}=+\infty$ for  $q\in (2,4).$
Moreover, we have 
$$
E_\lambda(\tilde v_\lambda)=E_{a,\nu}(u_\nu)= \left\{\begin{array}{rcl}
\mathcal E_\nu>0,\qquad\qquad \qquad\qquad  &if& \ \ q\in (2,2+\frac{4}{N}),\\
\frac{2+\alpha}{2(N+\alpha)}S_\alpha^{\frac{N+\alpha}{2+\alpha}}+O(\nu a^{q(1-\gamma_q)}), \  &if&  q\in [2+\frac{4}{N}, 2^*).
\end{array}\right.
$$
Let $v_\lambda\in H^1(\mathbb R^N)$ be the ground state solution of $(Q_\lambda)$,   then it follows from Theorem \ref{t29} and \eqref{e343} that
$$
\begin{array}{rcl}
E_\lambda(v_\lambda)&=&m_\lambda-\frac{1}{2}\|v_\lambda\|_2^2=\lambda^{-\frac{2}{q-2}}\left\{\frac{N(q-2)-4}{4q}S_q^{\frac{q}{q-2}}+O(\lambda^{-\frac{2(2+\alpha)}{(q-2)(N-2)}})\right\}\\
&\sim & \left\{\begin{array}{rcl}
\nu^{-\frac{2}{q-2}}\lambda_\nu^{\frac{2N-q(N-2)}{2(q-2)}}
\quad \qquad \quad &if & q>2+\frac{4}{N},\\
O(\nu^{-\frac{2(N+\alpha)}{(N-2)(q-2)}}\lambda_\nu^{\frac{(N+\alpha)[2N-q(N-2)]}{2(N-2)(q-2)}})\quad &if & q=2+\frac{4}{N},\\
-\nu^{-\frac{2}{q-2}}\lambda_\nu^{\frac{2N-q(N-2)}{2(q-2)}}\quad \qquad \quad  &if & q<2+\frac{4}{N}
\end{array}\right.\\
&\sim & \left\{\begin{array}{rcl}
\nu^{\frac{4(4-q)}{(q-2)^2}}
\quad \qquad \quad \ &if & \frac{10}{3}<q<4,\\
O(\nu^{\frac{4(4-q)(3+\alpha)}{(q-2)^2}}), \qquad &if&  q=\frac{10}{3},\\
-\nu^{\frac{4(4-q)}{(q-2)^2}}, \qquad \ \ \quad &if & 3<q<\frac{10}{3},\\
-\nu^{-\frac{2}{q-2}}\varepsilon_\nu^{\frac{3}{q-2}},\qquad  &if & q=3,\\
-\nu^{\frac{4}{10-3q}}, \qquad \qquad &if&   2<q<3.
\end{array}\right.
\end{array}
$$
Therefore, we conclude that  $E_\lambda(\tilde v_\lambda)>E_\lambda(v_\lambda) $ for large $\lambda$,  and hence,  $\tilde v_\lambda$ and $v_\lambda$  are different  positive solutions of $(Q_\lambda)$ with $\lambda=\nu\lambda_\nu^{-\frac{2N-q(N-2)}{4}}$. This completes the proof of last part in Theorem \ref{t11}. 
\end{proof}

\vskip 3mm 

\section{Proof of Theorem \ref{t12}}\label{s5}

In this section, we always assume that  $q=2^*$, $ p\in (\frac{N+\alpha}{N}, 1+\frac{\alpha}{N-2}]$ if $N\ge 4$ and $p\in (\frac{3+\alpha}{3}, 2+\alpha]$ if $N=3.$

\subsection{Non-existence}

Firstly, we assume that $p\in (\frac{N+\alpha}{N}, 1+\frac{\alpha}{N-2}]$. Then we have 

\begin{lemma} \label{l41}   Assume $p\in (\frac{N+\alpha}{N}, 1+\frac{\alpha}{N-2}]$. Then  $m_\mu$ is nonincreasing with respect to $\mu$  and there exists $\mu_q>0$ such that $m_\mu=\frac{1}{N}S^{\frac{N}{2}}$  for $\mu\in (0, \mu_q]$ and $m_\mu <\frac{1}{N}S^{\frac{N}{2}}$ for $\mu>\mu_q$.  In particularly, $(Q_\mu)$ has no ground state solution for $\mu\in (0,\mu_q)$ and admits a ground state solution for $\mu>\mu_q$. 
\end{lemma}
\begin{proof}  Suppose for the contrary that $(Q_\mu)$ admits a ground state for all $\mu >0$. Then it follows from Lemma \ref{l23}   that $\{v_\mu\}$ is bounded in $H^1(\mathbb R^N)$.  Since  $v_\mu$ satisfies the corresponding Nehari and Poho\v zaev identities
$$
\int_{\mathbb R^N}|\nabla v_\mu|^2+\int_{\mathbb R^N}|v_\mu|^2=\mu\int_{\mathbb R^N}(I_\alpha\ast |v_\mu|^p)|v_\mu|^p+\int_{\mathbb R^N}|v_\mu|^{2^*},
$$
$$
\frac{N-2}{2}\int_{\mathbb R^N}|\nabla v_\mu|^2+\frac{N}{2}\int_{\mathbb R^N}|v_\mu|^2=\frac{N+\alpha}{2p}\mu\int_{\mathbb R^N}(I_\alpha\ast |v_\mu|^p)|v_\mu|^p+\frac{N}{2^*}\int_{\mathbb R^N}|v_\mu|^{2^*},
$$
it follows that
\begin{equation}\label{e41}
\|v_\mu\|_2^2=\frac{N+\alpha-p(N-2)}{2p}\mu \int_{\mathbb R^N}(I_\alpha\ast |v_\mu|^p)|v_\mu|^p.
%\eqno(4.1)
\end{equation}
By the Hardy-Littlewood-Sobolev, the Sobolev  and the interpolation inequalities,  we obtain
$$
\int_{\mathbb R^N}(I_\alpha\ast |v_\mu|^p)|v_\mu|^p\le C\|w_\mu\|_{\tilde p}^{2p}\le C\left(\int_{\mathbb R^N}|v_\mu|^2\right)^{\frac{2^*-\tilde p}{2^*-2}\frac{N+\alpha}{N}}
\left(\int_{\mathbb R^N}|v_\mu|^{2^*}\right)^{\frac{\tilde p-2}{2^*-2}\frac{N+\alpha}{N}}
$$
and 
$$
\int_{\mathbb R^N}(I_\alpha\ast |v_\mu|^p)|v_\mu|^p
\le C\|v_\mu\|_2^{\frac{2(2^*-\tilde p)}{2^*-2}\frac{N+\alpha}{N}}
\left(\frac{1}{S}\int_{\mathbb R^N}|\nabla v_\mu|^2\right)^{\frac{2^*(\tilde p-2)}{2(2^*-2)}\frac{N+\alpha}{N}},
$$
here and in what follows, $\tilde p=\frac{2Np}{N+\alpha}$. Therefore, we get
$$
1\le C\frac{N+\alpha-p(N-2)}{2p}\mu \| v_\mu\|_2^{\frac{2(2^*-\tilde p)}{2^*-2}\frac{N+\alpha}{N}-2}\left(\frac{1}{S}\int_{\mathbb R^N}|\nabla v_\mu|^2\right)^{\frac{2^*(\tilde p-2)}{2(2^*-2)}\frac{N+\alpha}{N}}.
$$
Since $p\le 1+\frac{\alpha}{N-2}$ implies that $\frac{2(2^*-\tilde p)}{2^*-2}\frac{N+\alpha}{N}-2\ge 0$, it follows from the boundedness of $ \{v_\mu\}$  in $H^1(\mathbb R^N)$ that $\mu\ge \mu_0$ for some $\mu_0>0$, this is a contradiction. 
Therefore, $(Q_\mu)$ has no ground state solution for small $\mu>0$. 

Arguing as in the proof of Theorem \ref{t11}, we prove that there exists a constant $\mu_q>0$ such that $m_\mu$ is nonincreasing with respect to $\mu$ and $m_\mu=\frac{1}{N}S^{\frac{N}{2}}$  for $\mu\in (0, \mu_0]$ and $m_\mu <\frac{1}{N}S^{\frac{N}{2}}$ for $\mu>\mu_0$.  Moreover, $(Q_\mu)$ has no ground state solution for $\mu\in (0,\mu_0)$ and admits a ground state solution for $\mu>\mu_q$.  The proof is complete. 
\end{proof}

In what follows, we assume that $N=3$ and $p\in (\max\{2,1+\alpha\}, 2+\alpha]$.
It is easy to see that under the rescaling 
\begin{equation}\label{e42}
w(x)=\mu^{\frac{N-2}{2(N-2)(p-1)-2\alpha}}v(\mu^{\frac{1}{(N-2)(p-1)-\alpha}}x), 
%\eqno(4.2)
\end{equation}
the equation $(Q_\mu)$ is reduced to 
$$
-\Delta w+\mu^\sigma w=\mu^\sigma(I_\alpha\ast |w|^p)|w|^{p-2}w+|w|^{2^*-2}w, 
\eqno(\bar Q_\mu)
$$
where $\sigma:=\frac{2}{(N-2)(p-1)-\alpha}>1$.

The associated energy functional is defined by
$$
J_\mu(w):=\frac{1}{2}\int_{\mathbb R^N}|\nabla w|^2+\mu^\sigma |w|^2-\frac{1}{2p}\mu^\sigma\int_{\mathbb R^N}(I_\alpha\ast |w|^p)|w|^p-\frac{1}{2^*}\int_{\mathbb R^N}|w|^{2^*}=I_\mu(v).
$$

We define the Nehari manifolds as follows
$$
\mathcal{N}_\mu=
\left\{w\in H^1(\mathbb R^N)\setminus\{0\} \ \left | \ \int_{\mathbb R^N}|\nabla w|^2+\mu^\sigma \int_{\mathbb R^N}|w|^2=\mu^\sigma\int_{\mathbb R^N}(I_\alpha\ast |w|^p)|w|^p+
\int_{\mathbb R^N}|w|^{2^*}\  \right. \right\}
$$
and 
$$
\mathcal{N}_0=
\left\{w\in D^{1,2}(\mathbb R^N)\setminus\{0\} \ \left | \ \int_{\mathbb R^N}|\nabla w|^2=\int_{\mathbb R^N}|w|^{2^*}\  \right. \right\}. 
$$
Then 
$$
m_\mu:=\inf_{w\in \mathcal {N}_\mu}J_\mu(w) \quad {\rm and} \quad 
m_0:=\inf_{w\in \mathcal {N}_0}J_0(w)
$$
are well-defined and positive. Moreover, $m_0$ is attained on $\mathcal N_0$.

Define the Poho\v zaev manifold as follows
$$
\mathcal P_\mu:=\{w\in H^1(\mathbb R^N)\setminus\{0\} \   | \  P_\mu(w)=0 \},
$$
where 
$$
\begin{array}{rcl}
P_\mu(w):&=&\frac{N-2}{2}\int_{\mathbb R^N}|\nabla w|^2+\frac{ N}{2}\mu^\sigma \int_{\mathbb R^N}|w|^2
\\ \\
&\quad &-\frac{N+\alpha}{2p}\mu^\sigma\int_{\mathbb R^N}(I_\alpha\ast |w|^p)|w|^p-\frac{N}{2^*}\int_{\mathbb R^N}|w|^{2^*}.
\end{array}
$$
Then by Lemma \ref{l21}, $w_\mu\in \mathcal P_\mu$. Moreover,  we have a minimax characterizations for the least energy level $m_\mu$ as in Lemma \ref{l22}.

It has been shown in \cite{Ma-2} that there exists $\xi_\mu>0$ satisfying $\xi_\mu\to 0$ such that the rescaled family 
$\tilde w_\mu(x)=\xi_\mu^{\frac{N-2}{2}} w_\mu(\xi_\mu x)$ satisfies 
\begin{equation}\label{e43}
\|\nabla(\tilde w_\mu-W_1)\|_2\to 0, \qquad \|\tilde w_\mu-W_1\|_{2^*}\to 0,   \qquad {\rm as} \  \ \mu\to 0.
%\eqno(4.3)
\end{equation}
Under the rescaling 
\begin{equation}\label{e44}
\tilde w(x)=\xi_\mu^{\frac{N-2}{2}} w(\xi_\mu x), 
%\eqno(4.4)
\end{equation}
the equation $(\bar Q_\mu)$ reads as  
$$
-\Delta \tilde w+\mu^\sigma\xi_\mu^{2}\tilde w=\mu^\sigma \xi_\mu^{N+\alpha-p(N-2)}(I_\alpha\ast |\tilde w|^p)\tilde w^{p-1}+\tilde w^{2^*-1}.
\eqno(R_\mu)
$$
The  corresponding energy functional is given by
$$
\tilde J_\mu(\tilde w):=\frac{1}{2}\int_{\mathbb R^N}|\nabla\tilde  w|^2+\mu^\sigma\xi_\mu^2|\tilde w|^2-\frac{1}{2p}\mu^\sigma \xi_\mu^{N+\alpha-p(N-2)}\int_{\mathbb R^N}(I_\alpha\ast |\tilde w|^{{p}})|\tilde w|^{p}-\frac{1}{2^*}\int_{\mathbb R^N}|\tilde w|^{2^*}.
$$
Clearly, we have $\tilde J_\mu(\tilde w)=J_\mu(w)=I_\mu(v)$. Furthermore, we  have the following lemma \cite{Ma-2}.

\begin{lemma}\label{l42}    Let $v,w,\tilde w\in H^1(\mathbb R^N)$ satisfy \eqref{e42} and \eqref{e44}, then the following statements hold true

(1) \ $ \ \|\nabla \tilde w\|_2^2= \|\nabla w\|_{2}^{2}=\|\nabla v\|_{2}^{2}, \  \|\tilde w\|^{2^*}_{2^*}=\|w\|_{2^*}^{2^*}=\|v\|_{2^*}^{2^*},$

(2) \ $\xi_\mu^2\|\tilde w\|^2_2=\|w\|_2^2=\mu^{-\sigma}\| v\|_2^2,$

(3)  \ $\xi_\mu^{(N+\alpha)-p(N-2)}\int_{\mathbb R^N}(I_\alpha\ast |\tilde w|^p)|\tilde w|^p=\int_{\mathbb R^N}(I_\alpha\ast |w|^p)|w|^p=\mu^{1-\sigma} \int_{\mathbb R^N}(I_\alpha\ast |v|^p)|v|^p$.
\end{lemma}

Note that the corresponding Nehari and Poho\v zaev's identities are as follows
$$
\int_{\mathbb R^N}|\nabla \tilde w_\mu|^2+\mu^\sigma\xi_\mu^2\int_{\mathbb R^N}|\tilde w_\mu|^2=\mu^\sigma \xi_\mu^{N+\alpha-p(N-2)}\int_{\mathbb R^N}(I_\alpha\ast |\tilde w_\mu|^p)|\tilde w_\mu|^p+\int_{\mathbb R^N}|\tilde w_\mu|^{2^*}
$$
and 
$$
\frac{1}{2^*}\int_{\mathbb R^N}|\nabla \tilde w_\mu|^2+\frac{1}{2}
\mu^\sigma\xi_\mu^2\int_{\mathbb R^N}|\tilde w_\mu|^2=\frac{N+\alpha}{2Np}\mu^\sigma\xi_\mu^{N+\alpha-p(N-2)}\int_{\mathbb R^N}(I_\alpha\ast |\tilde w_\mu|^p)|\tilde w_\mu|^p+\frac{1}{2^*}\int_{\mathbb R^N}|\tilde w_\mu|^{2^*},
$$
it follows that 
$$
\left(\frac{1}{2}-\frac{1}{2^*}\right)\mu^\sigma\xi_\mu^2\int_{\mathbb R^N}|\tilde w_\mu|^2=\left(\frac{N+\alpha}{2Np}-\frac{1}{2^*}\right)\mu^\sigma\xi_\mu^{N+\alpha-p(N-2)}\int_{\mathbb R^N}(I_\alpha\ast |\tilde w_\mu|^p)|\tilde w_\mu|^p.
$$
Thus, we obtain
\begin{equation}\label{e45}
\xi_\mu^{(N-2)(p-1)-\alpha}\int_{\mathbb R^N}|\tilde w_\mu|^2=\frac{2(2^*-\tilde p)}{\tilde p(2^*-2)}\int_{\mathbb R^N}(I_\alpha\ast |\tilde w_\mu|^p)|\tilde w_\mu|^p,
%\eqno(4.5)
\end{equation}
where $\tilde p=\frac{2Np}{N+\alpha}\in (2,2^*)$.

To control the norm $\|\tilde w_\mu\|_2$, we note that  for any $\mu>0$, $\tilde w_\mu>0$ satisfies the linear inequality
\begin{equation}\label{e46}
-\Delta \tilde w_\mu+\mu^\sigma\xi_\mu^2\tilde w_\mu
%=\mu^\sigma \xi_\mu^{(N+\alpha)-p(N-2)}(I_\alpha\ast  |\tilde w_\mu|^p)|\tilde w_\mu|^{p-1}+\tilde w_\mu^{2^*-1}
>0,  \qquad x\in \mathbb R^N.
%\eqno(4.6)
\end{equation}

\begin{lemma}\label{l43}  There exists a constant $c>0$ such that 
\begin{equation}\label{e47}
\tilde w_\mu(x)\ge c|x|^{-(N-2)}\exp(-\mu^{\frac{\sigma}{2}}\xi_\mu |x|),  \quad |x|\ge 1.
%\eqno(4.7)
\end{equation}
\end{lemma}
The proof of the above lemma  is similar to that  in \cite[Lemma 4.8]{Moroz-1}.  

Adopting the Moser iteration, by  Proposition \ref{p54}
with $\lambda_\nu=\nu$ and $\sigma_\nu=\beta_\nu=\gamma_\nu=1$,we obtain   the best decay estimation of $\tilde w_\mu$.

\begin{lemma}\label{l44}   Assume $N=3$ and 
\begin{equation}\label{e48}
\max\left\{2, 1+\alpha\right\}<p<3+\alpha.
%\eqno(4.8)
\end{equation}
Then there exists a constant $C>0$ such that for small $\mu>0$, there holds 
\begin{equation}\label{e49}
\tilde w_\mu(x)\le C(1+|x|)^{-(N-2)}, \qquad x\in \mathbb R^N.
%\eqno(4.9)
\end{equation}
\end{lemma}

\begin{proof}[Proof of Theorem  \ref{t12}]  By virtue of Lemma \ref{l41}, we only need to prove the first part of Theorem \ref{t12} in the case of $N=3$ and $\max\{2,1+\alpha\}<p\le 2+\alpha$. Suppose for the contrary that $(Q_\mu)$ has a ground state $v_\mu$ for all small $\mu>0$. We define $w_\mu$ and $\tilde w_\mu$ as above. Since $\max\{2,1+\alpha\}<p<3+\alpha$, we have $p>\frac{3}{2}+\frac{\alpha}{2}$. Therefore, it follows from Lemma \ref{l44} that 
 $\int_{\mathbb R^N}(I_\alpha\ast |\tilde w_\mu|^p)|\tilde w_\mu|^p\lesssim 1$.
On the other hand, by \eqref{e45} and Lemma \ref{l43},  we have 
$$
\int_{\mathbb R^N}(I_\alpha\ast |\tilde w_\mu|^p)|\tilde w_\mu|^p\gtrsim \xi_\mu^{p-1-\alpha}\|\tilde w_\mu\|_2^2\gtrsim \mu^{-\sigma/2}\xi_\mu^{p-2-\alpha}.
$$
Thus, we get
$$
1\gtrsim \mu^{-\sigma/2}\xi_\mu^{p-2-\alpha},
$$
which together with the fact $1+\alpha<p\le 2+\alpha$ yields a contradiction. This prove that $(Q_\mu)$ has no ground state solution for small $\mu>0$. Arguing as in the proof of Theorem  \ref{t11}, we obtain that $m_\mu$ is nonincreasing with  respect to $\mu$ and there exists a constant $\mu_q>0$ such that $(Q_\mu)$ has no ground state solution for $\mu \in (0,\mu_q)$ and admits a ground state solution for $\mu>\mu_q$. Moreover,  $m_\mu=\frac{1}{3}S_\alpha^{\frac{3}{2}}$ for $\mu\in (0,\mu_q)$ and 
$m_\mu<\frac{1}{3}S_\alpha^{\frac{3}{2}}$ for $\mu>\mu_q$. This completes the proof of first part of Theorem \ref{t12}.
 \end{proof}

\vskip 3mm

\subsection{Multiplicity}

In order to obtain the second positive solution of $(Q_\mu)$, we  search for solutions to $(P_\varepsilon)$ having prescribed mass, and in this case $\varepsilon\in \mathbb R$ is part of the unknown.  That is, for a fixed $a>0$,  search for $u\in H^1(\mathbb R^N)$ and $\lambda\in \mathbb R$ satisfying
\begin{equation}\label{e410}
\left\{\begin{array}{rl}
&-\Delta u+\lambda u=\nu(I_\alpha\ast |u|^p)|u|^{p-2}u+|u|^{2^*-2}u, \  \ in  \  \mathbb R^N,\\
&u\in H^1(\mathbb R^N),  \   \    \   \int_{\mathbb R^N}|u|^2=a^2,
\end{array}\right.
%\eqno(4.10)
\end{equation}
where $\nu>0$ is a new parameter. The solution of \eqref{e410} is usually denoted by a pair $(u,\lambda)$ and  called a normalized solution.

%This approach seems to be particularly meaningful 
%from the physical point of view, and often offers a good insight of the dynamical properties of the standing--wave solutions for (1.1), such as stability or instability.

  It is standard to check that the Energy functional 
\begin{equation}\label{e411}
E_{a,\nu}(u)=\frac{1}{2}\int_{\mathbb R^N}|\nabla u|^2-\frac{\nu}{2p}\int_{\mathbb R^N}(I_\alpha\ast |u|^p)|u|^p-\frac{1}{2^*}\int_{\mathbb R^N}|u|^{2^*}
%\eqno(4.11)
\end{equation}
is of class $C^1$ and that a critical point of $E$ restricted to the (mass) constraint 
\begin{equation}\label{e412}
S_a=\{ u\in H^1(\mathbb R^N): \ \|u\|_2^2=a^2\}
%\eqno(4.12)
\end{equation}
gives a solution to \eqref{e410}.  Here   $\lambda\in \mathbb R$ arises as a Lagrange multiplier. 
Let 
$$
P_\nu(u)=\int_{\mathbb R^N}|\nabla u|^2-\nu\eta_p\int_{\mathbb R^N}(I_\alpha\ast |u|^p)|u|^p-\int_{\mathbb R^N}|u|^{2^*},
$$
where $\eta_p:=\frac{Np-N-\alpha}{2p}\in (0,1)$. 
Define 
$$
\mathcal{P}_{a,\nu}:=\{u\in S_a \  | \ P_\nu(u)=0\},
$$
$$
\mathcal{P}_+^{a,\nu}:=\{u\in S_a\cap \mathcal{P}_{a,\nu} \ | \ 2\|\nabla u\|_2^2>2\nu p\eta_p^2\int_{\mathbb R^N}(I_\alpha\ast |u|^p)|u|^p+2^*\|u\|_{2^*}^{2^*}\},
$$
$$
\mathcal{P}_0^{a,\nu}:=\{u\in S_a\cap \mathcal{P}_{a,\nu} \ | \ 2\|\nabla u\|_2^2=2\nu p\eta_p^2\int_{\mathbb R^N}(I_\alpha\ast |u|^p)|u|^p+2^*\|u\|_{2^*}^{2^*}\},
$$
$$
\mathcal{P}_-^{a,\nu}:=\{u\in S_a\cap \mathcal{P}_{a,\nu} \ | \ 2\|\nabla u\|_2^2<2\nu p\eta_p^2\int_{\mathbb R^N}(I_\alpha\ast |u|^p)|u|^p+2^*\|u\|_{2^*}^{2^*}\}.
$$
Put
$$
m_\nu:=\inf_{u\in \mathcal{P}_+^{a,\nu}}E_{a,\nu}(u), \qquad \mathcal E_\nu:=\inf_{u\in \mathcal{P}_-^{a,\nu}}E_{a,\nu}(u).
$$
By the Hardy-Littlewood-Sobolev inequality and Gagliardo-Nirenberg inequality,   we obtain
$$
\int_{\mathbb R^N}(I_\alpha\ast |u|^p)|u|^p\le C_\alpha(N)\|u\|_{\frac{2Np}{N+\alpha}}^{2p}\le C_{\alpha}(N)C_{Np}^{2p}\|u\|_2^{2p(1-\eta_p)}\|\nabla u\|_2^{2p\eta_p}.
$$
%where 
%$$
%C_\alpha(N)=\sup\frac{\int_{\mathbb R^N}(I_\alpha\ast |u|^p)|u|^p}{\|u\|_{\frac{2Np}{N+\alpha}}^{2p}}.
%$$
Therefore, for any $u\in S_a$, we have 
$$
\begin{array}{rcl}
E_{a,\nu}(u)&\ge &\frac{1}{2}\|\nabla u\|_2^2-\frac{\nu}{2p}C_\alpha(N)C_{Np}^{2p}a^{2p(1-\eta_p)}\|\nabla u\|_2^{2p\eta_p}-\frac{1}{2^*}S^{-\frac{2^*}{2}}\|\nabla u\|_2^{2^*}\\
&=&\|\nabla u\|_2^2\left\{\frac{1}{2}-\frac{\nu}{2p}C_\alpha(N)C_{Np}^{2p}a^{2p(1-\eta_p)}\|\nabla u\|_2^{2p\eta_p-2}-\frac{1}{2^*}S^{-\frac{2^*}{2}}\|\nabla u\|_2^{2^*-2}\right\}\\
&=&\|\nabla u\|_2^2f_{a,\nu}(\|\nabla u\|_2^2),
\end{array}
$$
where 
$$
f_{a,\nu}(\rho):=\frac{1}{2}-\frac{\nu}{2p}C_\alpha(N)C_{Np}^{2p}a^{2p(1-\eta_p)}\rho^{p\eta_p-1}-\frac{1}{2^*}S^{-\frac{2^*}{2}}\rho^{\frac{2^*-2}{2}}.
$$
Since $p\in(\frac{N+\alpha}{N}, 1+\frac{2+\alpha}{N})\subset  (\frac{N+\alpha}{N}, \frac{N+\alpha}{N-2})$, we have 
$$
0<\eta_p:=\frac{N}{2}-\frac{N+\alpha}{2p}=\frac{Np-N-\alpha}{2p}<1,
$$
and 
$$
p\eta_p-1=\frac{Np-N-\alpha}{2}-1=\frac{Np-N-2-\alpha}{2}<0.
$$
Therefore, we have 
$$
f_{\nu,a}(0)=-\infty, \quad f_{\nu,a}(+\infty)=-\infty.
$$
Notice that 
$$
\begin{array}{rcl}
f'_{a,\nu}(\rho)&=&-\frac{\nu}{2p}(p\eta_p-1)C_\alpha(N)C_{Np}^{2p}a^{2p(1-\eta_p)}\rho^{p\eta_p-2}-\frac{1}{N}S^{-\frac{2^*}{2}}\rho^{\frac{4-N}{N-2}}\\
&=&\rho^{p\eta_p-2}\left[\frac{\nu}{2p}(1-p\eta_p)C_\alpha(N)C_{Np}^{2p}a^{2p(1-\eta_p)}-\frac{1}{N}S^{-\frac{N}{N-2}}\rho^{\frac{N}{N-2}-p\eta_p}\right].
\end{array}
$$
Therefore, $f'_{a,\nu}(\rho)=0$ implies that 
$$
\rho^{\frac{N}{N-2}-p\eta_p}=\frac{\nu N}{2p}S^{\frac{N}{N-2}}(1-p\eta_p)C_\alpha(N)C_{Np}^{2p}a^{2p(1-p\eta_p)}.
$$
It is easy to see that 
$$
\frac{N}{N-2}-p\eta_p>0 \quad {\rm if \ and \ only \ if } \quad p<\frac{N}{N-2}+\frac{\alpha}{N}.
$$
Therefore, if $p<1+\frac{2+\alpha}{N}<\frac{N}{N-2}+\frac{\alpha}{N}$, we obtain that 
$$
\rho_{a,\nu}=\left(\frac{N(1-p\eta_p)C_\alpha(N)C_{Np}^{2p}S^{\frac{N}{N-2}}}{2p}\nu a^{2p(1-\eta_p)}\right)^{\frac{N-2}{N-p\eta_p(N-2)}}
$$
is the unique global maximum point of $f_{\nu,a}(\rho)$ and the maximum value is 
$$
\max_{\rho>0}f_{\nu,a}(\rho)=f_{a,\nu}(\rho_{a,\nu})=\frac{1}{2}-K_p[\nu a^{2p(1-\eta_p)}]^{\frac{2}{N-p\eta_p(N-2)}}<\frac{1}{2},
$$
where 
$$
K_p:=\frac{N[\alpha-(p-1)(N-2)]+2(N-\alpha)}{2N[(2+\alpha-N(p-1)]}S^{-\frac{N}{N-2}}\cdot \left(\frac{N(1-p\eta_p)C_\alpha(N)C_{Np}^{2p}}{2p}\right)^{\frac{2}{N-p\eta_p(N-2)}}.
$$

\begin{lemma}\label{l45}   Let $N\ge 3$, $\alpha\in (0,N), \ \nu>0, a>0$, $p\in (\frac{N+\alpha}{N},1+\frac{2+\alpha}{N})\subset (\frac{N+\alpha}{N}, \frac{N}{N-2}+\frac{\alpha}{N})$, and $K$ be defined above. Then 
\begin{equation}\label{e413}
\max_{\rho>0}f_{a,\nu}(\rho)=\left\{\begin{array}{rcl} 
>0, \quad &if &  \ \nu a^{2p(1-\eta_p)}<\left(\frac{1}{2K_p}\right)^{\frac{N-p\eta_p(N-2)}{2}}, \\
=0, \quad &if & \ \nu a^{2p(1-\eta_p)}=\left(\frac{1}{2K_p}\right)^{\frac{N-p\eta_p(N-2)}{2}},\\
<0, \quad &if &  \   \nu a^{2p(1-\eta_p)}>\left(\frac{1}{2K_p}\right)^{\frac{N-p\eta_p(N-2)}{2}}.
\end{array}\right.
%\eqno(4.13)
\end{equation}
 \end{lemma}

Adopting the arguments in \cite[Theorem 1.3, Theorem 1.5]{Li-4}, we obtain the following 

\begin{theorem}\label{t46}  Let $N\ge 3$, $\alpha\in (0,N), \nu>0, a>0$,  $p\in (\frac{N+\alpha}{N},1+\frac{2+\alpha}{N})$. If 
\begin{equation}\label{e414}
\nu a^{2p(1-\eta_p)}\le \left(2K_p\right)^{-\frac{N-p\eta_p(N-2)}{2}}.
%\eqno(4.14)
 \end{equation}
 then the following statements hold true:

(1)  $E_{a,\nu} |_{S_a}$ has a critical point $\tilde u$ at negative level $m_\nu<0$, which is radially symmetric and non-increasing ground state to \eqref{e410} with some $\lambda_\nu>0$.

(2) there exists a mountain pass solution $u_\nu$ to \eqref{e410} with some $\lambda_\nu>0$, which is radially symmetric and non-increasing, and satisfies 
\begin{equation}\label{e415}
0<E_{a,\nu}(u_\nu)=\mathcal E_\nu= \inf_{u\in P_{a,-}}E_{a,\nu}(u)<m_\nu+\frac{1}{N}S^{\frac{N}{2}}.
%\eqno(4.15)
\end{equation}
In particular, $u_\nu$ is not a ground state.  
\end{theorem}

If $p\ge 1+\frac{2+\alpha}{N}$, then $f'_{a,\nu}(\rho)<0$ for all $\rho>0$ and hence
$$
\max_{\rho>0}f_{a,\nu}(\rho)=\left\{\begin{array}{rcl}
\frac{1}{2}-\frac{\nu}{2p}C_\alpha(N)C_{Np}^{2p}a^{2p(1-\eta_p)}, \quad &{\rm  if}& \  p=1+\frac{2+\alpha}{N},\\
\frac{1}{2},   \qquad \qquad \qquad \qquad\qquad  \quad & {\rm if}& \  p>1+\frac{2+\alpha}{N}.
\end{array} \right.
$$

\begin{theorem}\label{t47}  Let $N\ge 3$, $\alpha\in (0,N), \nu>0, a>0$,  $p\in [1+\frac{2+\alpha}{N}, \frac{N+\alpha}{N-2})$. If  $p=1+\frac{2+\alpha}{N}$,  we further assume that
\begin{equation}\label{e416}
\nu a^{N+\alpha-p(N-2)}<\frac{p}{C_\alpha(N)C_{Np}^{2p}},
%\eqno(4.16))
\end{equation}
then problem \eqref{e410} has a mountain pass type normalized ground state $u_\nu$, which is radially symmetric and non-increasing, and satisfies
\begin{equation}\label{e417}
0<E_{a,\nu}(u_\nu)=\mathcal E_\nu= \inf_{u\in P_{a,-}}E_{a,\nu}(u)<\frac{1}{N}S^{\frac{N}{2}}.
%\eqno(4.17)
\end{equation}
 Moreover, $u_\nu$ solves  \eqref{e410} with some $\lambda_\nu>0$, and  as $\varpi:=\nu a^{2p(1-\eta_p)}\to 0$, 
 \begin{equation}\label{e418}
 E_{a,\nu}(u_\nu)=\frac{1}{N}S^{\frac{N}{2}}(1+O(\nu a^{2p(1-\eta_p)})).
%\eqno(4.18)
 \end{equation}
 \end{theorem}
\begin{proof}  The proof of the first statement is similar to that of \cite[Theorem 1.2]{Li-4}, we only need to prove the last statement. Let $u\in S_a$ and $u^t(x)=t^{\frac{N}{2}}u(tx)$ for all $t>0$, then $u^t\in S_a$ and 
\begin{equation}\label{e419}
\begin{array}{rcl}
E_{a,\nu}(u^t)&\ge &\frac{t^2}{2}\|\nabla u\|_2^2-\frac{\nu}{2p}C_\alpha(N)C_{Np}^{2p}a^{2p(1-\eta_p)}t^{2p\eta_p}\|\nabla u\|_2^{2p\eta_p}-\frac{1}{2^*}S^{-2^*/2}t^{2^*}\|\nabla u\|_2^{2^*}\\
&=&h(\|\nabla u^t\|_2^2),
\end{array}
%\eqno(4.19)
\end{equation}
where 
$$
h(\rho):=\rho\left\{\frac{1}{2}-\frac{\nu}{2p}C_\alpha(N)C_{Np}^{2p}a^{2p(1-\eta_p)}\rho^{p\eta_p-1}-\frac{1}{2^*}S^{-2^*/2}\rho^{\frac{2^*}{2}-1}\right\}.
$$
Clearly, there exists a unique $\rho_\nu>0$ such that $h'(\rho_\nu)=0$, which  implies that
$$
\nu \eta_p C_\alpha(N)C_{Np}^{2p}a^{2p(1-\eta_p)}\rho_\nu^{p\eta_p-1}+S^{-2^*/2}\rho_\nu^{\frac{2^*}{2}-1}=1.
$$
Therefore, $\rho_\nu\le S^{\frac{2^*}{2^*-2}}=S^{\frac{N}{2}}$.  If $p\in [1+\frac{2+\alpha}{N}, 2_\alpha^*)$, then $p\eta_p\ge 1$.  Thus we get
$$
\rho_\nu=S^{\frac{N}{2}}\left(1-\nu\eta_p C_\alpha(N)C_{Np}^{2p} a^{2p(1-\eta_p)}\rho_\nu^{p\eta_p-1}\right)^{\frac{N-2}{2}}
=S^{\frac{N}{2}}(1+O(\nu a^{2p(1-\eta_p)})), 
$$
 as $ \nu a^{2p(1-\eta_p)}\to 0$, which together with \eqref{e419} implies  that
\begin{equation}\label{e420}
\begin{array}{rcl}
\mathcal E_\nu&=&\inf_{u\in S_a}\sup_{s\in (0,+\infty)}E_{a,\nu}(u_s)\\
&\ge& h(\rho_\nu)=\frac{1}{N}S^{\frac{N}{2}}(1+O(\nu a^{2p(1-\eta_p)})),
 \ for \  \ p\in [1+\frac{2+\alpha}{N}, 2_\alpha^*).
 \end{array}
%\eqno(4.20)
\end{equation}
Thus,  by \eqref{e417} and \eqref{e420}, for  $p\in [1+\frac{2+\alpha}{N}, 2_\alpha^*)$, we get 
$$
\mathcal E_\nu= \frac{1}{N}S^{\frac{N}{2}}+O(\nu a^{2p(1-\eta_p)}), \quad  as \  \  \nu a^{2p(1-\eta_p)}\to 0.
$$
The proof is complete. 
\end{proof}

\begin{lemma}\label{l48} Assume that $p\in (\frac{N+\alpha}{N}, \frac{N+\alpha}{N-2})$, then $\lambda_\nu$ is  continuous respect to $\nu\in (0,\nu_0]$ for any fixed $\nu_0>0$.
\end{lemma}
\begin{proof} Denote
$
F(\nu)=\int_{\mathbb R^N}(I_\alpha\ast|u_\nu|^p)|u_\nu|^p.
$
Then, by the Hardy-Littlewood-Sobolev inequality, it follows that
$$
\begin{array}{rcl}
F(\nu_1)-F(\nu_2)&=&\int_{\mathbb R^N}(I_\alpha\ast|u_{\nu_1}|^p)|u_{\nu_1}|^p-\int_{\mathbb R^N}(I_\alpha\ast|u_{\nu_2}|^p)|u_{\nu_2}|^p\\
&=&\int_{\mathbb R^N}(I_\alpha\ast |u_{\nu_1}|^p)(|u_{\nu_1}|^p-|u_{\nu_2}|^p)-\int_{\mathbb R^N}(I_\alpha\ast |u_{\nu_2}|^p)(|u_{\nu_1}|^p-|u_{\nu_2}|^p)\\
&=&\int_{\mathbb R^N}(I_\alpha\ast (|u_{\nu_1}|^p-|\mu_{\nu_2}|^p))(|u_{\nu_1}|^p-|u_{\nu_2}|^p)\\
&\le &C_\alpha(N)\| |u_{\nu_1}|^p-|u_{\nu_2}|^p\|_{\frac{2N}{N+\alpha}}^2,
\end{array}
$$
and hence
$$
|F(\nu_1)-F(\nu_2)|\le C_\alpha(N)\| |u_{\nu_1}|^p-|u_{\nu_2}|^p\|_{\frac{2N}{N+\alpha}}^2.
$$
Let $\tilde p=\frac{2Np}{N+\alpha}$, then 
$$
\begin{array}{rcl}
\int_{\mathbb R^N}|u_{\nu_1}-u_{\nu_2}|^{\tilde p}&\le& \left(\int_{\mathbb R^N}|u_{\nu_1}-u_{\nu_2}|^2\right)^{\frac{2^*-\tilde p}{2^*-2}}\left(\frac{1}{S}\int_{\mathbb R^N}|\nabla (u_{\nu_1}-u_{\nu_2})|^2\right)^{\frac{2^*(\tilde p-2)}{2(2^*-2)}}\\
&\le &(4a)^{\frac{2(2^*-\tilde p)}{2^*-2}}S^{-\frac{2^*(\tilde p-2)}{2(2^*-2)}}\|\nabla (u_{\nu_1}-u_{\nu_2})\|_2^{\frac{2^*(\tilde p-2)}{2^*-2}}.
\end{array}
$$
Therefore,  $\|u_{\nu}\|_{\tilde p}$ and $F(\nu)$ are continuous in $(0,\nu_0]$. 
 Since
$$
\lambda_{\nu_i}a^2=\frac{N+\alpha-p(N-2)}{2p}\mu_i\int_{\mathbb R^N}(I_\alpha\ast |u_{\nu_i}|^p)|u_{\mu_i}|^p, \quad i=1,2,
$$
we know that
$$
\begin{array}{rcl}
(\lambda_{\nu_1}-\lambda_{\nu_2})a^2&=&\frac{N+\alpha-p(N-2)}{2p}[\nu_1F(\nu_1)-\nu_2F(\nu_2)]\\
&=&\frac{N+\alpha-p(N-2)}{2p}[\nu_1(F(\nu_1)-F(\nu_2))+(\nu_1-\nu_2)F(\nu_2)].
\end{array}
$$
From which we conclude that $\lambda_\nu$ is continuous in $(0,\nu_0]$. This completes the proof.
\end{proof}

\begin{proof} [Proof of Theorem \ref{t12} (continued) ]  We first prove that last part of Theorem \ref{t12} in the case of $p\in (\frac{N+\alpha}{N}, 1+\frac{\alpha}{N-2})$.
Clearly, we have 
$$
\frac{N+\alpha}{N}< 1+\frac{2+\alpha}{N}\le1+\frac{\alpha}{N-2}<\frac{N}{N-2}+\frac{\alpha}{N}<\frac{N+\alpha}{N-2}, \quad {\rm if} \ \alpha\ge N-2,
$$
and 
$$
\frac{N+\alpha}{N}< 1+\frac{\alpha}{N-2}<1+\frac{2+\alpha}{N}<\frac{N}{N-2}+\frac{\alpha}{N}<\frac{N+\alpha}{N-2}, \quad {\rm if} \ \alpha< N-2.
$$
Assume that $u_\nu$ is a mountain pass type solution  of \eqref{e410}, then
$$
-\Delta u_\nu +\lambda_\nu u_\nu =\nu(I_\alpha\ast |u_\nu|^p)|u_\nu|^{p-2}u_\nu+|u_\nu|^{2^*-2}u_\nu.
$$
Let
$$
\tilde v_\mu(x)=\lambda_\nu^{-\frac{N-2}{4}}u_\nu(\lambda_\nu^{-\frac{1}{2}}x).
$$
Then $v=\tilde v_\mu$ is a solution of $(Q_\mu)$ with $\mu=\nu \lambda_\nu^{-\frac{N+\alpha-p(N-2)}{2}}$. That is, 
$$
-\Delta \tilde v_\mu+\tilde v_\mu=\nu\lambda_\nu^{-\frac{N+\alpha-p(N-2)}{2}}(I_\alpha\ast|\tilde v_\mu|^p)|\tilde v_\mu|^{p-2}\tilde v_\mu+|\tilde v_\mu|^{2^*-2}\tilde v_\mu.
$$
It is easy to show that 
$$
\begin{array}{rcl}
\lambda_\nu\|u_\nu\|_2^2&=&\frac{N+\alpha-p(N-2)}{2p}\nu\int_{\mathbb R^N}(I_\alpha\ast|u_\nu|^p)|u_\nu|^p\\
&\lesssim &\frac{N+\alpha-p(N-2)}{2p}\nu\|u_\nu\|_2^{\frac{2(2^*-\tilde p)}{2^*-2}\frac{N+\alpha}{N}}\|\nabla u_\nu\|_2^{\frac{2^*(\tilde p-2)}{2^*-2}\frac{N+\alpha}{N}}.
\end{array}
$$
It is easy to show that $u_\nu$ is bounded in $D^{1,2}(\mathbb R^N)$. Therefore,  from the fact that $\|\nabla u_\nu\|_2=\|\nabla v_\nu\|_2$,  we get 
$$
\lambda_\nu\lesssim \nu a^{\alpha-(p-1)(N-2)}\|\nabla u_\nu\|_2^{\frac{2^*(\tilde p-2)}{2^*-2}\frac{N+\alpha}{N}}
\lesssim  \nu a^{\alpha-(p-1)(N-2)}.
$$
$$
\begin{array}{rcl}
\nu\lambda_\nu^{-\frac{N+\alpha-p(N-2)}{2}}=\nu \lambda_\nu^{-p(1-\eta_p)}
&\gtrsim &\nu\cdot(\nu^{-1}a^{(p-1)(N-2)-\alpha})^{\frac{N+\alpha-p(N-2)}{2}}\\
&=&[\nu^{-1}a^{-2p(1-\eta_p)}]^{\frac{[N+\alpha-p(N-2)][\alpha-(p-1)(N-2)]}{4p(1-\eta_p)}}\\
&=&[\nu^{-1}a^{-2p(1-\eta_p)}]^{\frac{\alpha-(p-1)(N-2)}{2}}\to +\infty, 
\end{array}
$$
 as  $\nu a^{2p(1-\eta_p)}\to 0$   in  the case that $p<1+\frac{\alpha}{N-2}$.
 
 Let $v_\mu$ is the ground state of $(Q_\mu)$ with $\mu=\nu \lambda_\nu^{-\frac{N+\alpha-p(N-2)}{2}}$. Then it follows from Theorem \ref{t210}  that 
\begin{equation}\label{e421}
\begin{array}{rcl}
E_\mu(v_\mu)&=&m_\mu-\frac{1}{2}\|v_\mu\|_2^2\\
&=&\mu^{-\frac{1}{p-1}}\left\{\frac{N(p-1)-2-\alpha}{4p}S_p^{\frac{p}{p-1}}+O(\mu^{-\frac{2}{(N-2)(p-1)}})\right\}\\
&\sim & \left\{\begin{array}{rcl}
\nu^{-\frac{1}{p-1}}\lambda_\nu^{\frac{N+\alpha-p(N-2)}{2(p-1)}}
\quad \qquad \quad &if & p>1+\frac{2+\alpha}{N},\\
O(\nu^{-\frac{N}{(N-2)(p-1)}}\lambda_\nu^{\frac{N[N+\alpha-p(N-2)]}{2(N-2)(p-1)}})\quad &if & p=1+\frac{2+\alpha}{N},\\
-\nu^{-\frac{1}{p-1}}\lambda_\nu^{\frac{N+\alpha-p(N-2)}{2(p-1)}}\quad \qquad \quad  &if & p<1+\frac{2+\alpha}{N}.
\end{array}\right.\\
&= & \left\{\begin{array}{rcl}
\nu^{-\frac{1}{p-1}}\lambda_\nu^{\frac{p(1-\eta_p)}{p-1}}\quad \qquad  \quad &if & p>1+\frac{2+\alpha}{N},\\
O(\nu^{-\frac{N}{(N-2)(p-1)}}\lambda_\nu^{\frac{Np(1-\eta_p)}{(N-2)(p-1)}})\quad &if & p=1+\frac{2+\alpha}{N},\\
-\nu^{-\frac{1}{p-1}}\lambda_\nu^{\frac{p(1-\eta_p)}{p-1}}\quad \qquad \quad  &if & p<1+\frac{2+\alpha}{N}.
\end{array}\right.
\end{array}
%\eqno(4.21)
\end{equation}
Therefore,  in  the case that $1+\frac{2+\alpha}{N}<p<1+\frac{\alpha}{N-2}$, as $\varpi=\nu a^{2p(1-\eta_p)}$ is sufficiently small (  this implies that $\mu=\nu \lambda_\nu^{-\frac{N+\alpha-p(N-2)}{2}}=\nu \lambda_\nu^{-p(1-\eta_p)}$ is sufficiently large), we have
$$
E_\mu(v_\mu)\lesssim  [\nu a^{2p(1-\eta_p)}]^{\frac{\alpha-(p-1)(N-2)}{2(p-1)}}< \frac{1}{N}S^{\frac{N}{2}}+O(\nu a^{2p(1-\eta_p)})\le   \inf_{u\in P_{a,-}}E_{a,\nu}(u)=E_\mu(\tilde v_\mu).
$$

In the case that  $p=1+\frac{2+\alpha}{N}<1+\frac{\alpha}{N-2}$, as $\varpi=\nu a^{2p(1-\eta_p)}$ is sufficiently small (  this implies that $\mu=\nu \lambda_\nu^{-\frac{N+\alpha-p(N-2)}{2}}=\nu \lambda_\nu^{-p(1-\eta_p)}$ is sufficiently large), we have
$$
E_\mu(v_\mu)\lesssim  [\nu a^{2p(1-\eta_p)}]^{\frac{N[\alpha-(p-1)(N-2)]}{2(N-2)(p-1)}}<
 \frac{1}{N}S^{\frac{N}{2}}+O(\nu a^{2p(1-\eta_p)})\le   \inf_{u\in P_{a,-}}E_{a,\nu}(u)=E_\mu(\tilde v_\mu).
$$ 
Therefore,   $v_\mu$ and $\tilde v_\mu$ are different positive solutions of $(Q_\mu)$ with $\mu=\nu \lambda_\nu^{-\frac{N+\alpha-p(N-2)}{2}}$. 
 \end{proof}

\vskip 3mm

Let 
$
V_\varepsilon(x)=W_\varepsilon(x)\varphi(R_\varepsilon^{-1}x)
$
be the function defined in the last section.
Then for small $\varepsilon>0$,  it is not hard to show that 
\begin{equation}\label{e422}
\int_{\mathbb R^N}(I_\alpha\ast |V_\varepsilon|^p)|V_\varepsilon|^p\sim  \varepsilon^{N+\alpha-p(N-2)},  \quad {\rm  if} \quad  \  \frac{N+\alpha}{2(N-2)}<p<\frac{N+\alpha}{N-2}.
%\eqno(4.22)
\end{equation}
Discussed as in Soave \cite[Lemma 4.2, 5.1 and 6.1]{Soave-1},  there exists $t_{\nu,\varepsilon}>0$ such that 
$$
\|\nabla V_\varepsilon\|_2^2=\nu\eta_pt_{\nu,\varepsilon}^{2(p\eta_p-1)}\int_{\mathbb R^N}(I_\alpha\ast |V_\varepsilon|^p)|V_\varepsilon|^p +t_{\nu,\varepsilon}^{\frac{4}{N-2}}\|V_\varepsilon||_{2^*}^{2^*}
$$
and 
$$
2\|\nabla V_\varepsilon\|_2^2<\nu p\eta_p^2t_{\nu,\varepsilon}^{2(p\eta_p-1)}\int_{\mathbb R^N}(I_\alpha\ast |V_\varepsilon|^p)|V_\varepsilon|^p +2^*t_{\nu,\varepsilon}^{\frac{4}{N-2}}\|V_\varepsilon||_{2^*}^{2^*}.
$$
It follows that $\{t_{\nu,\varepsilon}\}$ is uniformly bounded and bounded below from 0 for all $\nu, \varepsilon>0$ sufficiently small. 
Since
$$
\|\nabla V_\varepsilon\|_2^2=S^{\frac{N}{2}}+O((R_\varepsilon\varepsilon^{-1})^{2-N}), \quad 
\|V_\varepsilon\|_{2^*}^{2^*}=S^{\frac{N}{2}}+O((R_\varepsilon\varepsilon^{-1})^{-N})
$$
for small $\varepsilon>0$, which imply that
$$
S^{\frac{N}{2}}(1-t_{\nu,\varepsilon}^{\frac{4}{N-2}})=\nu\eta_pt_{\nu,\varepsilon}^{2(p\eta_p-1)}\int_{\mathbb R^N}(I_\alpha\ast |V_\varepsilon|^p)|V_\varepsilon|^p
+O((R_\varepsilon\varepsilon^{-1})^{2-N}),
$$
 we get
$$
t_{\nu,\varepsilon}=1-(1+o_\nu(1))\frac{\nu\eta_p\int_{\mathbb R^N}(I_\alpha\ast |V_\varepsilon|^p)|V_\varepsilon|^p
+O((R_\varepsilon\varepsilon^{-1})^{2-N})
}{(2^*-2)S^{\frac{N}{2}}}
$$
and 
\begin{equation}\label{e423}
\int_{\mathbb R^N}(I_\alpha\ast |u_\nu|^p)|u_\nu|^p\ge (1+o_\nu(1))\left(\int_{\mathbb R^N}(I_\alpha\ast |V_\varepsilon|^p)|V_\varepsilon|^p-C\nu^{-1}(R_\varepsilon\varepsilon^{-1})^{2-N}\right).
%\eqno(4.23)
\end{equation}
If $N=3$, then
\begin{equation}\label{e424}
\int_{\mathbb R^N}(I_\alpha\ast |u_\nu|^p)|u_\nu|^p\gtrsim
\varepsilon^{3+\alpha-p}-C\nu^{-1}\varepsilon^2, \qquad    \frac{3}{2}+\frac{\alpha}{2}<p<3+\alpha.
%\eqno(4.24)
\end{equation}
Since $3+\alpha-p<2$ for $1+\alpha<p<3+\alpha$, we can choose 
\begin{equation}\label{e425}
\varepsilon_\nu=\left(\frac{3+\alpha-p}{2C}\nu\right)^{\frac{1}{p-1-\alpha}}\sim \nu^{\frac{1}{p-1-\alpha}}
%\eqno(4.25)
\end{equation}
as $\nu\to 0$, such that the right hand side of the  above estimate take the maximum $\frac{p-1-\alpha}{2}\varepsilon_\nu^{3+\alpha-p}$, and hance
\begin{equation}\label{e426}
\int_{\mathbb R^N}(I_\alpha\ast |u_\nu|^p)|u_\nu|^p\gtrsim
\varepsilon_\nu^{3+\alpha-p}, \qquad 1+\alpha<p<3+\alpha.
%\eqno(4.26)
\end{equation}
In what follows, we fix $a>0$ and give a precise asymptotic description of $\lambda_\nu$ in the case that $N=3$. Since $\lambda_\nu\sim \lambda_\nu a^2\sim\nu\int_{\mathbb R^N}(I_\alpha\ast |u_\nu|^p)|u_\nu|^p$, it follows from \eqref{e425} and \eqref{e426}  that
\begin{equation}\label{e427}
\lambda_\nu\gtrsim \nu\varepsilon_\nu^{3+\alpha-p}\sim \nu^{\frac{2}{p-1-\alpha}}, \qquad 1+\alpha<p<3+\alpha.
%\eqno(4.27)
\end{equation}

Arguing as in \cite[Proposition 4.1]{Wei-1}, we  also show that  as $\nu\to 0$, 
$$
\|\nabla u_\nu\|_2^2, \   \   \ \|u_\nu\|_{2^*}^{2^*}\to S^{\frac{N}{2}}.
$$
Moreover, up to a subsequence, $\{u_\nu\}$ is a minimizing sequence of the following minimizing problem:
$$
S=\inf_{u\in D^{1,2}(\mathbb R^N)\setminus \{0\}}\frac{\|\nabla u\|_2^2}{\|u\|_{2^*}^{2}}.
$$
Since $u_\nu$ is radially symmetric,  by Lemma \ref{l26}, there exists $\sigma_\nu>0$ such that 
\begin{equation}\label{e428}
v_\nu(x)=\sigma_\nu^{\frac{N-2}{2}}u_\nu(\sigma_\nu x)\to W_1 \quad \ in \  \ D^{1,2}(\mathbb R^N), \qquad as \ \ \nu\to 0.
%\eqno(4.28)
\end{equation}
Clearly,  $v=v_\nu$ satisfies 
\begin{equation}\label{e429}
-\Delta v_\nu +\lambda_\nu\sigma_\nu^2 v_\nu=\nu\sigma_\nu^{N+\alpha-p(N-2)}(I_\alpha\ast |v_\nu|^{p})|v_\nu|^{p-2}v_\nu+|v_\nu|^{2^*-2}v_\nu.
%\eqno(4.29)
\end{equation}
We remark that since $W_{1}\not\in L^2(\mathbb R^N)$ for $N=3$ and $\sigma_\nu^2\|v_\nu\|_2^2=a^2$, by the Fatou's lemma, we have $\lim_{\nu\to 0}\sigma_\nu=0$. 

Since $a>0$ is fixed, we have
\begin{equation}\label{e430}
\begin{array}{rcl}
\lambda_\nu\sim \lambda_\nu a^2&=&\lambda_\nu\|u_\nu\|_2^2\\
&=&\nu \frac{N+\alpha-p(N-2)}{2p}\int_{\mathbb R^N}(I_\alpha\ast |u_\nu|^p)|u_\nu|^p\\
&\sim& \nu\sigma_\nu^{N+\alpha-p(N-2)}\int_{\mathbb R^N}(I_\alpha\ast |v_\nu|^p)|v_\nu|^p, 
\end{array}
%\eqno(4.30)
\end{equation}

Since $v_\nu\to W_1$ in $L^{2^*}(\mathbb R^N)$, as in \cite[Lemma 4.6]{Moroz-1},  by using the embeddings $L^{2^*}(B_1)\hookrightarrow L^q(B_1)$ we prove that
$\liminf_{\nu\to 0}\int_{\mathbb R^N}(I_\alpha\ast |v_\nu|^p)|v_\nu|^p>0$. Therefore, it follows from \eqref{e427} and \eqref{e430} that
\begin{equation}\label{e431}
\nu\lesssim \lambda_\nu^{\frac{p-1-\alpha}{2}}, \qquad \lambda_\nu\gtrsim \nu\sigma_\nu^{N+\alpha-p(N-2)}.
%\eqno(4.31)
\end{equation}
Since $\max\{2,1+\alpha\}<p<3+\alpha$, by virtue of Proposition \ref{p54} with $\beta_\nu=\gamma_\nu=1$,  we obtain 
$$
v_\nu(x)\lesssim (1+|x|)^{-(N-2)}, \quad x\in \mathbb R^N.
$$
Therefore, $\int_{\mathbb R^N}(I_\alpha\ast |v_\nu|^p)|v_\nu|^p\lesssim 1$, and it follows from \eqref{e430} and \eqref{e431} that  
$$
\lambda_\nu\sim \nu\sigma_\nu^{N+\alpha-p(N-2)}.
$$

We define $ w_\nu=\varepsilon_\nu^{\frac{N-2}{2}}u_\nu(\varepsilon_\nu x)$ with $\varepsilon_\nu$ being given in \eqref{e425},  then $\|\nabla  w_\nu\|_2^2\sim \| w_\nu\|_{2^*}^{2^*}\sim 1$ as $\nu\to 0$. 
It is easy to see that
\begin{equation}\label{e432}
\begin{array}{rcl}
\sigma_\nu^{N+\alpha-p(N-2)}\int_{\mathbb R^N}(I_\alpha\ast |v_\nu|^p)|v_\nu|^p&=&\int_{\mathbb R^N}(I_\alpha\ast |u_\nu|^p)|u_\nu|^p\\
&=&\varepsilon_\nu^{N+\alpha-p(N-2)}\int_{\mathbb R^N}(I_\alpha\ast |w_\nu|^p)|w_\nu|^p,
\end{array}
%\eqno(4.32)
\end{equation}
which together with \eqref{e426} and the fact that $\int_{\mathbb R^N}(I_\alpha\ast |v_\nu|^p)|v_\nu|^p\lesssim 1$ implies that  
\begin{equation}\label{e433}
\sigma_\nu\gtrsim \varepsilon_\nu, \qquad for \ \ 1+\alpha<p<3+\alpha.
%\eqno(4.33)
\end{equation}

Let 
$$
\tilde w_\nu(x)=v_\nu\left((\frac{\varepsilon_\nu}{\sigma_\nu})^{\frac{N-2}{2}}x\right),\quad x\in \mathbb R^N.
$$
Then $w=\tilde w_\nu$ satisfies 
$$
-\Delta w+\lambda_\nu \sigma_\nu^2 (\frac{\varepsilon_\nu}{\sigma_\nu})^{N-2}w=\nu\sigma_\nu^{N+\alpha-p(N-2)}(\frac{\varepsilon_\nu}{\sigma_\nu})^{\frac{(N-2)(2+\alpha)}{2}}(I_\alpha\ast |w|^p)|w|^{p-2}w+(\frac{\varepsilon_\nu}{\sigma_\nu})^{N-2}|w|^{2^*-2}w.
$$
Since $\lambda_\nu\lesssim \nu$ and $\varepsilon_\nu/\sigma_\nu\lesssim 1$, adopting the Moser iteration, by Proposition \ref{p54} with $\beta_\nu=(\frac{\varepsilon_\nu}{\sigma_\nu})^{\frac{N-2}{2}}$ and $\gamma_\nu=(\frac{\varepsilon_\nu}{\sigma_\nu})^{\frac{(N-2)(2+\alpha)}{2}}$,  we obtain 
\begin{equation}\label{e434}
\tilde w_\nu(x)\lesssim (1+|x|)^{-(N-2)}, \qquad x\in \mathbb R^N.
%\eqno(4.34)
\end{equation}
By Hardy-Littlewood-Sobolev inequality and \eqref{e432}, it  follows that
$$
\begin{array}{rcl}
\int_{\mathbb R^N}(I_\alpha\ast |w_\nu|^p)|w_\nu|^p&=&(\frac{\varepsilon_\nu}{\sigma_\nu})^{(N-2)p-\frac{N-4}{2}(N+\alpha)}\int_{\mathbb R^N}(I_\alpha\ast |\tilde w_\nu|^p)|\tilde w_\nu|^p\\
&\le &C_p(N)(\frac{\varepsilon_\nu}{\sigma_\nu})^{(N-2)p-\frac{N-4}{2}(N+\alpha)}\|\tilde w_\nu\|_{\frac{2Np}{N+\alpha}}^{2p}\\
&\lesssim& (\frac{\varepsilon_\nu}{\sigma_\nu})^{(N-2)p+\frac{4-N}{2}(N+\alpha)}.
\end{array}
$$
Hence, by \eqref{e432}, we get
$$
\begin{array}{rcl}
\sigma_\nu^{N+\alpha-p(N-2)}\int_{\mathbb R^N}(I_\alpha\ast |v_\nu|^p)|v_\nu|^p&=&\varepsilon_\nu^{N+\alpha-p(N-2)}\int_{\mathbb R^N}(|_\alpha\ast |w_\nu|^p)|w_\nu|^p\\
&\lesssim& \varepsilon_\nu^{N+\alpha-p(N-2)}(\frac{\varepsilon_\nu}{\sigma_\nu})^{(N-2)p+\frac{4-N}{2}(N+\alpha)},
\end{array}
$$ 
which together with the fact that $\int_{\mathbb R^N}(I_\alpha\ast |v_\nu|^p)|v_\nu|^p\gtrsim 1$  implies that $\sigma_\nu\lesssim \varepsilon_\nu$ and hence, by \eqref{e425}, we obtain that $\sigma_\nu\sim \varepsilon_\nu\sim \nu^{\frac{1}{p-1-\alpha}}$ as $\nu\to 0$ for $1+\alpha<p<3+\alpha$.

Therefore, we obtain the following 

\begin{lemma}\label{l49}  Assume $\max\{2,1+\alpha\}<p<2+\alpha$. Then 
$$
\lambda_\nu\sim \nu^{\frac{2}{p-1-\alpha}}, \qquad as \   \ \nu\to 0.
$$
\end{lemma}

Now, we are in the position to prove the last part of Theorem \ref{t12} in the case of $N=3, p\in (\max\{2,1+\alpha\}, 2+\alpha)$.

\begin{proof}[Proof of Theorem \ref{t12} (completed)]  Assume that $N=3, p\in (\max\{2,1+\alpha\}, 2+\alpha)$ and $u_\nu$ is a mountain pass type solution  of \eqref{e410}, then
$$
-\Delta u_\nu +\lambda_\nu u_\nu =\nu(I_\alpha\ast |u_\nu|^p)|u_\nu|^{p-2}u_\nu+|u_\nu|^{2^*-2}u_\nu.
$$
Let
$$
\tilde v_\mu(x)=\lambda_\nu^{-\frac{N-2}{4}}u_\nu(\lambda_\nu^{-\frac{1}{2}}x).
$$
Then  $v=\tilde v_\mu$ satisfies $(Q_\mu)$ with $\mu=\nu\lambda_\nu^{-\frac{N+\alpha-q(N-2)}{2}}$. That is, 
$$
-\Delta \tilde v_\mu +\tilde v_\mu=\nu\lambda_\nu^{-\frac{N+\alpha-p(N-2)}{2}}(I_\alpha\ast|\tilde v_\mu|^{p})|\tilde v_\mu|^{p-2}\tilde v_\mu+|\tilde v_\mu|^{2^*-2}\tilde v_\mu.
$$
Since $\max\{2,1+\alpha\}<p<2+\alpha$, by Lemma \ref{l49}, we obtain 
\begin{equation}\label{e435}
\mu=\nu\lambda_\nu^{-\frac{3+\alpha-p}{2}}\sim \nu^{-\frac{2(2+\alpha-p)}{p-1-\alpha}}\to +\infty, \qquad as \   \nu\to 0.
\end{equation}
 Moreover, by Theorem \ref{t46} and Theorem \ref{t47}, we have 
$$
E_\mu(\tilde v_\mu)=E_{a,\nu}(u_\nu)=\left\{\begin{array}{rcl}
\mathcal E_\nu>0,\qquad \qquad \qquad  &if& p\in (\frac{3+\alpha}{3},1+\frac{2+\alpha}{3}),\\
\frac{1}{N}S^{\frac{N}{2}}+O(\nu a^{2p(1-\eta_p)}),  &if& p\in [1+\frac{2+\alpha}{3}, 2_\alpha^*).
\end{array}\right.
$$
Let $v_\mu\in H^1(\mathbb R^N)$ be the ground state solution of $(Q_\mu)$,   then it follows from Theorem \ref{t210} and  \eqref{e435} that
$$
\begin{array}{rcl}
E_\mu(v_\mu)&=&m_\mu-\frac{1}{2}\|v_\mu\|_2^2\\
&=&\mu^{-\frac{1}{p-1}}\left\{\frac{N(p-1)-2-\alpha}{4p}S_p^{\frac{p}{p-1}}+O(\mu^{-\frac{2}{(p-1)(N-2)}})\right\}\\
&\sim &
 \nu^{\frac{2(2+\alpha-p)}{(p-1)(p-1-\alpha)}}, \quad   if \ \ p\in (\max\{2,1+\alpha\}, 2+\alpha).
\end{array}
$$
Therefore, we conclude that  $E_\mu(\tilde v_\mu)>E_\mu(v_\mu) $ for large $\mu$,  and hence,  $\tilde v_\mu$ and $v_\mu$  are different  positive solutions of $(Q_\mu)$ with $\mu=\nu\lambda_\nu^{-\frac{2N-q(N-2)}{4}}$. The proof of Theorem \ref{t12} is completed. 
\end{proof}

\vskip 5mm 

\section{Appendix}\label{s3}

In this section, we present some optimal decay estimates for positive solutions of some elliptic equations with parameters, which should be useful in other related context. 

\subsection{Hardy-Littlewood-Sobolev critical case.}
Consider the following equations 
 \begin{equation}\label{e51}
 -\Delta \tilde w+\lambda_\nu\sigma_\nu^2 \beta_\nu^2 \tilde w=\nu \sigma_\nu^{N-\frac{N-2}{2}q}\beta_\nu^2 |\tilde w|^{q-2}\tilde w+\gamma_\nu (I_\alpha\ast |\tilde w|^{2_\alpha^*})|\tilde w|^{2_\alpha^*-2}\tilde w,\quad x\in \mathbb R^N,
% \eqno(5.1)
 \end{equation}
where $\nu>0$ is a small parameter, $\lim_{\nu\to 0}\lambda_\nu=0$ and $\lambda_\nu, \sigma_\nu, \beta_\nu, \gamma_\nu>0$ are bounded.  
  
Let $W_1$ be  the Talenti function given by
$$
W_1(x):= [N(N-2)]^{\frac{N-2}{4}}\left(\frac{1}{1+|x|^2}\right)^{\frac{N-2}{2}}.
$$
In this section, we always assume that $N\ge 3, \alpha>N-4$ and $2<q<2^*$. Then adopting the Moser iteration, we establish the following uniform decay estimates of the weak solutions.
 
\begin{proposition}\label{p51}  Assume that   $N\ge 3$,  $\tilde w_\nu\in H^1(\mathbb R^N)$ is a weak positive  solution of  \eqref{e51}, which is radially symmetric and  non-increasing, and 
 \begin{equation}\label{e52}
 \nu\lesssim \min\{\lambda_\nu^{\frac{1}{4}(q-2)}, \lambda_\nu^{\frac{1}{4}[4-N(q-2)]}\},\quad 
\nu\sigma_\nu^{N-\frac{N-2}{2}q}\lesssim \lambda_\nu.
%\eqno(5.2)
\end{equation}
If  $\gamma_\nu\lesssim \beta_\nu^{2+\alpha}$, $\nu^{\tau_1}\sigma_\nu^{\tau_2}\lambda_\nu^{\tau_3}\lesssim \beta_\nu$ for some $(\tau_1,\tau_2,\tau_3)\in \mathbb R^3$, $v_\nu:=\tilde w(\beta_\nu^{-1}x)\to W_1$ in $D^{1,2}(\mathbb R^N)$ as $\nu\to 0$, and $u_\nu:=\sigma_\nu^{-(N-2)/2}v_\nu(\sigma_\nu^{-1}x)$ is bounded in  $H^1(\mathbb R^N)$,  then  there exists a constant $C>0$ such that 
for small $\nu>0$, there holds
$$
\tilde w_\nu(x)\le  C(1+|x|)^{-(N-2)}, \qquad x\in \mathbb R^N.
$$
\end{proposition}

The following result is a technique lemma.
\smallskip
 
\begin{lemma}\label{l52}    Under the assumptions of Proposition \ref{p51}, there exist constants $L_0>0$ and $C_0>0$ such that for any small $\nu>0$ and $|x|\ge L_0\lambda_\nu^{-1/2}\sigma_\nu^{-1}\beta_\nu^{-1}$, 
$$
\tilde w_\nu(x)\le C_0\lambda_\nu^{N/4}\sigma_\nu^{(N-2)/2}\exp(-\frac{1}{2}\lambda_\nu^{1/2}\sigma_\nu\beta_\nu |x|).
$$
\end{lemma}
\begin{proof} Since $u_\nu$ is bounded in $H^1(\mathbb R^N)$, radially symmetric and non-increasing with respect to $r=|x|>0$, we get
$$
|u_\nu(x)|^2\le \frac{1}{|B_{|x|}|}\int_{B_{|x|}}|u_\nu|^2\le C|x|^{-N}.
$$ 
Note that $\alpha>N-4$ implies that $2_\alpha^*>2$ and hence 
$$
| u_\nu(x)|^{2_\alpha^*}|x|^N\le C|x|^{-\frac{N}{2}(2_\alpha^*-2)}\to 0,  \  \ {\rm as} \   |x|\to \infty.
$$ 
Therefore, it follows from \cite[Lemma 3.10]{Ma-2} that
$$
(I_\alpha\ast |u_\nu|^{2_\alpha^*})(x)\le C|x|^{-(N-\alpha)},
$$
for all large $|x|$.  Thus,   if $|x|\ge L_0\lambda_\nu^{-1/2}\sigma_\nu^{-1}\beta_\nu^{-1}$ with  $L_0>0$ being large enough, we have
$$
\begin{array}{rl}
&\gamma_\nu(I_\alpha\ast |\tilde w_\nu|^{2_\alpha^*})(x)|\tilde w_\nu|^{2_\alpha^*-2}(x)\\
&=\gamma_\nu\beta_\nu^{-\alpha}\sigma_\nu^{(N-2)(2_\alpha^*-1)-\alpha}(I_\alpha\ast |u_\nu|^{2_\alpha^*})(\sigma_\nu\beta_\nu x)|u_\nu|^{2_\alpha^*-2}(\sigma_\nu\beta_\nu x)\\
&\le \beta_\nu^{2+\alpha}\beta_\nu^{-\alpha}\sigma_\nu^2L_0^{-\frac{N^2-N\alpha+4\alpha}{4(N-2)}}\lambda_\nu^{\frac{N^2-N\alpha+4\alpha}{4(N-2)}}\\
&\le \frac{1}{4}\lambda_\nu\sigma_\nu^2\beta_\nu^2, 
\end{array}
$$
here we have used the fact that $\frac{N^2-N\alpha+4\alpha}{4(N-2)}>1$. Since $\nu\lesssim \lambda_\nu^{\frac{1}{4}[4-N(q-2)]}$ for $2<q<2^*$,   and 
$$
\frac{1}{4}[4-N(q-2)]+\frac{N}{2}(q-2)=1+\frac{1}{4}N(q-2)>1,
$$
for $|x|\ge L_0\lambda_\nu^{-1/2}\sigma_\nu^{-1}\beta_\nu^{-1}$, by the boundedness of $u_\nu\in H^1(\mathbb R^N)$, we get
$$
\begin{array}{rcl}
\nu\sigma_\nu^{N-\frac{N-2}{2}q}\beta_\nu^2\tilde w_\nu^{q-2}(x)&=&\nu\sigma_\nu^{N-\frac{N-2}{2}q}\beta_\nu^2\sigma_\nu^{\frac{N-2}{2}(q-2)}u_\nu^{q-2}(\sigma_\nu\beta_\nu x)\\
&\le& \nu\sigma_\nu^{2}\beta_\nu^2\cdot C|\sigma_\nu \beta_\nu x|^{-\frac{N}{2}(q-2)}\\
&\le &CL_0^{-N(q-2)/2}\nu \sigma_\nu^2\beta_\nu^2 \lambda_\nu^{\frac{N}{4}(q-2)}\\
&\lesssim & \frac{1}{4}\sigma_\nu^2\beta_\nu^2\lambda_\nu^{\frac{1}{4}[4-N(q-2)]+\frac{N}{2}(q-2)}\\
&\le & \frac{1}{4}\lambda_\nu \sigma_\nu^2\beta_\nu^2.
\end{array}
$$
Therefore, we obtain 
$$
-\Delta \tilde w_\nu(x)+\frac{1}{2}\lambda_\nu\sigma_\nu^2 \beta_\nu^2\tilde w_\nu(x)\le 0, \quad 
{\rm for \ all} \  |x|\ge L_0\lambda_\nu^{-\frac{1}{2}} \sigma_\nu^{-1}\beta_\nu^{-1}.
$$ 
Since $\tilde w_\nu(x)=\sigma_\nu^{\frac{N-2}{2}}u_\nu(\sigma_\nu\beta_\nu x)$, it follows that
$$
\tilde w_\nu(L_0\lambda_\nu^{-\frac{1}{2}}\sigma_\nu^{-1}\beta_\nu^{-1})
\lesssim  \sigma_\nu^{-\frac{N-2}{2}}|L_0\lambda_\nu^{-\frac{1}{2}}|^{-\frac{N}{2}}=L_0^{-\frac{N}{2}}\lambda_\nu^{\frac{N}{4}}\sigma_\nu^{\frac{N-2}{2}}.
$$

We adopt an argument as used in  \cite[Lemma 3.2]{Akahori-2}.
 Let $R>L_0\lambda_\nu^{-1/2}\sigma_\nu^{-1}\beta_\nu^{-1}$, and introduce a positive function 
$$
\psi_R(r):=\exp(-\frac{1}{2}\lambda_\nu^{1/2}\sigma_\nu\beta_\nu(r-L_0\lambda_\nu^{-1/2}\sigma_\nu^{-1}\beta_\nu^{-1}))+\exp(\frac{1}{2}\lambda_\nu^{1/2}\sigma_\nu\beta_\nu(r-R)).
$$
It is easy to see that 
$$
|\psi'_R(r)|\le \frac{1}{2}\lambda_\nu^{1/2}\sigma_\nu\beta_\nu\psi_R(r),  \quad 
\psi''_R(r)=\frac{1}{4}\lambda_\nu\sigma_\nu^2\beta_\nu^2\psi_R(r).
$$
We use the same symbol $\psi_R$ to denote the radial function $\psi_R(|x|)$ on $\mathbb R^N$. Then  for $L_0\ge 2(N-1)$ and $L_0\lambda_\nu^{-1/2}\sigma_\nu^{-1}\beta_\nu^{-1}<r<R$, we have 
$$
\begin{array}{rcl}
-\Delta \psi_R+\frac{1}{2}\lambda_\nu\sigma_\nu^2\beta_\nu^2\psi_R&=&-\psi''_R-\frac{N-1}{r}\psi_R'+\frac{1}{2}\lambda_\nu\sigma_\nu^2\beta_\nu^2\psi_R\\
&\ge& -\frac{1}{4}\lambda_\nu\sigma_\nu^2\beta_\nu^2\psi_R-\frac{N-1}{L_0}\lambda_\nu^{1/2}\sigma_\nu\beta_\nu\cdot \frac{1}{2}\lambda_\nu^{1/2}\sigma_\nu\beta_\nu\psi_R+\frac{1}{2}\lambda_\nu\sigma_\nu^2\beta_\nu^2\psi_R\\
&\ge &0.
\end{array}
$$
Furthermore, $\psi_R(L_0\lambda_\nu^{-1/2}\sigma_\nu^{-1}\beta_\nu^{-1})\ge 1$ and $\psi_R(R)\ge 1$, thus we have 
$$
\tilde w_\lambda(R)\le \tilde w_\lambda(L_0\lambda_\nu^{-1/2}\sigma_\nu^{-1}\beta_\nu^{-1})\le CL_0^{-\frac{N}{2}}\lambda_\nu^{\frac{N}{4}}\sigma_\nu^{\frac{N-2}{2}}\min \{\psi_R(L_0\lambda_\nu^{-1/2}\sigma_\nu^{-1}\beta_\nu^{-1}), \psi_R(R)\}.
$$
Hence, the comparison principle implies that if $L_0\lambda_\nu^{-1/2}\sigma_\nu^{-1}\beta_\nu^{-1}\le |x|\le R$, then
$$
\tilde w_\lambda(x)\le CL_0^{-\frac{N}{2}}\lambda_\nu^{\frac{N}{4}}\sigma_\nu^{\frac{N-2}{2}}\psi_R(|x|).
$$
Since $R>L_0\lambda_\nu^{-1/2}\sigma_\nu^{-1}\beta_\nu^{-1}$ is arbitrary,  taking $R\to\infty$, we find that 
$$
\tilde w_\lambda(x)\le CL_0^{-\frac{N}{2}}\lambda_\nu^{\frac{N}{4}}\sigma_\nu^{\frac{N-2}{2}}e^{-\frac{1}{2}\lambda_\nu^{1/2}\sigma_\nu\beta_\nu |x|},
$$
for all $|x|\ge L_0\lambda_\nu^{-1/2}\sigma_\nu^{-1}\beta_\nu^{-1}$. The proof is complete.
\end{proof}

We consider the Kelvin transform of $\tilde w_\nu$:
$$
K[\tilde w_\nu](x):=|x|^{-(N-2)}\tilde w_\nu\left(\frac{x}{|x|^2}\right). 
$$
It is easy to see that $\|K[\tilde w_\nu]\|_{L^\infty(B_1)}\lesssim 1$ implies that 
$$
\tilde w_\nu(x)\lesssim |x|^{-(N-2)},  \qquad |x|\ge 1,
$$
uniformly for small $\nu>0$.

Thus, to prove Proposition \ref{p51}, it suffices to show that there exists $\nu_0>$ such that 
\begin{equation}\label{e53}
\sup_{\nu\in (0,\nu_0)}\|K[\tilde w_\nu]\|_{L^{\infty}(B_1)}<\infty.
%\eqno(5.3)
\end{equation}

It is easy to verify that $K[\tilde w_\nu]$ satisfies 
\begin{equation}\label{e54}
-\Delta K[\tilde w_\nu]+\frac{\lambda_\nu\sigma_\nu^{2}\beta_\nu^2}{|x|^4}K[\tilde w_\nu]=\frac{\gamma_\nu}{|x|^{4}}(I_\alpha\ast |\tilde w_\nu|^{2_\alpha^*})(\frac{x}{|x|^2})\tilde w_\nu^{2_\alpha^*-2}(\frac{x}{|x|^2})K[\tilde w_\nu]+\frac{\nu\sigma_\nu^{\gamma/2}\beta_\nu^2}{|x|^{\gamma}}K[\tilde w_\nu]^{q-1},
%\eqno(5.4)
\end{equation}
here and in what follows, we set
\begin{equation}\label{e55}
\gamma:=2N-(N-2)q>0.
%\eqno(5.5)
\end{equation}
 
We also see from Lemma \ref{l52} that if $|x|\le \lambda_\nu^{1/2}\sigma_\nu\beta_\nu/L_0$, then 
\begin{equation}\label{e56}
K[\tilde w_\nu](x)\lesssim \frac{1}{|x|^{N-2}}\lambda_\nu^{\frac{N}{4}}\sigma_\nu^{\frac{N-2}{2}}e^{-\frac{1}{2}\lambda_\nu^{1/2}\sigma_\nu\beta_\nu |x|^{-1}}.
%\eqno(5.6)
\end{equation}

   Let 
$$
a(x)=\frac{\lambda_\nu\sigma_\nu^2\beta_\nu^2}{|x|^4},
$$
and
$$  
 b(x)=\frac{\gamma_\nu}{|x|^{4}}(I_\alpha\ast |\tilde w_\nu|^{2_\alpha^*})(\frac{x}{|x|^2})\tilde w_\nu^{2_\alpha^*-2}(\frac{x}{|x|^2})
+\frac{\nu\sigma_\nu^{\gamma/2}\beta_\nu^2}{|x|^\gamma}K[\tilde w_\nu]^{q-2}(x).
$$
Then \eqref{e54} reads as 
$$
-\Delta K[\tilde w_\nu]+a(x)K[\tilde w_\nu]=b(x)K[\tilde w_\nu].
$$
We shall apply the Moser iteration to prove \eqref{e53}.

We note that it follows from \eqref{e56} that for any $v\in H^1_0(B_4)$, 
\begin{equation}\label{e57}
\int_{B_4}\frac{\lambda_\nu\sigma_\nu^2\beta_\nu^2}{|x|^4}K[\tilde w_\nu](x)|v(x)|dx<\infty.
%\eqno(5.7)
\end{equation}
Since $v_\nu\to W_1$  in $L^{2^*}(\mathbb R^N)$   as $\nu\to 0$,  and  for any $s>1$, the Lebesgue space $L^{s}(\mathbb R^N)$ has the Kadets-Klee property,   it is easy to see that
$$
\lim_{\nu\to 0}\int_{\mathbb R^N}|v_\nu^{2_\alpha^*}-W_1^{2_\alpha^*}|^{\frac{2N}{N+\alpha}}dx=0
$$
and 
$$
\lim_{\nu\to 0}\int_{\mathbb R^N}|v_\nu^{2_\alpha^*-2}-W_1^{2_\alpha^*-2}|^{\frac{2N}{\alpha-N+4}}dx=0.
$$
Therefore, by the H\"older inequality and the Hardy-Littlewood-Sobolev inequality, we find 
\begin{equation}\label{e58}
\begin{array}{rl}
&\left\|\frac{1}{|x|^4}(I_\alpha\ast (|v_\nu|^{2_\alpha^*}-|W_1|^{2_\alpha^*}))(\frac{x}{|x|^2})v_\nu^{2_\alpha^*-2}(\frac{x}{|x|^2})\right\|_{L^{\frac{N}{2}}(\mathbb R^N)}^{N/2}\\
&\quad =\int_{\mathbb R^N}\frac{1}{|x|^{2N}}|(I_\alpha\ast (|v_\nu|^{2_\alpha^*}-|W_1|^{2_\alpha^*}))(\frac{x}{|x|^2})|^{N/2}|v_\nu^{2_\alpha^*-2}(\frac{x}{|x|^2})|^{N/2}dx\\
&\quad= \int_{\mathbb R^N}|(I_\alpha\ast (|v_\nu|^{2_\alpha^*}-|W_1|^{2_\alpha^*}))(z)|^{N/2}|v_\nu^{2_\alpha^*-2}(z)|^{N/2}dz\\
&\quad \lesssim \left(\int_{\mathbb R^N}|v_\nu^{2_\alpha^*}-W_1^{2_\alpha^*}|^{\frac{2N}{N+\alpha}}\right)^{\frac{N^2-\alpha^2}{4(N-2)}}\left(\int_{\mathbb R^N}|v_\nu|^{2^*}\right)^{\frac{\alpha-N+4}{4}}\\
&\quad \to 0, \quad {\rm as} \ \ \nu\to 0
\end{array}
%\eqno(5.8)
\end{equation}
and 
\begin{equation}\label{e59}
\begin{array}{rl}
&\left\|\frac{1}{|x|^4}(I_\alpha\ast |W_1|^p)(\frac{x}{|x|^2})[v_\nu^{2_\alpha^*-2}(\frac{x}{|x|^2})-W_1^{2_\alpha^*-2}(\frac{x}{|x|^2})]\right\|_{L^{\frac{N}{2}}(\mathbb R^N)}^{N/2}\\
&\quad =\int_{\mathbb R^N}\frac{1}{|x|^{2N}}|(I_\alpha\ast |W_1|^{2_\alpha^*}))(\frac{x}{|x|^2})|^{N/2}|v_\nu^{2_\alpha^*-2}(\frac{x}{|x|^2})-W_1^{2_\alpha^*-2}(\frac{x}{|x|^2})|^{N/2}dx\\
&\quad= \int_{\mathbb R^N}|(I_\alpha\ast |W_1|^{2_\alpha^*})(z)|^{N/2}|v_\nu^{2_\alpha^*-2}(z)-W_1^{2_\alpha^*-2}(z)|^{N/2}dz\\
&\quad \lesssim \left(\int_{\mathbb R^N}|W_1|^{2^*}\right)^{\frac{N^2-\alpha^2}{4(N-2)}}\left(\int_{\mathbb R^N}|v_\nu^{2_\alpha^*-2}(z)-W_1^{2_\alpha^*-2}(z)|^{\frac{2N}{\alpha-N+4}}\right)^{\frac{\alpha-N+4}{4}}\\
&\quad \to 0, \quad {\rm as} \ \ \nu\to 0.
\end{array}
%\eqno(5.9)
\end{equation}
It  follows from \eqref{e58} and \eqref{e59} that
\begin{equation}\label{e510}
\begin{array}{rl}
&\left\|\frac{1}{|x|^{4}}(I_\alpha\ast |v_\nu|^{2_\alpha^*})(\frac{x}{|x|^2})v_\nu^{2_\alpha^*-2}(\frac{x}{|x|^2})-\frac{1}{|x|^{4}}(I_\alpha\ast |W_1|^{2_\alpha^*})(\frac{x}{|x|^2})W_1^{2_\alpha^*-2}(\frac{x}{|x|^2})\right\|_{L^{\frac{N}{2}}(\mathbb R^N)}\\
&\le \left\| \frac{1}{|x|^{4}}(I_\alpha\ast |W_1|^{2_\alpha^*})(\frac{x}{|x|^2})\left[v_\nu^{2_\alpha^*-2}(\frac{x}{|x|^2})-W_1^{2_\alpha^*-2}(\frac{x}{|x|^2})\right]\right\|_{L^{\frac{N}{2}}(\mathbb R^N)}\\
&\quad +\left\|\frac{1}{|x|^{4}}(I_\alpha\ast ( |v_\nu|^{2_\alpha^*}-|W_1|^{2_\alpha^*}))(\frac{x}{|x|^2})v_\nu^{2_\alpha^*-2}(\frac{x}{|x|^2})\right\|_{L^{\frac{N}{2}}(\mathbb R^N)}\\
&\quad \to 0, \quad {\rm as} \ \ \nu\to 0.
\end{array}
%\eqno(5.10)
\end{equation}
Therefore, we obtian
\begin{equation}\label{e511}
\begin{array}{rcl}
J_\nu(\frac{N}{2})&:=&\int_{|x|\le 4}\left|\frac{\gamma_\nu}{|x|^{4}}(I_\alpha\ast |v_\nu|^{2_\alpha^*})(\frac{x}{|x|^2})v_\nu^{2_\alpha^*-2}(\frac{x}{|x|^2})\right|^{N/2}dx\\
&=&\gamma_\nu^{N/2}\beta_\nu^{-N}\int_{|x|\le 4\beta_\nu^{-1}}\left|\frac{1}{|x|^{4}}(I_\alpha\ast |v_\nu|^{2_\alpha^*})(\frac{x}{|x|^2})v_\nu^{2_\alpha^*-2}(\frac{x}{|x|^2})\right|^{N/2}dx\\
&\lesssim &\beta_\nu^{(2+\alpha)N/2-N}\int_{\mathbb R^N}\left|\frac{1}{|x|^{4}}(I_\alpha\ast |v_\nu|^{2_\alpha^*})(\frac{x}{|x|^2})v_\nu^{2_\alpha^*-2}(\frac{x}{|x|^2})\right|^{N/2}dx\\
&\lesssim &\beta_\nu^{\frac{N}{2}\alpha}.
\end{array}
%\eqno(5.11)
\end{equation}

\begin{lemma}\label{l53}  Assume $N\ge 3$ and $2<q<2^*$. Then it holds that 
$$
\limsup_{\nu\to 0}\int_{|x|\le 4}\left|\frac{\nu\sigma_\nu^{\gamma/2}\beta_\nu^2}{|x|^\gamma}K[\tilde w_\nu]^{q-2}(x)\right|^{N/2}dx<\infty.
$$
\end{lemma}

\begin{proof}  We divide the integral into three parts:
$$
I_\nu^{(1)}(\frac{N}{2}):=\int_{|x|\le \lambda_\nu^{1/2}\sigma_\nu\beta_\nu/L_0}\left |\frac{\nu\sigma_\nu^{\gamma/2}\beta_\nu^2}{|x|^\gamma}K[\tilde w_\nu]^{q-2}(x)\right |^{N/2}dx,
$$
$$
I_\nu^{(2)}(\frac{N}{2}):=\int_{\lambda_\nu^{1/2}\sigma_\nu\beta_\nu/L_0\le |x|\le \sigma_\nu/L_0}\left |\frac{\nu\sigma_\nu^{\gamma/2}\beta_\nu^2}{|x|^\gamma}K[\tilde w_\nu]^{q-2}(x)\right |^{N/2}dx,
$$
$$
I_\nu^{(3)}(\frac{N}{2}):=\int_{\sigma_\nu/L_0\le |x|\le 4}\left |\frac{\nu\sigma_\nu^{\gamma/2}\beta_\nu^2}{|x|^\gamma}K[\tilde w_\nu]^{q-2}(x)\right |^{N/2}dx.
$$
We denote $\eta_\nu:=\lambda_\nu^{1/2}\sigma_\nu\beta_\nu/L_0$, then it follows from  \eqref{e56} that 
$$
\begin{array}{rcl}
I_\nu^{(1)}(\frac{N}{2})&\le&\nu^{\frac{N}{2}}\lambda_\nu^{\frac{N^2(q-2)}{8}}\sigma_\nu^{\frac{N}{4}(\gamma+(N-2)(q-2))}\beta_\nu^N\\
&\mbox{}&\qquad \qquad \qquad \cdot \int_{|x|\le \eta_\nu}|x|^{-\frac{\gamma N}{2}-\frac{N(N-2)(q-2)}{2}}e^{-\frac{N(q-2)}{4}\lambda_\nu^{1/2}\sigma_\nu\beta_\nu |x|^{-1}}dx\\
&=&\nu^{\frac{N}{2}}\lambda_\nu^{\frac{N^2(q-2)}{8}}\sigma_\nu^{\frac{N}{4}(\gamma+(N-2)(q-2))}\beta_\nu^N\\
&\mbox{}&\qquad \qquad \qquad \cdot \int_0^{\eta_\nu}r^{-\frac{\gamma N}{2}-\frac{N(N-2)(q-2)}{2}+N-1}e^{-\frac{N(q-2)}{4}\lambda_\nu^{1/2}\sigma_\nu\beta_\nu r^{-1}}dr\\
&=&\nu^{\frac{N}{2}}\lambda_\nu^{\frac{N}{8}[N(q-2)-2\gamma-2(N-2)(q-2)+4]}\\
&\mbox{}&\qquad \qquad \qquad \cdot \int_{L_0}^{+\infty}s^{\frac{\gamma N}{2}+\frac{N(N-2)(q-2)}{2}-N-1}e^{-\frac{N(q-2)}{4}s}ds\\
&\lesssim &\nu^{\frac{N}{2}}\lambda_\nu^{\frac{N}{8}[N(q-2)-4]}\lesssim 1,
\end{array}
$$
here we have used the fact that $\nu\lesssim \lambda_\nu^{\frac{1}{4}[4-N(q-2)]}$ in the last inequality.

If $2^*>q>2+\frac{4}{N}
%=\frac{10}{3}
$, then $-\frac{N^2}{4}(q-2)+N=\frac{N}{4}[4-N(q-2)]<0$, and hence
$$
\begin{array}{rcl}
I_\nu^{(2)}(\frac{N}{2})&=&\nu^{\frac{N}{2}}\sigma_\nu^{\frac{\gamma N}{4}}\beta_\nu^N\int_{\eta_\nu\le |x|\le \sigma_\nu/L_0}|x|^{-\frac{\gamma N}{2}}K[\tilde w_\nu]^{\frac{N}{2}(q-2)}(x)dx\\
&=&\nu^{\frac{N}{2}}\sigma_\nu^{\frac{\gamma N}{4}}\beta_\nu^N\cdot \int_{L_0\sigma_\nu^{-1}\le |z|\le \eta_\nu^{-1}}\tilde w_\nu^{\frac{N}{2}(q-2)}(z)dz\\
&=&\nu^{\frac{N}{2}}\sigma_\nu^{\frac{\gamma N}{4}}\beta_\nu^N\cdot \sigma_\nu^{\frac{1}{4}N(N-2)(q-2)}\int_{L_0\sigma_\nu^{-1}\le |z|\le \eta_\nu^{-1}}u_\nu^{\frac{N}{2}(q-2)}(\sigma_\nu\beta_\nu z)dz\\
&\lesssim &\nu^{\frac{N}{2}}\sigma_\nu^{\frac{\gamma N}{4}}\beta_\nu^N\cdot \sigma_\nu^{\frac{1}{4}N(N-2)(q-2)}\int_{L_0\sigma_\nu^{-1} \le |z|\le
\eta_\nu^{-1}}|\sigma_\nu\beta_\nu z|^{-\frac{N^2}{4}(q-2)}dz\\
&=&\nu^{\frac{N}{2}}\sigma_\nu^{\frac{\gamma N}{4}}\beta_\nu^N\cdot \sigma_\nu^{-\frac{N(q-2)}{2}}\int_{L_0\sigma_\nu^{-1}}^{\eta_\nu^{-1}}r^{-\frac{N^2}{4}(q-2)+N-1}dr\\
&\lesssim&\nu^{\frac{N}{2}}\sigma_\nu^{\frac{\gamma N}{4}}\beta_\nu^N\cdot \sigma_\nu^{-\frac{N(q-2)}{2}}\cdot \sigma_\nu^{-\frac{N}{4}[4-N(q-2)]}
=\nu^{\frac{N}{2}}\beta_\nu^N\to 0, \qquad as \ \nu\to 0,
\end{array}
$$
and 
$$
\begin{array}{rcl}
I_\nu^{(3)}(\frac{N}{2})
&\le &\nu^{\frac{N}{2}}\sigma_\nu^{\frac{\gamma N}{4}}\beta_\nu^N\left(\int_{\sigma_\nu/L_0\le |x|\le 4}K[\tilde w_\nu]^{2^*}dx\right)^{1-\frac{\gamma}{4}}
\left(\int_{\sigma_\nu/L_0\le |x|\le 4} |x|^{-2N}dx\right)^{\frac{\gamma}{4}}\\
&=&\nu^{\frac{N}{2}}\sigma_\nu^{\frac{\gamma N}{4}}\beta_\nu^{N+\frac{N}{4}(\gamma-4)}\left(\int_{\sigma_\nu\beta_\nu^{-1}/L_0\le |x|\le 4\beta_\nu^{-1}}K[v_\nu]^{2^*}dx\right)^{1-\frac{\gamma}{4}}\\
&\mbox{}&\qquad \qquad \qquad\qquad  \cdot 
\left(\int_{\sigma_\nu/L_0\le |x|\le 4} |x|^{-2N}dx\right)^{\frac{\gamma}{4}}\\
&\lesssim &\nu^{\frac{N}{2}}\sigma_\nu^{\frac{\gamma N}{4}}\beta_\nu^{N+\frac{N}{4}(\gamma-4)}\left(\int_{\sigma_\nu/L_0}^4r^{-N-1}dr\right)^{\frac{\gamma}{4}}\\
&\lesssim &\nu^{\frac{N}{2}}\sigma_\nu^{\frac{\gamma N}{4}}\beta_\nu^{\frac{N}{4}\gamma}\cdot\sigma_\nu^{-\frac{\gamma N}{4}} =\nu^{\frac{N}{2}}\beta_\nu^{\frac{N}{4}\gamma}\to 0, \qquad as \ \nu\to 0.
\end{array}
$$
If $q=2+\frac{4}{N}
%=\frac{10}{3}
$, then $-\frac{N^2}{4}(q-2)+N=\frac{N}{4}[4-N(q-2)]=0$, and hence
$$
\begin{array}{rcl}
I_\nu^{(2)}(\frac{N}{2})&\lesssim &\nu^{\frac{N}{2}}\sigma_\nu^{\frac{\gamma N}{4}}\beta_\nu^N\cdot \sigma_\nu^{-\frac{N(q-2)}{2}}\int_{L_0\sigma_\nu^{-1}}^{\eta_\nu^{-1}}r^{-1}dr\\
&\lesssim&\nu^{\frac{N}{2}}\beta_\nu^N\cdot (\ln(\eta_\nu^{-1})-\ln(L_0\sigma_\nu^{-1}))\\
&=&\nu^{\frac{N}{2}}\beta_\nu^N\cdot \ln(\lambda_\nu^{-1/2}\beta_\nu^{-1})\\
&\lesssim&\lambda_\nu^{\frac{N}{2}\cdot \frac{q-2}{4}}\beta_\nu^N\cdot \ln(\lambda_\nu^{-1/2}\beta_\nu^{-1})\to 0, \qquad as \ \nu\to 0,
\end{array}
$$
and 
$$
\begin{array}{rcl}
I_\nu^{(3)}(\frac{N}{2})
&\le &\nu^{\frac{N}{2}}\sigma_\nu^{\frac{\gamma N}{4}}\beta_\nu^N\left(\int_{\sigma_\nu/L_0\le |x|\le 4}K[\tilde w_\nu]^{2^*}dx\right)^{1-\frac{\gamma}{4}}
\left(\int_{\sigma_\nu/L_0\le |x|\le 4} |x|^{-2N}dx\right)^{\frac{\gamma}{4}}\\
&=&\nu^{\frac{N}{2}}\sigma_\nu^{\frac{\gamma N}{4}}\beta_\nu^{N+\frac{N}{4}(\gamma-4)}\left(\int_{\sigma_\nu\beta_\nu^{-1}/L_0\le |x|\le 4\beta_\nu^{-1}}K[v_\nu]^{2^*}dx\right)^{1-\frac{\gamma}{4}}\\
&\mbox{}&\qquad\qquad\qquad \qquad \cdot
\left(\int_{\sigma_\nu/L_0\le |x|\le 4} |x|^{-2N}dx\right)^{\frac{\gamma}{4}}\\
&\lesssim &\nu^{\frac{N}{2}}\sigma_\nu^{\frac{\gamma N}{4}}\beta_\nu^{N+\frac{N}{4}(\gamma-4)}\left(\int_{\sigma_\nu/L_0}^4r^{-N-1}dr\right)^{\frac{\gamma}{4}}\\
&\lesssim &\nu^{\frac{N}{2}}\sigma_\nu^{\frac{\gamma N}{4}}\beta_\nu^{\frac{N}{4}\gamma}\cdot\sigma_\nu^{-\frac{\gamma N}{4}} =\nu^{\frac{N}{2}}\beta_\nu^{\frac{N}{4}\gamma}\to 0, \qquad as \ \nu\to 0.
\end{array}
$$
here we have used the fact that
$$
\nu\lesssim \lambda_\nu^{\frac{q-2}{4}}\lesssim \lambda_\nu^{\frac{1}{4}[4-N(q-2)]},  \ q\in (2+\frac{4}{N+1}, 2^*).
$$
If $2<q<2+\frac{4}{N}
%=\frac{10}{3}
$, then $-\frac{N^2}{4}(q-2)+N=\frac{N}{4}[4-N(q-2)]>0$, and hence
$$
\begin{array}{rcl}
I_\nu^{(2)}(\frac{N}{2})+I_\nu^{(3)}(\frac{N}{2})&=&\nu^{\frac{N}{2}}\sigma_\nu^{\frac{\gamma N}{4}}\beta_\nu^N\cdot \int_{\frac{1}{4}\le |z|\le \eta_\nu^{-1}}\tilde w_\nu^{\frac{N}{2}(q-2)}(z)dz\\
&\lesssim &\nu^{\frac{N}{2}}\sigma_\nu^{\frac{\gamma N}{4}}\beta_\nu^N\cdot \sigma_\nu^{-\frac{N(q-2)}{2}}\beta_\nu^{-\frac{N^2}{4}(q-2)}\int_{\frac{1}{4}}^{\eta_\nu^{-1}}r^{-\frac{N^2}{4}(q-2)+N-1}dr\\
&\lesssim&\nu^{\frac{N}{2}}\sigma_\nu^{\frac{\gamma N}{4}}\beta_\nu^N\cdot \sigma_\nu^{-\frac{N(q-2)}{2}}\beta_\nu^{-\frac{N^2}{4}(q-2)}\\
&\mbox{}&\qquad\qquad\qquad\qquad \cdot \lambda_\nu^{-\frac{N}{8}[4-N(q-2)]}(\sigma_\nu\beta_\nu)^{-\frac{N}{4}[4-N(q-2)]}\\
%&=&\nu^{\frac{N}{2}}\lambda_\nu^{-\frac{N}{8}[4-N(q-2)]}\\
&\lesssim &\lambda_\nu^{\frac{N}{2}\cdot \frac{4-N(q-2)}{4}-\frac{N}{8}[4-N(q-2)]}\lesssim 1, \qquad as \ \nu\to 0.
\end{array}
$$
From which the conclusion follows.
\end{proof}

\begin{proof}[Proof of Proposition \ref{p51}] 

Since the Kelvin transform is linear and preserves the $D^{1,2}(\mathbb R^N )$ norm, 
it follows from \eqref{e57}, \eqref{e511}, Lemma \ref{l53} and Lemma \ref{l28} (i) that for any $r>1$, there exists $\nu_r>0$ such that 
\begin{equation}\label{e512}
\sup_{\nu\in (0,\nu_r)}\|K[\tilde w_\nu]^r\|_{H^1(B_1)}
\le C_r.
%\eqno(5.12)
\end{equation}

Since $\alpha>N-4$, for any  $r_0\in (\frac{N}{2}, \frac{2N}{N-\alpha})$, we have 
$$
 s_1=\frac{2N}{2N-(N-\alpha)r_0}>1, \quad  s_2=\frac{2N}{(N-\alpha)r_0}>1,
 $$
and 
 $$
 \frac{1}{s_1}+\frac{1}{s_2}=1.
 $$
Note that $(N-\alpha)r_0s_2=2N$, by the Hardy-Littlewood-Sobolev inequality, we get 
\begin{equation}\label{e513}
\begin{array}{rl}
&\int_{\mathbb R^N}\frac{1}{|x|^{(N-\alpha)r_0s_2}}|(I_\alpha\ast |\tilde w_\nu|^{2_\alpha^*})(\frac{x}{|x|^2})|^{r_0s_2}dx\\
&=\int_{\mathbb R^N}\frac{\beta_\nu^{-\alpha}}{|x|^{(N-\alpha)r_0s_2}}|(I_\alpha\ast (|v_\nu|^{2_\alpha^*})(\beta_\nu\frac{x}{|x|^2})|^{r_0s_2}dx\\
&=\beta_\nu^{-\alpha} \int_{\mathbb R^N}|(I_\alpha\ast |v_\nu|^{2_\alpha^*})(\beta_\nu z)|^{r_0s_2}dz\\
&=\beta_\nu^{-N-\alpha} \int_{\mathbb R^N}|(I_\alpha\ast |v_\nu|^{2_\alpha^*})( z)|^{r_0s_2}dz\\
&\lesssim \beta_\nu^{-N-\alpha} \left(\int_{\mathbb R^N}|v_\nu|^{\frac{2N}{N-2}}\right)^{\frac{(N-2)2_\alpha^*}{(N-2)2_\alpha^*-2\alpha}}.
 \end{array}
%\eqno(5.13)
\end{equation}
 By \eqref{e512}, \eqref{e513} and the H\"older inequality,  we have 
$$
\begin{array}{rcl}
J_\nu(r_0)&=&\int_{|x|\le 4}\left|\frac{\gamma_\nu}{|x|^{N-\alpha}}(I_\alpha\ast |\tilde w_\nu|^{2_\alpha^*})(\frac{x}{|x|^2})K[\tilde w_\nu]^{2_\alpha^*-2}(x)\right|^{r_0}dx\\
&\lesssim &\beta_\nu^{(2+\alpha)r_0}\left(\int_{|x|\le 4} K[\tilde w_\nu]^{(2_\alpha^*-2)r_0s_1}\right)^{\frac{1}{s_1}}\\
&\mbox{}& \cdot\left(\int_{\R^N}|x|^{-(N-\alpha)r_0s_2}\left|(I_\alpha\ast |\tilde w_\nu|^{2_\alpha^*})(\frac{x}{|x|^2})\right|^{r_0s_2}dx\right)^{\frac{1}{s_2}}\\
&\lesssim &\beta_\nu^{(2+\alpha)r_0-(N+\alpha)/s_2}=\beta_\nu^{\frac{N(4-N)+2N\alpha+\alpha^2}{2N}r_0}.
\end{array}
$$
Since $\alpha>N-4$, it is easy to see that  $N(4-N)+2N\alpha+\alpha^2>0$, and hence $J_\nu(r_0)\lesssim 1$ as $\nu\to 0$.

Since $\gamma=2N-q(N-2)$, it is easy to see that $0<\gamma<4$.  It remains to prove that there exists $r_0>\frac{N}{2}$ and $\nu_0>0$ such that 
\begin{equation}\label{e514}
\sup_{\nu\in (0,\nu_0)}\int_{|x|\le 4}\left|\frac{\nu\sigma_\nu^{\gamma/2}\beta_\nu^2}{|x|^\gamma}K[v_\nu]^{q-2}(x)\right|^{r_0}dx\le C_{r_0}.
%\eqno(5.14)
\end{equation}
Put $\eta_\nu=\lambda_\nu^{1/2}\sigma_\nu\beta_\nu/L_0$, we divide the integral in \eqref{e514} into two parts:
$$
I_\nu^{(1)}(r_0):=\int_{|x|\le \eta_\nu}\left |\frac{\nu\sigma_\nu^{\gamma/2}\beta_\nu^2}{|x|^\gamma}K[\tilde w_\nu]^{q-2}(x)\right |^{r_0}dx,
$$
$$
I_\nu^{(2)}(r_0):=\int_{\eta_\nu\le |x|\le 4}\left |\frac{\nu\sigma_\nu^{\gamma/2}\beta_\nu^2}{|x|^\gamma}K[\tilde w_\nu]^{q-2}(x)\right |^{r_0}dx.
$$ 
Let $\theta\in (0,1)$ and $s_1,s_2>1$ with $\frac{1}{s_1}+\frac{1}{s_2}=1$. Then by \eqref{e56} and the H\"older inequality, we have 
$$
\begin{array}{rcl}
I_\nu^{(1)}(r_0):&\lesssim &\nu^{r_0}\sigma_\nu^{\gamma r_0/2}\beta_\nu^{2r_0}\left(\int_{|x|\le \eta_\nu}\left|\frac{1}{|x|^\gamma}K[\tilde w_\nu]^{(q-2)\theta}(x)\right|^{r_0s_1}dx\right)^{\frac{1}{s_1}}\\
&\mbox{}&\qquad\qquad  \cdot\left(\int_{|x|\le \eta_\nu}\left|K[\tilde w_\nu]^{(q-2)(1-\theta)}(x)\right|^{r_0s_2}dx\right)^{\frac{1}{s_2}}\\
&\lesssim & \nu^{r_0}\lambda_\nu^{\frac{N(q-2)}{4}r_0}\sigma_\nu^{\frac{\gamma}{2}r_0+\frac{(N-2)(q-2)}{2}r_0}
\beta_\nu^{2r_0}\\
&\mbox{}&\qquad \cdot \left(\int_{|x|\le \eta_\nu}|x|^{-\gamma r_0s_1-(N-2)(q-2)\theta r_0s_1}e^{-\frac{1}{2}(q-2)\theta r_0s_1\lambda_\nu^{1/2}\sigma_\nu\beta_\nu |x|^{-1}}dx\right)^{\frac{1}{s_1}}\\
 &\lesssim &\nu^{r_0}\lambda_\nu^{\frac{N(q-2)}{4}r_0}\sigma_\nu^{\frac{\gamma}{2}r_0+\frac{(N-2)(q-2)}{2}r_0}\beta_\nu^{2r_0}\\
 &\mbox{}&\qquad \cdot \lambda_\nu^{-\frac{1}{2}[\gamma+(N-2)(q-2)\theta ]r_0+\frac{N}{2s_1}}\sigma_\nu^{-[\gamma+(N-2)(q-2)\theta ]r_0+\frac{N}{s_1}}\beta_\nu^{-[\gamma+(N-2)(q-2)\theta ]r_0+\frac{N}{s_1}}\\
&\mbox{}& \qquad \cdot \left(\int_{L_0}^{+\infty}t^{\gamma r_0s_1+(N-2)(q-2)\theta r_0s_1-N-1}e^{-\frac{1}{2}(q-2)\theta r_0s_1t}dt\right)^{\frac{1}{s_1}}\\
&\lesssim&\nu^{r_0}\lambda_\nu^{\Gamma_1(r_0,s_1,\theta)}\sigma_\nu^{\Gamma_2(r_0,s_1,\theta)}\beta_\nu^{\Gamma_3(r_0,s_1,\theta)}.
\end{array}
$$
where 
$$
\Gamma_1(r_0,s_1,\theta)=\frac{N(q-2)}{4}r_0-\frac{1}{2}[\gamma+(N-2)(q-2)\theta ]r_0+\frac{N}{2s_1},
$$
$$
\Gamma_2(r_0,s_1,\theta)=\frac{\gamma}{2}r_0+\frac{(N-2)(q-2)}{2}r_0-[\gamma+(N-2)(q-2)\theta ]r_0+\frac{N}{s_1},
$$
$$
\Gamma_3(r_0,s_1,\theta)=2r_0-[\gamma+(N-2)(q-2)\theta ]r_0+\frac{N}{s_1}.
$$
Clearly, we have 
$$
\Gamma_1(\frac{N}{2},1,0)=\frac{N}{8}(N(q-2)+2(2-\gamma)),
$$
$$
\Gamma_2(\frac{N}{2},1,0)=
\Gamma_3(\frac{N}{2},1,0)=\frac{N}{2}(4-\gamma)>0.
$$
If $\gamma\le 2$, then we can choose $r_0>\frac{N}{2}, s_1>1$ and $\theta\in (0,1)$ such that $\Gamma_1(r_0,s_1,\theta)>0$, $\Gamma_2(r_0,s_1,\theta)>0$ and $\Gamma_3(r_0,s_1,\theta)>0$, and hence 
\begin{equation}\label{e515}
\lim_{\nu\to 0}I_\nu^{(1)}(r_0)=\lim_{\nu\to 0}\nu^{r_0}\lambda_\nu^{\Gamma_1(r_0,s_1,\theta)}\sigma_\nu^{\Gamma_2(r_0,s_1,\theta)}\beta_\nu^{\Gamma_3(r_0,s_1,\theta)}=0.
%\eqno(5.15)
\end{equation}
To prove $\lim_{\nu\to 0}I_\nu^{(1)}(r_0)=0$ for some $r_0>\frac{N}{2}$, we assume $\gamma>2$ and  split into two cases:

\noindent{\bf Case 1: $\gamma<\frac{8}{3}$}. 
In this case,  since $\lambda_\nu\gtrsim \nu\sigma_\nu^{N-\frac{N-2}{2}q}=\nu\sigma_\nu^{\frac{\gamma}{2}}$,  we have 
$$
I^{(1)}_\nu(r_0)\lesssim \lambda_\nu^{r_0+\Gamma_1(r_0,s_1,\theta)}\sigma_\nu^{-\frac{\gamma r_0}{2}+\Gamma_2(r_0,s_1,\theta)}\beta_\nu^{\Gamma_3(r_0,s_1,\theta)}.
$$
Since 
$$
\left[r_0+\Gamma_1(r_0,s_1,\theta)\right ]_{(r_0,s_1,
\theta)=(\frac{N}{2},1,0)}=\frac{N}{8}(8-3\gamma+2q)>0,
$$
$$
\left[-\frac{\gamma r_0}{2}+\Gamma_2(r_0,s_1,\theta)\right]_{(r_0,s_1,
\theta)=(\frac{N}{2},1,0)}=\frac{N}{4}(8-3\gamma)>0,
$$
therefore, we can choose $r_0>\frac{N}{2}$, $s_1>1$ and $\theta\in (0,1)$ such that $r_0+\Gamma_1(r_0,s_1,\theta)>0$ and 
$-\frac{\gamma r_0}{2}+\Gamma_2(r_0,s_1,\theta)>0,$
then 
\begin{equation}\label{e516}
\lim_{\nu\to 0}I_\nu^{(1)}(r_0)=\lim_{\nu\to 0}\lambda_\nu^{r_0+\Gamma_1(r_0,s_1,\theta)}\sigma_\nu^{-\frac{\gamma r_0}{2}+\Gamma_2(r_0,s_1,\theta)}\beta_\nu^{\Gamma_3(r_0,s_1,\theta)}=0.
%\eqno(5.16)
\end{equation}
{\bf Case 2: \ $\gamma\ge \frac{8}{3}$}. In this case,  since $\lambda_\nu\gtrsim \nu\sigma_\nu^{N-\frac{N-2}{2}q}=\nu\sigma_\nu^{\frac{\gamma}{2}}$,  we have 
$$
I^{(1)}_\nu(r_0)\lesssim
\nu^{r_0-\frac{2}{\gamma}\Gamma_2(r_0,s_1,\theta)}\lambda_\nu^{\Gamma_1(r_0,s_1,\theta)+\frac{2}{\gamma}\Gamma_2(r_0,s_1,\theta)}\beta_\nu^{\Gamma_3(r_0,s_1,\theta)}.
$$
Since 
$$
\frac{d}{dr_0}\left\{r_0-\frac{2}{\gamma}\Gamma_2(r_0,1,0)\right\}_{r_0=\frac{N}{2}}=\frac{1}{\gamma}(3\gamma-4)>0,
$$
and for any $r_0>\frac{N}{2}, s_1>1$, $\theta\in (0,1)$, there holds
$$
r_0-\frac{2}{\gamma}\Gamma_2(r_0,s_1,\theta)>r_0-\frac{2}{\gamma}\Gamma_2(r_0,1,0)>\frac{N}{2}-\frac{2}{\gamma}\Gamma_2(\frac{N}{2},1,0)=\frac{N}{2\gamma}(3\gamma-8)\ge 0,
$$
it follows from $\nu\lesssim \lambda_\nu^{\frac{1}{4}[4-N(q-2)]}$ that
$$
I^{(1)}_\nu(r_0)\lesssim  \lambda_\nu^{\Gamma_4(r_0,s_1,\theta)}\beta_\nu^{\Gamma_3(r_0,s_1,\theta)},
$$
where
$$
\Gamma_4(r_0,s_1,\theta)=\frac{4-N(q-2)}{4}\cdot (r_0-\frac{2}{\gamma}\Gamma_2(r_0,s_1,\theta))+\Gamma_1(r_0,s_1,\theta)+\frac{2}{\gamma}\Gamma_2(r_0,s_1,\theta).
$$
Since 
$$
\begin{array}{rcl}
\Gamma_4(\frac{N}{2},1,0)&=&\frac{4-N(q-2)}{4}\cdot \frac{N}{2\gamma}(3\gamma-8)+\frac{N^2(q-2)}{8}+\frac{N}{4}(2-\gamma)+\frac{N}{\gamma}(4-\gamma)\\
&=&\frac{N}{4\gamma}(4-\gamma)(\gamma+N(q-2))>0,
\end{array}
$$
 we can choose $r_0>\frac{N}{2}$, $s_1>1$ and $\theta\in (0,1)$ such that $\Gamma_4(r_0,s_1,\theta)>0$ and hence
\begin{equation}\label{e517}
\lim_{\nu\to 0}I_\nu^{(1)}(r_0)=\lim_{\nu\to 0} \lambda_\nu^{\Gamma_4(r_0,s_1,\theta)}\beta_\nu^{\Gamma_3(r_0,s_1,\theta)}=0.
%\eqno(5.17)
\end{equation}

To show that $\lim_{\nu\to 0}I_\nu^{(2)}(r_0)=0$ for some $r_0>\frac{N}{2}$, we consider the following two cases:

\noindent{\bf Case 1: \ $N\ge 4, q\in (2,2^*)$ or $N=3, q\in [3,4)$}. In this case, by \eqref{e512} and the H\"older inequality, we also have
$$
\begin{array}{rcl}
I_\nu^{(2)}(r_0):&=&\nu^{r_0}\sigma_\nu^{\gamma r_0/2}\beta_\nu^{2r_0}\int_{\eta_\nu\le |x|\le 4}\left|\frac{1}{|x|^\gamma}K[\tilde w_\nu]^{q-2}(x)\right|^{r_0}dx\\
&\lesssim & \nu^{r_0}\sigma_\nu^{\gamma r_0/2}\beta_\nu^{2r_0}\left(\int_{|x|\le 4}K[\tilde w_\nu]^{\frac{(q-2)r_0s_0}{s_0-1}}\right)^{1-\frac{1}{s_0}}\left(\int_{\eta_\nu\le |x|\le 4}|x|^{-\gamma r_0s_0}dx\right)^{\frac{1}{s_0}}\\
&\lesssim &\nu^{r_0}\sigma_\nu^{\gamma r_0/2}\beta_\nu^{2r_0}\left(\int_{\eta_\nu}^4r^{-\gamma r_0s_0+N-1}dr\right)^{\frac{1}{s_0}}\\
&\lesssim &\nu^{r_0}\sigma_\nu^{\gamma r_0/2}\beta_\nu^{2r_0}\eta_\nu^{\frac{1}{s_0}[N-\gamma r_0s_0]}\\
&\lesssim& \nu^{r_0}\lambda_\nu^{\frac{1}{2s_0}[N-\gamma r_0s_0]}\sigma_\nu^{\frac{1}{2s_0}[2N-\gamma r_0s_0]}\beta_\nu^{\frac{1}{s_0}[N+2r_0s_0]}.
\end{array}
$$
If $\gamma\le 2$, then
$$
\left[\frac{1}{2s_0}(N-\gamma r_0s_0)\right]_{r_0=\frac{N}{2}, s_0=1}=\frac{N}{4}(2-\gamma)\ge 0, 
$$ 
 therefore, it follows from the fact  $\nu\lesssim \lambda_\nu^{\frac{1}{4}(q-2)}$ that
\begin{equation}\label{e518}
I_\nu^{(2)}(r_0)\lesssim \nu^{r_0}\lambda_\nu^{\frac{1}{2s_0}[N-\gamma r_0s_0]}\lesssim \lambda_\nu^{\frac{r_0}{4}(q-2)+\frac{1}{2s_0}[N-\gamma r_0s_0]}\to 0, \quad as \ \nu\to 0.
%\eqno(5.18)
\end{equation}
Thus, we assume that $\gamma>2$. Since 
$$
\left[\frac{1}{2s_0}(2N-\gamma r_0s_0)\right]_{r_0=\frac{N}{2},s_0=1}=\frac{N}{4}(4-\gamma)>0
$$
and 
$$
\left[r_0-\frac{1}{2 s_0}(2N-\gamma r_0s_0)\right]_{r_0=\frac{N}{2}, s_0=1}=\frac{N}{4}(\gamma-2)>0,
$$
it follows from the facts $\nu\sigma_\nu^{\frac{\gamma}{2}}\lesssim \lambda_\nu$ and $\nu\lesssim \lambda_\nu^{\frac{4-N(q-2)}{4}}$ that
$$
\begin{array}{rcl}
I_\nu^{(2)}(r_0)&\lesssim &\nu^{r_0}(\nu^{-1}\lambda_\nu)^{\frac{1}{2 s_0}(2N-\gamma r_0s_0)}\lambda_\nu^{\frac{1}{2s_0}[N-\gamma r_0s_0]}\\
&=&\nu^{r_0-\frac{1}{2 s_0}(2N-\gamma r_0s_0)}\lambda_\nu^{\frac{1}{2 s_0}(2N-\gamma r_0s_0)}\lambda_\nu^{\frac{1}{2s_0}[N-\gamma r_0s_0]}\\
&\lesssim &\lambda_\nu^{\Gamma_5(r_0,s_0)},
\end{array}
$$
where 
$$
\Gamma_5(r_0,s_0):=\frac{4-N(q-2)}{4}[r_0-\frac{1}{2 s_0}(2N-\gamma r_0s_0)]+\frac{1}{2 s_0}(2N-\gamma r_0s_0)+\frac{1}{2s_0}[N-\gamma r_0s_0].
$$
It is easy to check that
$$
\Gamma_5(\frac{N}{2}, 1)=\frac{N^2(N-2)}{16}(q-2)\left(q-\frac{2(N^2-3N+4)}{N(N-2)}\right)>0, \quad if \ \left\{\begin{array}{rcl}
 N=3, \ q>\frac{8}{3}, \\
N\ge 4,\  q>2.
\end{array}\right.
$$
Therefore, by choosing $r_0>\frac{N}{2}$ and $s_0>1$ such that  $\Gamma_2(r_0,s_0)>0$ and hence
\begin{equation}\label{e519}
\lim_{\nu\to 0}I_\nu^{(2)}(r_0)=\lim_{\nu\to 0}\lambda_\nu^{\Gamma_5(r_0,s_0)}\beta_\nu^{\frac{1}{s_0}[N+2r_0s_0]}=0.
%\eqno(5.19)
\end{equation}
{\bf Case 2: \ $N=3$ and $q\in (2,3)$}. In this case, by $\beta_\nu\gtrsim \nu^{\tau_1}\sigma_\nu^{\tau_2}\lambda_\nu^{\tau_3}$,  for $r_0>\frac{N}{2}$, we get
$$
\begin{array}{rcl}
I_\nu^{(2)}(r_0)&=&\nu^{r_0}\sigma_\nu^{\gamma r_0/2}\beta_\nu^{2r_0}\int_{\eta_\nu\le |x|\le 4}|x|^{-\gamma r_0}K[\tilde w_\nu]^{r_0(q-2)}(x)dx\\
&=&\nu^{r_0}\sigma_\nu^{\gamma r_0/2}\beta_\nu^{2r_0}\cdot \int_{\frac{1}{4}\le |z|\le \eta_\nu^{-1}}|z|^{4r_0-2N}\tilde w_\nu^{r_0(q-2)}(z)dz\\
&=&\nu^{r_0}\sigma_\nu^{\gamma r_0/2}\beta_\nu^{2r_0}\cdot \sigma_\nu^{\frac{1}{2}r_0(N-2)(q-2)}\int_{\frac{1}{4}\le |z|\le \eta_\nu^{-1}}|z|^{4r_0-2N}u_\nu^{r_0(q-2)}(\sigma_\nu \beta_\nu z)dz\\
&\lesssim &\nu^{r_0}\sigma_\nu^{\gamma r_0/2}\beta_\nu^{2r_0}\cdot \sigma_\nu^{\frac{1}{2}r_0(N-2)(q-2)}\int_{\frac{1}{4} \le |z|\le\eta_\nu^{-1}}|z|^{4r_0-2N}|\sigma_\nu \beta_\nu z|^{-\frac{N}{2}r_0(q-2)}dz\\
&=&\nu^{r_0}\sigma_\nu^{\gamma r_0/2}\beta_\nu^{2r_0}\cdot \sigma_\nu^{-r_0(q-2)}\beta_\nu^{-\frac{N}{2}r_0(q-2)}\int_{\frac{1}{4}}^{\eta_\nu^{-1}}r^{4r_0-\frac{N}{2}r_0(q-2)-N-1}dr\\
&\lesssim&\nu^{r_0}\sigma_\nu^{\gamma r_0/2}\beta_\nu^{2r_0}\cdot \sigma_\nu^{-r_0(q-2)}\beta_\nu^{-\frac{N}{2}r_0(q-2)}\cdot \eta_\nu^{-4r_0+\frac{N}{2}r_0(q-2)+N}\\
%&=&\nu^{r_0}\sigma_\nu^{\gamma r_0/2-r_0(q-2)-4r_0+\frac{N}{2}r_0(q-2)+N}\lambda_\nu^{-2r_0+\frac{N}{4}r_0(q-2)+\frac{N}{2}}\\
&\lesssim &\nu^{r_0}\sigma_\nu^{\Gamma_6(r_0)}\lambda_\nu^{-2r_0+\frac{N}{4}r_0(q-2)+\frac{N}{2}}\beta_\nu^{N-2r_0}\\
&\lesssim &\nu^{r_0+\tau_1(N-2r_0)}\sigma_\nu^{\Gamma_6(r_0)+\tau_2(N-2r_0)}\lambda_\nu^{-2r_0+\frac{N}{4}r_0(q-2)+\frac{N}{2}+\tau_3(N-2r_0)},
\end{array}
$$
where 
$$
\Gamma_6(r_0)=\gamma r_0/2-r_0(q-2)-4r_0+\frac{N}{2}r_0(q-2)+N.
$$
Clearly, we have  $\Gamma_6(\frac{N}{2})=\frac{N}{2}(N-q)>0$ for $q\in (2,3)$, therefore, for $r_0>\frac{N}{2}$ with $r_0-\frac{N}{2}>0$ sufficiently small, 
by $\lambda_\nu\gtrsim \nu\sigma_\nu^{\frac{\gamma}{2}}$, we have 
$$
\nu^{r_0+\tau_1(N-2r_0)}\sigma_\nu^{\Gamma_6(r_0)+\tau_2(N-2r_0)}\lesssim \nu^{r_0+\tau_1(N-2r_0)-\frac{2}{\gamma}(\Gamma_6(r_0)+\tau_2(N-2r_0))}\lambda_\nu^{\frac{2}{\gamma}(\Gamma_6(r_0)+\tau_2(N-2r_0))}.
$$
Since 
$$
\left[r_0-\frac{2}{\gamma}\Gamma_6(r_0)\right]_{r_0=N/2}=\frac{N}{2}-\frac{2}{\gamma}\Gamma_6(\frac{N}{2})
=\frac{N}{2}-\frac{2}{\gamma}\cdot \frac{N}{2}(N-q)=\frac{N}{2\gamma}q>0,
$$
it follows from the fact $\nu\lesssim \lambda_\nu^{\frac{1}{4}[4-N(q-2)]}=\lambda_\nu^{\frac{10-3q}{4}}$ that
$$
I_\nu^{(2)}(r_0)\lesssim \lambda_\nu^{\Gamma_7(r_0)+\tau_3(N-2r_0)},
$$
where
$$
\begin{array}{rcl}
\Gamma_7(r_0):&=&\frac{10-3q}{4}[r_0+\tau_1(N-2r_0)-\frac{2}{\gamma}(\Gamma_6(r_0)+\tau_2(N-2r_0))]\\
&\mbox{}& \quad +\frac{2}{\gamma}(\Gamma_6(r_0)+\tau_2(N-2r_0))-2r_0+\frac{N}{4}r_0(q-2)+\frac{N}{2}.
\end{array}
$$
Since 
$$
\Gamma_7(\frac{N}{2})=\frac{10-3q}{4}\cdot \frac{N}{2\gamma}q+\frac{N}{\gamma}(N-q)-\frac{N}{2}+\frac{N^2}{8}(q-2)
=-\frac{27}{8\gamma}(q-2)(q-3)>0,
$$
for  all $q\in (2,3)$, it follows that there is some $r_0>\frac{N}{2}$ with $r_0-\frac{N}{2}$ small enough such that
\begin{equation}\label{e520}
\lim_{\nu\to 0}I_\nu^{(2)}(r_0)=\lim_{\nu\to 0}\lambda_\nu^{\Gamma_7(r_0)+\tau_3(N-2r_0)}=0.
%\eqno(5.20)
\end{equation}
Finally, \eqref{e514} follows from \eqref{e515}-\eqref{e520},  the proof of Proposition \ref{p51} is completed. \end{proof}
 \vskip 3mm

\subsection{Sobolev critical case.}

Consider the following equations 
 \begin{equation}\label{e521}
 -\Delta \tilde w+\lambda_\nu\sigma_\nu^2 \beta_\nu ^2\tilde w=\nu \sigma_\nu^{N+\alpha-p(N-2)}\gamma_\nu (I_\alpha\ast |w|^p)|\tilde w|^{p-2}\tilde w+
 \beta_\nu^2 |\tilde w|^{2^*-2}\tilde w,\quad x\in \mathbb R^N,
% \eqno(5.21)
 \end{equation}
where $\nu>0$ is a small parameter, $\lim_{\nu\to 0}\lambda_\nu=0$ and $\lambda_\nu, \sigma_\nu, \beta_\nu, \gamma_\nu>0$ are bounded.  
  
Let $W_1$ be  the Talenti function given above. 
Adopting the Moser iteration, we obtain the following uniform decay estimates of the weak solutions.

\begin{proposition}\label{p54}  Assume that  $N\ge 3$, $\tilde w_\nu\in H^1(\mathbb R^N)$ is a weak positive solution of  \eqref{e521}, which is radially symmetric and  non-increasing, 
 and 
 \begin{equation}\label{e522}
 \max\left\{2, 1+\frac{\alpha}{N-2}\right\}<p<\frac{N+\alpha}{N-2}.
 %\eqno(5.22)
 \end{equation}
If $\gamma_\nu\lesssim \beta_\nu^{2+\alpha}$, $v_\nu:=\tilde w(\beta_\nu^{-1}x)\to W_1$ in $D^{1,2}(\mathbb R^N)$ as $\nu\to 0$, and $u_\nu:=\sigma_\nu^{-(N-2)/2}v_\nu(\sigma_\nu^{-1}x)$ is bounded in  $H^1(\mathbb R^N)$,  then  there exists a constant $C>0$ such that 
for small $\nu>0$, there holds
$$
\tilde w_\nu(x)\le  C(1+|x|)^{-(N-2)}, \qquad x\in \mathbb R^N.
$$
\end{proposition}

To prove Proposition \ref{p54}, we need to show that there exists $\nu_0>0$ such that the Kelvin transform of $\tilde w_\nu$ satisfies
\begin{equation}\label{e523}
\sup_{\nu\in (0,\nu_0)}\|K[\tilde w_\nu]\|_{L^{\infty}(B_1)}<\infty.
%\eqno(5.23)
\end{equation}
It is easy to verify that $K[\tilde w_\nu]$ satisfies 
\begin{equation}\label{e524}
-\Delta K[\tilde w_\nu]+\frac{\lambda_\nu\sigma_\nu^{2}\beta_\nu^2}{|x|^4}K[\tilde w_\nu]=\frac{\nu\sigma_\nu^{\eta}\gamma_\nu}{|x|^{4}}(I_\alpha\ast |\tilde w_\nu|^{p})(\frac{x}{|x|^2})\tilde w_\nu^{p-2}(\frac{x}{|x|^2})K[\tilde w_\nu]+\beta_\nu^2K[\tilde w_\nu]^{2^*-1},
%\eqno(5.24)
\end{equation}
here and in what follows, we set
\begin{equation}\label{e525}
\eta:=N+\alpha-p(N-2)>0.
%\eqno(5.25)
\end{equation}
   Let 
$$
a(x)=\frac{\lambda_\nu\sigma_\nu^2\beta_\nu^2}{|x|^4}, 
$$
and 
$$
\begin{array}{lcl}
 b(x)&=&\frac{\nu\sigma_\nu^\eta\gamma_\nu}{|x|^{4}}(I_\alpha\ast |\tilde w_\nu|^p)(\frac{x}{|x|^2})\tilde w_\nu^{p-2}(\frac{x}{|x|^2})
+\beta_\nu^2K[\tilde w_\nu]^{2^*-2}(x)\\
&=&\frac{\nu\sigma_\nu^\eta\gamma_\nu}{|x|^{\gamma}}(I_\alpha\ast |\tilde w_\nu|^p)(\frac{x}{|x|^2})K[\tilde w_\nu]^{p-2}(x)
+\beta_\nu^2 K[\tilde w_\nu]^{2^*-2}(x),
\end{array}
$$
where and in what follows, we set
\begin{equation}\label{e526}
\gamma:=2N-p(N-2).
%\eqno(5.26)
\end{equation}
Then \eqref{e524} reads as 
$$
-\Delta K[\tilde w_\nu]+a(x)K[\tilde w_\nu]=b(x)K[\tilde w_\nu].
$$
We shall apply the Moser iteration to prove  \eqref{e523}.  Arguing as in the proof of Lemma \ref{l52}, we prove the following

\begin{lemma}\label{l55}  Under the assumptions of Proposition \ref{p54},  there exist constants $L_0>0$ and $C_0>0$ such that for any small $\mu>0$ and $|x|\ge L_0\lambda_\nu^{-1/2}\sigma_\nu^{-1}\beta_\nu^{-1}$, 
$$
\tilde w_\nu(x)\le C_0\lambda_\nu^{N/4}\sigma_\nu^{(N-2)/2}\exp(-\frac{1}{2}\lambda_\nu^{1/2}\sigma_\nu\beta_\nu |x|).
$$
\end{lemma}

It follows from Lemma \ref{l55} that if $|x|\le \lambda_\nu^{1/2}\sigma_\nu\beta_\nu/L_0$, then 
\begin{equation}\label{e527}
K[\tilde w_\nu](x)\lesssim \frac{1}{|x|^{N-2}}\lambda_\nu^{\frac{N}{4}}\sigma_\nu^{\frac{N-2}{2}}e^{-\frac{1}{2}\lambda_\nu^{1/2}\sigma_\nu\beta_\nu |x|^{-1}}.
%\eqno(5.27)
\end{equation}
Therefore, for any $v\in H^1_0(B_4)$, we have
\begin{equation}\label{e528}
\int_{B_4}\frac{\lambda_\nu\sigma_\nu^2\beta_\nu^2}{|x|^4}K[\tilde w_\nu](x)|v(x)|dx<\infty.
%\eqno(5.28)
\end{equation}
It is well known that  the Kelvin transform is linear and preserves the $L^{2^*}(\mathbb R^N)$ norm.  Since  $ v_\nu\to W_1$  in $L^{2^*}(\mathbb R^N)$   as $\nu\to 0$, we have 
$$
\lim_{\nu\to 0}\|K[v_\nu]-K[W_1]\|_{L^{2^*}(\mathbb R^N)}=\lim_{\nu\to 0}\|v_\nu-W_1\|_{L^{2^*}(\mathbb R^N)}=0,
$$
which implies that
\begin{equation}\label{e529}
\limsup_{\nu\to 0}\int_{|x|\le 4}|\beta_\nu^2K[\tilde w_\nu]^{\frac{4}{N-2}}|^{N/2}dx\le \limsup_{\nu\to 0}\int_{\mathbb R^N}K[v_\nu]^{\frac{2N}{N-2}}dx<\infty.
%\eqno(5.29)
\end{equation}

\begin{lemma}\label{l56}  Assume $N\ge 3$ and $\max\{2, 1+\frac{\alpha}{N-2}\}<p<\frac{N+\alpha}{N-2}$. Then 
\begin{equation}\label{e530}
\lim_{\nu\to 0}\int_{|x|\le 4}\left|\frac{\nu\sigma_\nu^{\eta}\gamma_\nu}{|x|^{4}}(I_\alpha\ast |\tilde w_\nu|^{p})(\frac{x}{|x|^2})\tilde w_\nu^{p-2}(\frac{x}{|x|^2})\right|^{N/2}dx=0.
%\eqno(5.30)
\end{equation}
\end{lemma}
\begin{proof} 
To prove this lemma, we split into two cases:

\noindent{\bf Case 1: $p>\frac{2\alpha}{N-2}$}. In this case, 
it follows from the H\"older inequality, the Hardy-Littlewood-Sobolev inequality  and the boundedness of $u_\nu$ in $H^1(\mathbb R^N)$ that 
\begin{equation}\label{e531}
\begin{array}{rcl}
I_\nu(\frac{N}{2})&:=&\int_{|x|\le 4}\left|\frac{\nu\sigma_\nu^{\eta}\gamma_\nu}{|x|^{4}}(I_\alpha\ast |\tilde w_\nu|^{p})(\frac{x}{|x|^2})\tilde w_\nu^{p-2}(\frac{x}{|x|^2})\right|^{N/2}dx\\
&=&\int_{\frac{1}{4}\le |z|<+\infty} |\nu\sigma_\nu^{\eta}\gamma_\nu(I_\alpha\ast |\tilde w_\nu|^{p})(z)\tilde w_\nu^{p-2}(z)|^{N/2}dz\\
&\le &\nu^{\frac{N}{2}}\sigma_\nu^{\frac{\eta N}{2}}\gamma_\nu^{\frac{N}{2}}\left(\int_{\frac{1}{4}\le |z|<+\infty}|\tilde w_\nu|^{2^*}dz\right)^{\frac{(N-2)(p-2)}{4}}\\
&\quad & \cdot \left(\int_{\frac{1}{4}\le |z|<+\infty} |(I_\alpha\ast |\tilde w_\nu|^p)(z)|^{\frac{2N}{2N-(N-2)p}}dz\right)^{\frac{2N-(N-2)p}{4}}\\
&\lesssim & \nu^{\frac{N}{2}}\sigma_\nu^{\frac{\eta N}{2}}\gamma_\nu^{\frac{N}{2}}\beta_\nu^{-N\cdot\frac{(N-2)(p-2)}{4}}\left(\int_{\frac{1}{4}\le |z|<+\infty}|v_\nu|^{2^*}dz\right)^{\frac{(N-2)(p-2)}{4}}\\
&\quad &\cdot \left(\int_{\mathbb R^N}|\tilde w_\nu|^{\frac{2Np}{2N-(N-2)p+2\alpha}}\right)^{\frac{2N-(N-2)p+2\alpha}{4}}\\
&\lesssim & \nu^{\frac{N}{2}}\beta_\nu^{\frac{N}{2}(2+\alpha)-\frac{N(N-2)(p-2)}{4}}\sigma_\nu^{\frac{\eta N}{2}+\frac{N(N-2)p}{4}-\frac{N[2N-(N-2)p+2\alpha]}{4}}\\
&\mbox{}&\quad \cdot \beta_\nu^{-\frac{N[2N-(N-2)p+2\alpha]}{4}}\left(\int_{\mathbb R^N}|u_\nu|^{\frac{2Np}{2N-(N-2)p+2\alpha}}\right)^{\frac{2N-(N-2)p+2\alpha}{4}}\\
&\lesssim &\nu^{ N/2}
%\beta_\nu^{\frac{N}{4}(N-2)(q-2)}
\to 0,
\end{array}
%\eqno(5.31)
\end{equation}
as $\nu\to 0$, here we have used the facts that 
$$
2\le \frac{2Np}{2N-(N-2)p+2\alpha}\le 2^*,
$$
which follows from $\frac{N+\alpha}{N-2}>p\ge\max\{2, \frac{2\alpha}{N-2}\}\ge \frac{N+\alpha}{N-1}$ and 
$$
\frac{\eta N}{2}+\frac{N(N-2)p}{4}-\frac{N[2N-(N-2)p+2\alpha]}{4}=0.
$$
{\bf Case 2: $\max\{2,1+\frac{\alpha}{N-2}\}<p\le \frac{2\alpha}{N-2}$}.   In this case, $2+\alpha>N$.  For any $\tau>\frac{2+\alpha-N}{N-\alpha}$, let
$$
s_1=\frac{2\tau}{2\tau-2-\alpha+N},\quad s_2=\frac{2\tau}{2+\alpha-N}.
$$
Then 
$$
\frac{1}{s_1}+\frac{1}{s_2}=1.
$$
Put
$$
\quad s=\frac{N\tau}{2+\alpha+\tau \alpha-N}.
$$
Then $1<s<\frac{N}{\alpha}$. 
Since $\tau>\frac{2+\alpha-N}{N-\alpha}>\frac{4+2\alpha-2N}{Np-2\alpha}$, it follows that
$$
ps=\frac{pN\tau}{2+\alpha+\tau\alpha -N}\ge 2.
$$
Since $p\le \frac{2\alpha}{N-2}$, it follows that
$$
ps=\frac{pN\tau}{2+\alpha+\tau\alpha-N}\le \frac{2N}{N-2}.
$$
It follows from the H\"older inequality, the Hardy-Littlewood-Sobolev inequality  and the boundedness of $u_\nu$ in $H^1(\mathbb R^N)$ that 
\begin{equation}\label{e532}
\begin{array}{rcl}
I_\nu(\frac{N}{2})&:=&\int_{|x|\le 4}\left|\frac{\nu\sigma_\nu^{\eta}\gamma_\nu}{|x|^{4}}(I_\alpha\ast |\tilde w_\nu|^{p})(\frac{x}{|x|^2})\tilde w_\nu^{p-2}(\frac{x}{|x|^2})\right|^{N/2}dx\\
&=&\int_{\frac{1}{4}\le |z|<+\infty} |\nu\sigma_\nu^{\eta}\gamma_\nu(I_\alpha\ast |\tilde w_\nu|^{p})(z)\tilde w_\nu^{p-2}(z)|^{N/2}dz\\
&\le &\nu^{\frac{N}{2}}\sigma_\nu^{\frac{\eta N}{2}}\gamma_\nu^{\frac{N}{2}}\left(\int_{\frac{1}{4}\le |z|<+\infty}|\tilde w_\nu|^{\frac{N}{2}(p-2)s_1}dz\right)^{\frac{1}{s_1}}\\
&\quad & \cdot \left(\int_{\frac{1}{4}\le |z|<+\infty} |(I_\alpha\ast |\tilde w_\nu|^p)(z)|^{\frac{N}{2}s_2}dz\right)^{\frac{1}{s_2}}\\
&\le &\nu^{\frac{N}{2}}\sigma_\nu^{\frac{\eta N}{2}}\gamma_\nu^{\frac{N}{2}}\left(\int_{\frac{1}{4}\le |z|<+\infty}|\tilde w_\nu|^{\frac{N}{2}(p-2)s_1}dz\right)^{\frac{1}{s_1}}\cdot \left(\int_{\mathbb R^N}  |\tilde w_\nu|^{ps}dz\right)^{\frac{N}{2s}}\\
&\lesssim & \nu^{\frac{N}{2}}\sigma_\nu^{T_1(\tau)}\beta_\nu^{T_2(\tau)}\left(\int_{\mathbb R^N}|u_\nu|^{\frac{N}{2}(p-2)s_1}dz\right)^{\frac{1}{s_1}} \left(\int_{\mathbb R^N}|u_\nu|^{ps}\right)^{\frac{N}{2s}}\\
&\lesssim & \nu^{\frac{N}{2}}\sigma_\nu^{T_1(\tau)}\beta_\nu^{T_2(\tau)}=\nu^{ N/2}
%\beta_\nu^{\frac{N}{4}(N-2)(q-2)}
\to 0,
\end{array}
%\eqno(5.32)
\end{equation}
as $\nu\to 0$, where we have used the following facts
$$
T_1(\tau):= \frac{\eta N}{2}+\frac{1}{4}N(N-2)(p-2)-\frac{N}{s_1}+\frac{1}{4}N(N-2)p-\frac{N^2}{2s}
=\frac{N}{2}[2+\alpha-\frac{2}{s_1}-\frac{N}{s}]=0,
$$
and 
$$
T_2(\tau):=(2+\alpha)\frac{N}{2}-\frac{N}{s_1}-\frac{N^2}{2s}=\frac{N}{2}[2+\alpha-\frac{2}{s_1}-\frac{N}{s}]=0.
$$
The proof is complete. \end{proof}

Now, we are in the position to prove Proposition \ref{p54}.

\begin{proof}[Proof of Proposition \ref{p54}]   It follows from \eqref{e528}, \eqref{e529} and Lemma \ref{l28} (i) that for any $r>1$, there exists $\nu_r>0$ such that 
\begin{equation}\label{e533}
\sup_{\nu\in (0,\nu_r)}\|K[\tilde w_\nu]^r\|_{H^1(B_1)}\le C_r.
%\eqno(5.33)
\end{equation}
To verify the condition in Lemma \ref{l28} (ii), we show that there exists $r_0>\frac{N}{2}$ such that 
\begin{equation}\label{e534}
\lim_{\nu\to 0}\int_{|x|\le 4}\left|\frac{\nu\sigma_\nu^{\eta}\gamma_\nu}{|x|^{4}}(I_\alpha\ast |\tilde w_\nu|^{p})(\frac{x}{|x|^2})\tilde w_\nu^{p-2}(\frac{x}{|x|^2})\right|^{r_0}dx=0.
%\eqno(5.34)
\end{equation}

Since $2N-p(N-2)<4$, we can choose $r_0>\frac{N}{2}$ such that $(2N-p(N-2))r_0<2N$. Put
$$
s_1=\frac{2N}{2N-(2N-p(N-2))r_0},\quad s_2=\frac{2N}{(2N-p(N-2))r_0}.
$$
To prove \eqref{e534}, we split into two cases:

\noindent{\bf Case 1: $p>\max\{2,\frac{2\alpha}{N-2}\}$}. In this case, since $\gamma r_0s_2=2N$,  it follows from the H\"older inequality,  the Hardy-Littlewood-Sobolev inequality  and the boundedness of $u_\nu$ in $H^1(\mathbb R^N)$ that 
\begin{equation}\label{e535}
\begin{array}{rcl}
I_\nu(r_0)&:=&\int_{|x|\le 4}\left|\frac{\nu\sigma_\nu^{\eta}\gamma_\nu}{|x|^{\gamma}}(I_\alpha\ast |\tilde w_\nu|^{p})(\frac{x}{|x|^2})K[\tilde w_\nu]^{p-2}(x)\right|^{r_0}dx\\
&\le &\nu^{r_0}\sigma_\nu^{\eta r_0}\gamma_\nu^{r_0}\left(\int_{|x|\le 4}|K[\tilde w_\nu]|^{(p-2)r_0s_1}dx\right)^{\frac{1}{s_1}}\\
&\quad & \cdot \left(\int_{ |x|\le 4} |x|^{-\gamma r_0s_2}|(I_\alpha\ast |\tilde w_\nu|^p)(\frac{x}{|x|^2})|^{r_0s_2}dx\right)^{\frac{1}{s_2}}\\
&\le &\nu^{r_0}\sigma_\nu^{\eta r_0}\gamma_\nu^{r_0}\cdot \left(\int_{\frac{1}{4}\le |z|<+\infty} |(I_\alpha\ast |\tilde w_\nu|^p)(z)|^{r_0s_2}dz\right)^{\frac{1}{s_2}}\\
&\lesssim & \nu^{r_0}\sigma_\nu^{\eta r_0}\gamma_\nu^{r_0}\cdot \left(\int_{\mathbb R^N}|\tilde w_\nu|^{\frac{2Np}{2N-(N-2)p+2\alpha}}\right)^{\frac{2N-(N-2)p+2\alpha}{2N}r_0}\\
&\lesssim & \nu^{r_0}\beta_\nu^{r_0(2+\alpha)-\frac{2N-p(N-2)+2\alpha}{2}r_0}\sigma_\nu^{\eta r_0+\frac{N-2}{2}pr_0-\frac{2N-(N-2)p+2\alpha}{2}r_0}\\
&\mbox{}&\quad \cdot \left(\int_{\mathbb R^N}| u_\nu|^{\frac{2Np}{2N-(N-2)p+2\alpha}}\right)^{\frac{2N-(N-2)p+2\alpha}{2N}r_0}\\
&\lesssim &\nu^{r_0}\beta_\nu^{\frac{1}{2}(N-2)(p-2)r_0}
\to 0,
\end{array}
%\eqno(5.35)
\end{equation}
as $\nu\to 0$, here we have used the facts that 
$$
2\le \frac{2Np}{2N-(N-2)p+2\alpha}\le 2^*,
$$
which follows from $\frac{N+\alpha}{N-2}>p\ge\max\{2, \frac{2\alpha}{N-2}\}\ge \frac{N+\alpha}{N-1}$ and 
$$
\eta+\frac{1}{2}(N-2)p-\frac{2N-(N-2)p+2\alpha}{2}=0.
$$
%\newpage 
{\bf Case 2: $\max\{2,1+\frac{\alpha}{N-2}\}<p\le \frac{2\alpha}{N-2}$}.  In this case, $2+\alpha>N$.  For any $\epsilon>0$ sufficiently small, we set 
 $\tau=\frac{2+\alpha-N}{N-\alpha}(1+\epsilon)$, and 
$$
s_1=\frac{2(1+\epsilon)}{2+\alpha-N}, \quad s_2=\frac{2(1+\epsilon)}{N-\alpha}, \quad s_3=\frac{1+\epsilon}{\epsilon}.
$$
Then 
$$
\frac{1}{s_1}+\frac{1}{s_2}+\frac{1}{s_3}=1.
$$
Put
$$
\quad s=\frac{N(1+\epsilon)}{N-\alpha+(1+\epsilon)\alpha}.
$$
Then $1<s<\frac{N}{\alpha}$ and $2\le ps\le 2^*$. Set
$$
\gamma_1=\frac{N(N-\alpha)}{(1+\epsilon)r_0},\quad \gamma_2=\gamma-\gamma_1.
$$
Let 
$
T(r_0,\epsilon):=-\gamma_2r_0s_1+N.
$
Then $T(\frac{N}{2},0)=\frac{N}{2+\alpha-N}[(N-2)(p-1)-\alpha]>0$. We choose $r_0>\frac{N}{2}$ and $\epsilon>0$ such that 
$$
T(r_0,\epsilon)=-\gamma_2r_0s_1+N>0.
$$ 
Note that $\gamma_1r_0s_2=2N$,  it follows from H\"older inequality, the Hardy-Littlewood-Sobolev inequality  and the boundedness of $u_\nu$ in $H^1(\mathbb R^N)$ that 
\begin{equation}\label{e536}
\begin{array}{rcl}
I_\nu(r_0)&:=&\int_{|x|\le 4}\left|\frac{\nu\sigma_\nu^{\eta}\gamma_\nu}{|x|^{\gamma}}(I_\alpha\ast |\tilde w_\nu|^{p})(\frac{x}{|x|^2})K[\tilde w_\nu]^{p-2}(x)\right|^{r_0}dx\\
&\le &\nu^{r_0}\sigma_\nu^{\eta r_0}\gamma_\nu^{r_0}\left(\int_{|x|\le 4}|x|^{-\gamma_2r_0s_1}dx\right)^{\frac{1}{s_1}}\left(\int_{|x|\le 4}|K[\tilde w_\nu]|^{(p-2)r_0s_3}dx\right)^{\frac{1}{s_3}}\\
&\quad & \cdot \left(\int_{ |x|\le 4} |x|^{-\gamma_1 r_0s_2}|(I_\alpha\ast |\tilde w_\nu|^p)(\frac{x}{|x|^2})|^{r_0s_2}dx\right)^{\frac{1}{s_2}}\\
&\le &\nu^{r_0}\sigma_\nu^{\eta r_0}\gamma_\nu^{r_0}\left(\int_0^4 r^{-\gamma_2r_0s_1+N-1}dr\right)^{\frac{1}{s_1}}\\
&\mbox{}&\qquad \cdot \left(\int_{\frac{1}{4}\le |z|<+\infty}|(I_\alpha\ast |\tilde w_\nu|^p)(z)|^{r_0s_2}dz\right)^{\frac{1}{s_2}}\\
&\lesssim & \nu^{r_0}\sigma_\nu^{\eta r_0}\gamma_\nu^{r_0}\cdot \left(\int_{\mathbb R^N}|\tilde w_\nu|^{ps}\right)^{\frac{r_0}{s}}\\
&\lesssim & \nu^{r_0}\beta_\nu^{r_0(2+\alpha)-\frac{Nr_0}{s}}\sigma_\nu^{\eta r_0+\frac{N-2}{2}pr_0-\frac{Nr_0}{s}} \cdot \left(\int_{\mathbb R^N}| u_\nu|^{ps}\right)^{\frac{r_0}{s}}\\
&\lesssim &\nu^{r_0}\beta_\nu^{[2-\frac{N-\alpha}{1+\epsilon}]r_0}\sigma_\nu^{[N-\frac{1}{2}p(N-2)-\frac{N-\alpha}{1+\epsilon}]r_0}\\
&\lesssim &\nu^{r_0}\beta_\nu^{(2+\alpha-N)r_0}\sigma_\nu^{\frac{1}{2}(2\alpha-p(N-2))r_0}
\to 0,
\end{array}
%\eqno(5.36)
\end{equation}
Finally,\eqref{e534} follows from \eqref{e535} and \eqref{e536}. Therefore, by virtue of  Lemma \ref{l28} (ii), \eqref{e528}, \eqref{e529} and \eqref{e534}, we obtain \eqref{e523}, which implies that
\begin{equation}\label{e537}
\sup_{|x|\ge 1} \tilde w_\nu(x)\lesssim |x|^{-(N-2)}.
%\eqno(5.37)
\end{equation}
Let $\hat a(x)=\lambda_\nu\sigma_\nu^2\beta_\nu^2$ and 
$$
\hat b(x)=\nu\sigma_\nu^{(N+\alpha)-p(N-2)}\gamma_\nu(I_\alpha\ast  |\tilde w_\nu|^p)|\tilde w_\nu|^{p-2}+\beta_\nu^2\tilde w_\nu^{2^*-2}.
$$ 
Then 
$$
-\Delta \tilde w_\nu+\hat a(x)\tilde w_\nu=\hat b(x)\tilde w_\nu,  \qquad x\in \mathbb R^N.
$$
Applying the Moser iteration again and arguing as before,  we obtain some $\hat \nu_0>0$ such that 
$$
\sup_{\nu\in (0,\hat \nu_0)}\|\tilde w_\nu\|_{L^{\infty}(B_1)}<\infty,
$$
which together with \eqref{e537} completes the proof.
\end{proof}

\bigskip
{\small
\noindent {\bf Acknowledgements.} 
This research is  supported by National Natural Science Foundation of China
(Grant Nos.11571187, 11771182)

\end {document}